\documentclass[12pt]{amsart}

\usepackage[usenames,dvipsnames]{xcolor}
\usepackage{enumitem,kantlipsum}

\usepackage{mathtools}
\usepackage{xfrac}
\usepackage[utf8]{inputenc}
\usepackage{amssymb}
\usepackage{amsmath}
\usepackage{amsfonts}
\usepackage{amsthm}
\usepackage{tikz}
\usepackage{algorithmic}
\usepackage{asymptote}
\usepackage{url}
\usepackage{hyperref}
\usepackage{bbm}
\usepackage{enumitem,kantlipsum}

\newtheorem{theorem}{Theorem}[section]
\newtheorem{corollary}{Corollary}[theorem]
\newtheorem{lemma}[theorem]{Lemma}
\newtheorem{proposition}[theorem]{Proposition}
\newtheorem{conjecture}[theorem]{Conjecture}
\newtheorem{definition}[theorem]{Definition}

\usepackage[OT2,T1]{fontenc}
\DeclareSymbolFont{cyrletters}{OT2}{wncyr}{m}{n}
\DeclareMathSymbol{\Sha}{\mathalpha}{cyrletters}{"58}

\theoremstyle{definition}

\newtheorem{innercustomprop}{Proposition}
\newenvironment{customprop}[1]
  {\renewcommand\theinnercustomprop{#1}\innercustomprop}
  {\endinnercustomprop}

\newcommand{\ceil}[1]{\lceil#1\rceil}
\newcommand{\floor}[1]{\lfloor#1\rfloor}
\newcommand\restr[2]{{% we make the whole thing an ordinary symbol
  \left.\kern-\nulldelimiterspace % automatically resize the bar with \right
  #1 % the function
  \vphantom{\big|} % pretend it's a little taller at normal size
  \right|_{#2} % this is the delimiter
  }}

\newcommand{\Hom}{\text{Hom}}
\newcommand{\Aut}{\text{Aut}}
\newcommand{\Sym}{\text{Sym}}
\newcommand{\Sur}{\text{Sur}}
\newcommand{\im}{\text{Im}}
\newcommand{\Coker}{\text{Coker}}

\newcommand{\C}{\mathbb{C}}

\newcommand{\Z}{\mathbb{Z}}
\newcommand{\E}{\mathbb{E}}
\newcommand{\EE}{\mathcal{E}}
\newcommand{\D}{\mathcal{D}}

\newcommand{\MM}{\mathbb{M}}
\newcommand{\PPP}{\mathbb{P}}
\newcommand{\CC}{\mathcal{C}}
\newcommand{\G}{\mathcal{G}}

\newcommand{\Zp}{\mathbb{Z}_p}
\newcommand{\Zpn}{\mathbb{Z}_p^n}
\newcommand{\Zpno}{\mathbb{Z}_p^{n_1}}
\newcommand{\Zpnt}{\mathbb{Z}_p^{n_2}}
\newcommand{\Zpnot}{\mathbb{Z}_p^{n_1+n_2}}
\newcommand{\Znot}{Z_{n_1+n_2}}

\newcommand{\sg}[1]{\left\langle #1 \right\rangle}
\newcommand{\sgo}[1]{\left\langle #1 \right\rangle_{1}}
\newcommand{\set}[1]{\text{set}(#1)}
\newcommand{\abs}[1]{\left\vert#1\right\vert}

\newcommand{\Lnot}{L_{n_1,n_2}}
\newcommand{\Lna}{L_{n}^{(\alpha)}}
\newcommand{\Erho}{\E_{\rho}}
\newcommand{\Erhop}{\E_{\rho}'}

\newcommand{\msympinf}{P^{\Sym}_{\infty,p}}
\newcommand{\SSS}[3]{\mathcal{S}_{#1}(#2,#3)}
\newcommand{\treq}{\trianglelefteq_{\text{tr}}}

\newcommand{\rank}{\text{rank}}
\newcommand{\Diag}{\text{Diag}}

\newcommand{\qpochn}[3]{\left(#1;#2\right)_{#3}}

\begin{document}
\title{Distribution of Sandpile groups of random  bipartite graphs}
\author{Deepesh Singhal}

\begin{abstract}
Fix a prime $p$ and a constant $\frac{1}{p}<\alpha\leq 1$. Consider the random Erd\H{o}s--R\'enyi bipartite graph $G_{\alpha}(n,u)$ with bipartition $(V_1,V_2)$ of sizes $|V_1|=n$ and $|V_2|=\ceil{\alpha n}$, and edge probability $0<u<1$. 
The authors of \cite{bhargava2023rank} and \cite{FulmanKaplanSinghalWarnaar_SylowSandpileBipartite} conjectured a limiting distribution for the $p$-Sylow subgroup of the sandpile group of $G_{\alpha}(n,u)$ as $n\to\infty$.
We prove this conjecture for odd primes~$p$.

Similar results have previously been proved by computing the expected number of surjections from the random abelian $p$-group to $H$, for each finite abelian $p$-group $H$.
However, in our setting, these surjective moments often diverge to infinity, despite the conjectured limiting distribution having finite moments.
We resolve this issue by discarding the graphs for which too many vertices have degrees divisible by $p$. Once we remove the contribution of this rare set of graphs, then the surjective moments converge to the expected values.
When $p$ is odd, applying Wood's universality theorem yields the desired convergence in distribution.

For $p=2$, our computed moments (after excluding the rare set of graphs) match those of the conjectured distribution. However, these moments do not uniquely determine a distribution.
\end{abstract}

\maketitle

\section{Introduction}

Let $\Gamma$ be a graph on vertices $1,\dots,n$ without loops.  Write $\deg(i,j)$ for the number of edges between $i$ and $j$, and $\deg(j)=\sum_{i:i\neq j}\deg(i,j)$.
Define the Laplacian $L(\Gamma)$ by
\[
L_{ij}=\begin{cases}
\deg(i,j), & i\neq j,\\
-\deg(j), & i=j,
\end{cases}
\qquad\text{so that}\qquad
\sum_{i=1}^n L_{ij}=0\ \text{ for all }j.
\]
Let $\Z_0^n$ be the subset of $\Z^n$ consisting of vectors that sum up to $0$.
Hence $L:\Z^n\to\Z^n$ satisfies $\im(L)\subseteq \Z_0^n$.
The \emph{sandpile group} of $\Gamma$ is $S_{\Gamma}:=\Z_0^n/\im(L)$. 
If $\Gamma$ is connected, then $\text{corank}(L)=1$ and $S_{\Gamma}$ is a finite abelian group.

Let $G(n,v)$ denote the Erd\H{o}s-R\'enyi random graph on $n$ vertices in which each edge is included independently with probability $v\in(0,1)$. For fixed $v$, $G(n,v)$ is connected with probability tending to $1$ as $n\to\infty$. Consequently, the sandpile group $S_{G(n,v)}$ is a finite abelian group with probability tending to $1$.

Wood \cite[Theorem~1.1]{wood2017distribution} shows that as $n \rightarrow \infty$, the distribution of Sylow $p$-subgroups of the sandpile group of an Erd\H{o}s-R\'enyi random graph with $n$ vertices converges to $\msympinf$. The probability distribution $\msympinf$ on finite abelian $p$-groups is defined as follows
\[
\msympinf(G)
:=\frac{\abs{\{\text{symmetric, bilinear, perfect pairings }G\times G\to \C^{\times}\}}}{\abs{G} \abs{\Aut(G)}}
\prod_{k\geq 0} (1-p^{-1-2k}).
\]
\begin{theorem}\cite[Theorem~1.1]{wood2017distribution}
Consider a finite abelian $p$-group $G$. Then for a random graph~$G(n,v)$,
\[
\lim_{n \rightarrow \infty} 
\PPP[(S_{ G(n,v)})_p \cong G] 
=\msympinf(G).
\]
\end{theorem}
M\'esz\'aros \cite{meszaros2020distribution} determined the distribution of Sylow $p$-subgroups of sandpile groups of random $d$-regular graphs on $n$ vertices, as $n\to \infty$.  Suppose $d \ge 3$ and $n$ is even.  The random $d$-regular graph $H_n$ is obtained by taking the union of $d$ independent uniform random perfect matchings on the set of $n$ vertices.
Given a finite abelian $p$-group $G$, its \emph{rank} is $\log_p(\abs{G/pG})$.
\begin{theorem}\cite[Theorem 1.2]{meszaros2020distribution}\label{mes_thm}
Consider a prime $p$ and a finite abelian $p$-group $G$.
\begin{enumerate}
\item If $d$ is odd, or $d$ is even and $p$ is odd, we have 
\[
\lim_{n \rightarrow \infty} 
\PPP[(S_{H_n})_p \cong G] 
=\msympinf(G).
\]
\item Suppose $d$ is even and $p = 2$.  Then $(S_{H_n})_2$ always has odd rank.
Moreover, if $G$ is a finite abelian $2$-group of odd rank, then
\[
\lim_{n \rightarrow \infty} 
\PPP[(S_{H_n})_2 \cong G] 
= \abs{G/2G} P^{\Sym}_{\infty,2}(G).
\]
\end{enumerate}
\end{theorem}

We focus on the study of Sylow $p$-subgroups of random bipartite graphs.  Let $n_1 \ge n_2$ be positive integers and $u$ be a fixed real number $0<u<1$.
The Erd\H{o}s-R\'enyi random bipartite graph $G(n_1, n_2, u)$ is a
bipartite graph with vertex sets $V_1$ of size $n_1$ and $V_2$ of size
$n_2$, where each of the $n_1 n_2$ potential edges between a vertex in
$V_1$ and a vertex in $V_2$ is included independently with probability
$u$.  We are most interested in the case where $n_1$ and $n_2$ go to
infinity together and $n_2=\ceil{\alpha n_1}$ for a fixed constant
$0<\alpha\leq 1$. We denote $G_{\alpha}(n, u)=G(n, \ceil{\alpha
n}, u)$.
Let $S_{G_{\alpha}(n,u)}$ denote the sandpile group of the graph $G_{\alpha}(n, u)$.
In \cite{bhargava2023rank}, Bhargava, de Pascale, and Koenig give a conjecture for the distribution of $p$-Sylow subgroups of $S_{G_{\alpha}(n,u)}$ for odd $p$.
In \cite{FulmanKaplanSinghalWarnaar_SylowSandpileBipartite}, the authors gave the corresponding conjecture for $p=2$.

\begin{conjecture}\label{conj: p-Syl_conj}
Let $p$ be a prime, $0<u<1$, $ \frac{1}{p}<\alpha \leq 1$ and $G$ be a finite abelian $p$-group.  
\begin{enumerate}
\item If $p$ is odd, then
\[
\lim_{n \rightarrow \infty} 
\PPP[(S_{G_{\alpha}(n,u)})_p \cong G] 
= \msympinf(G).
\]
\item If $p = 2$, then
\[
\lim_{n \rightarrow \infty} 
\PPP[(S_{G_{\alpha}(n,u)})_2 \cong G] 
= \frac{\abs{G/2G}}{2} P^{\Sym}_{\infty,2}(G).
\]
\end{enumerate}
\end{conjecture}

Our main result is to prove Conjecture~\ref{conj: p-Syl_conj} for odd primes $p$.

\begin{theorem}\label{Thm: dist sandpile odd p}
Let $p$ be an odd prime, $0<u<1$, $ \frac{1}{p}<\alpha \leq 1$ and $G$ be a finite abelian $p$-group. Then we have
\[
\lim_{n \rightarrow \infty} 
\PPP[(S_{G_{\alpha}(n,u)})_p \cong G] 
= \msympinf(G).
\]
\end{theorem}

The authors of \cite{FulmanKaplanSinghalWarnaar_SylowSandpileBipartite} also conjectured the limiting distribution of ranks of $(S_{G_{\alpha}(n,u)})_p$. For this we denote the q-shifted factorials
\[
\qpochn{x}{q}{i}
:= (1-x) (1-xq) \cdots (1-xq^{i-1}),
\quad
\qpochn{x}{q}{\infty} 
:= \prod_{j=0}^{\infty} (1-xq^j).
\]

\begin{conjecture}\label{p-Syl_ranks_conj}\cite{FulmanKaplanSinghalWarnaar_SylowSandpileBipartite}
Let $p$ be prime, $\frac{1}{p}<\alpha\leq 1$, and $r$ be a nonnegative integer.
\begin{enumerate}
\item If $p$ is odd, then
\[
\lim_{n \rightarrow \infty} \PPP[\rank((S_{G_{\alpha}(n,u)})_p) = r]
=\frac{1}{\qpochn{-p^{-1}}{p^{-1}}{\infty}}  p^{-\binom{r+1}{2}} \frac{1}{\qpochn{p^{-1}}{p^{-1}}{r}}.
\]
\item If $p=2$, we have
\[
\lim_{n \rightarrow \infty} \PPP[\rank((S_{G_{\alpha}(n,u)})_2) = r] 
=\frac{1}{\qpochn{-1}{2^{-1}}{\infty}} 
2^{-\binom{r}{2}} 
\frac{1}{\qpochn{2^{-1}}{2^{-1}}{r}}.
\]
\end{enumerate}
\end{conjecture}

Our other main result is to prove Conjecture~\ref{p-Syl_ranks_conj} for all primes, including $p=2$.

\begin{theorem}\label{Thm: sandpile rank dist}
Let $p$ be prime, $\frac{1}{p}<\alpha\leq 1$, and $r$ be a nonnegative integer.
\begin{enumerate}
\item If $p$ is odd, then
\[
\lim_{n \rightarrow \infty} \PPP[\rank((S_{G_{\alpha}(n,u)})_p) = r]
=\frac{1}{\qpochn{-p^{-1}}{p^{-1}}{\infty}}  p^{-\binom{r+1}{2}} \frac{1}{\qpochn{p^{-1}}{p^{-1}}{r}}.
\]
\item If $p=2$, we have
\[
\lim_{n \rightarrow \infty} \PPP[\rank((S_{G_{\alpha}(n,u)})_2) = r] 
=\frac{1}{\qpochn{-1}{2^{-1}}{\infty}} 2^{-\binom{r}{2}} \frac{1}{\qpochn{2^{-1}}{2^{-1}}{r}}.
\]
\end{enumerate}
\end{theorem}

\subsection{Random matrix model}

Let $\Z_p$ be the ring of $p$-adic integers.
Fix $0<\epsilon<1$ for the rest of the paper.
A probability measure $\mu$ on $\Zp$ is called \emph{$\epsilon$-balanced} if for each $a\in \Z/p\Z$,
\[
\mu[\{x\in\Zp: x\equiv a\pmod{p}\}]<1-\epsilon.
\]
For the rest of this paper, fix a $\epsilon$-balanced probability measure $\mu$.

We consider a random matrix model for the Laplacian of $G_{\alpha}(n,u)$ as a random matrix $\Lnot\in \MM_{n_1+n_2}(\Z_p)$ as follows:
\[
\Lnot
=\begin{bmatrix}
-\sum a_{1j} & 0&  \dots & 0 & a_{11}& a_{12} & \dots & a_{1n_2}\\
0 & -\sum a_{2j} &\dots  & 0 & a_{21}& a_{22} & \dots & a_{2n_2}\\
\vdots & &\ddots & \vdots & \vdots & & & \vdots\\
0 & 0 &\dots  & -\sum a_{n_1j} & a_{n_11}& a_{n_12} & \dots & a_{n_1n_2}\\
a_{11}& a_{21} & \dots & a_{n_11} &-\sum a_{i1} & 0&  \dots & 0 \\
a_{12}& a_{22} & \dots & a_{n_12} &0 & -\sum a_{i2} &\dots  & 0 \\
\vdots & & & \vdots &\vdots & &\ddots & \vdots \\
a_{1n_2}& a_{2n_2} & \dots & a_{n_1n_2} &0 & 0 &\dots  & -\sum a_{in_2}\\
\end{bmatrix},
\]
where all $a_{ij}$ are independent random variables distributed according to $\mu$.
Let $Z_n$ consist of the vectors in $\Z_p^n$ whose entries sum to $0$.
We have $\Lnot\in\Hom(\Zpnot,\Zpnot)$ and $\im(\Lnot) \subseteq \Znot$. We denote $G(\Lnot)=\Znot/\im(\Lnot)$.
For $0<\alpha\leq 1$, we denote $L_{n,\ceil{\alpha n}}$ as $\Lna$.

In particular, if $\mu$ was the measure such that $\mu(\{1\})=u$ and $\mu(\{0\})=1-u$, then $\Lna$ would be the random matrix arising from the Laplacian of the random graph $G_{\alpha}(n,u)$.
Moreover, in the (high probability) event that $G_{\alpha}(n,u)$ is connected, the random finite abelian $p$-group $G(\Lna)$ would become isomorphic to the $p$-Sylow subgroup of $S_{G_{\alpha}(n,u)}$.

We will prove the following.
\begin{theorem}\label{Thm: dist laplacian odd p}
Let $p$ be an odd prime, $\frac{1}{p}<\alpha \leq 1$ and $G$ be a finite abelian $p$-group. 
Then we have
\[
\lim_{n \rightarrow \infty} \PPP[G(\Lna) \cong G] 
= \msympinf(G).
\]
\end{theorem}

\begin{theorem}\label{Thm: laplacian rank dist}
Let $p$ be prime, $\frac{1}{p}<\alpha\leq 1$, and $r$ be a nonnegative integer.
\begin{enumerate}
\item If $p$ is odd, then
\[
\lim_{n \rightarrow \infty} \PPP[\rank(G(\Lna)) = r]
=\frac{1}{\qpochn{-p^{-1}}{p^{-1}}{\infty}}  p^{-\binom{r+1}{2}} \frac{1}{\qpochn{p^{-1}}{p^{-1}}{r}}.
\]
\item If $p=2$, we have
\[
\lim_{n \rightarrow \infty} \PPP[\rank(G(\Lna)) = r] 
=\frac{1}{\qpochn{-1}{2^{-1}}{\infty}} 
2^{-\binom{r}{2}} 
\frac{1}{\qpochn{2^{-1}}{2^{-1}}{r}}.
\]
\end{enumerate}
\end{theorem}

It is clear that Theorem~\ref{Thm: dist laplacian odd p} implies Theorem~\ref{Thm: dist sandpile odd p} at once and Theorem~\ref{Thm: laplacian rank dist} implies Theorem~\ref{Thm: sandpile rank dist}.

\subsection{Previous results}

The Cohen-Lenstra distributions are defined as follows, for $0<u<p$,
\[
P_{\infty,u}(G)
:=\frac{\abs{G}^{\log_p(u)}}{\abs{\Aut(G)}}
\prod_{i=1}^{\infty}
\left(1-\frac{u}{p^i}\right).
\]
Friedman and Washington \cite{friedman1989distribution} showed that the limiting distribution of $\Coker(X)$, where $X$ is a Haar uniform element of $\mathbb{M}_n(\Zp)$ is $P_{\infty,1}$.
% \begin{theorem}\cite{friedman1989distribution}
% Let $X_n \in \mathbb{M}_n(\Z_p)$ be Haar-random. Then for every finite abelian $p$-group $G$,
% \[
% \lim_{n\to\infty} \PPP[\Coker(X_n)\cong G] = P_{\infty,1}(G).
% \]
% \end{theorem}
This result gives one of the earliest random matrix realizations of the Cohen-Lenstra distribution. It shows that the weighting by $\abs{\Aut(G)}^{-1}$ arises naturally from the cokernels of Haar-random $p$-adic matrices.

Wood \cite{wood2019random} proved a stronger result for rectangular matrices with independent $\epsilon$-balanced entries.

\begin{theorem}\cite{wood2019random}\label{Thm: wood coker rectangle}
Fix a prime $p$ and integer $u\geq 0$. Let $X_n\in\mathbb{M}_{n,n+u}(\Zp)$ be a random matrix whose entries are independent and $\epsilon$-balanced.
Then for every finite abelian $p$-group $G$,
\[
\lim_{n\to\infty} 
\PPP[\Coker(X_n) \cong G] 
= P_{\infty,p^{-u}}(G).
\]
\end{theorem}

Nguyen and Wood \cite{nguyen2022random} proved a global version for random integral matrices over $\Z$. 
Koplewitz \cite{koplewitz2017sandpile} showed that the Sylow $p$-subgroups of sandpile groups of random Erd\H{o}s--R\'enyi directed graphs converge to the distribution $P_{\infty,1/p}$.
Nguyen and Wood \cite{nguyen2022random} extended this result to the distribution of the entire sandpile group of random directed graphs.
In the companion paper \cite{Singhal_SandpileBipartite_Directed}, we study random directed bipartite graphs and determine their limiting distribution.

These results about cokernels of families of random $p$-adic matrices have been extended in several directions.
Cheong and Kaplan \cite{cheong2022generalizations} studied cokernels of polynomial expressions $P(A)$, where $A$ is a Haar-random matrix over $\Zp$.
Cheong and Yu \cite{cheong2023distribution} proved the $\epsilon$-balanced analogue for cokernels of polynomial expressions $P(A)$.
Cheong and Huang \cite{cheong2025cokernel} studied the distribution of $P(A+pB)$, where $P$ is a polynomial expression, $A$ is a fixed matrix and $B$ is Haar-random.
Lee \cite{lee2023joint} studied the joint distribution of several cokernels of the form $\Coker(P_1(A)),\dots,\Coker(P_r(A))$, showing that one can study finer questions than the distribution of a single cokernel. 
Finally, Nguyen and Wood \cite{nguyen2025local} developed a local and global universality theory for more structured random matrix ensembles, including symmetric, skew-symmetric, and Laplacian-type models.

For symmetric matrices, the limiting distribution differs from the Cohen-Lenstra distribution $P_{\infty,u}$. This is because the cokernel of a symmetric matrix naturally carries an additional algebraic structure, namely a canonical symmetric bilinear pairing.
Clancy, Kaplan, Leake, Payne, and Wood \cite{clancy2015cohen} proved that Haar-random symmetric matrices over $\Zp$ give rise to the distribution $\msympinf$.
Bhargava, Kane, Lenstra, Poonen, and Rains \cite{bhargava2015modeling} showed an analogous result for cokernels of Haar-random alternating matrices over $\Zp$. Let $\Aut(G,\delta)$ be the set of the automorphisms of $G$ that preserve a pairing $\delta$.

\begin{theorem}\cite{clancy2015cohen}
Let $X_n\in\mathbb{M}_n(\Zp)$ be Haar-random in the additive group of symmetric matrices. Let $G$ be a finite abelian $p$-group and let $\delta$ be a perfect symmetric bilinear pairing on $G$. Then
\[
\lim_{n\to\infty}
\PPP[(\Coker(X_n),\delta_{X_n})\cong (G,\delta)]
=
\frac{1}{\abs{G}\abs{\Aut(G,\delta)}}
\prod_{k\geq 0}(1-p^{-1-2k}).
\]
In particular,
\[
\lim_{n\to\infty}
\PPP[\Coker(X_n)\cong G]
=
\msympinf(G).
\]
\end{theorem}

Wood \cite{wood2017distribution} extended this result to random symmetric matrices, whose entries on and below the diagonal are independent and $\epsilon$-balanced.
Restricting her result to a single prime $p$, one obtains the following statement.

\begin{theorem}\cite{wood2017distribution}\label{Thm: Wood dist of sym matrix coker}
Let $X_n\in\mathbb{M}_n(\Zp)$ be a random symmetric matrix, whose entries on and below the diagonal are independent and $\epsilon$-balanced. Then for every finite abelian $p$-group $G$,
\[
\lim_{n\to\infty}
\PPP[\Coker(X_n)\cong G]
=
\msympinf(G).
\]
\end{theorem}

Hodges \cite{hodges2024distribution} computed the distribution of the sandpile group of a random graph along with its canonical pairing. 

\begin{theorem}\cite[Theorem 1.1]{hodges2024distribution}
Consider a finite abelian $p$-group $G$, equipped with a perfect symmetric bilinear pairing $\delta$. Then for a random graph~$G(n,v)$,
\[
\lim_{n \rightarrow \infty} 
\PPP[((S_{ G(n,v)})_p, \delta_{G(n,v),p} ) \cong (G,\delta)] 
=\frac{1}{\abs{G}\abs{\Aut(G,\delta)}}
\prod_{k\geq 0}(1-p^{-1-2k}).
\]
\end{theorem}

Shen \cite{shen2026quantative} later proved the corresponding universality theorem for cokernels of random symmetric matrices with their canonical pairings, using a different method and obtaining effective error bounds.

\section{Strategy}

Previous results about the distribution of sandpile groups of random graphs and cokernels of random matrices have been proven using the following strategy developed by Wood \cite{wood2017distribution}. Suppose $G_n$ is a sequence of random finite abelian $p$-groups and $\nu$ is the conjectured limiting distribution as $n\to\infty$.
The first step is to show that the surjective moments match. That is, for every finite abelian $p$-group $H$, we have
\begin{equation}\label{Eqn: step 1}
\lim_{n\to\infty} \E[\abs{\Sur(G_n,H)}]
=\E[\abs{\Sur(X,H)}],
\end{equation}
where $X$ is a random finite abelian $p$-group chosen according to $\nu$.
The next step is to use a result of Wood \cite{wood2017distribution}, which says that if \eqref{Eqn: step 1} is satisfied and the moments are not too big, then for every finite abelian $p$-group $H$
\[
\lim_{n\to\infty} \PPP[G_n\cong H]
=\nu(H).
\]

The moments of $\msympinf$ were computed by Clancy, Kaplan, Leake, Payne, and Wood \cite{clancy2015cohen}.
\begin{theorem}\cite[Theorem 11]{clancy2015cohen}\label{thm: moment of P Sym}
Suppose $p$ is a prime and $G$ is a finite abelian $p$-group. Let $X$ be a random finite abelian $p$-group distributed according to $\msympinf$. Then we have
\[
\E[\abs{\Sur(X,G)}]
=\abs{\Lambda^2 G}.
\]
\end{theorem}

The authors of \cite{FulmanKaplanSinghalWarnaar_SylowSandpileBipartite} compute the surjective moments of several families of distributions. One of them is defined in terms of a parameter $-1\leq a\leq 1$ as
\[
P^{M,1}_{p,a}(G)
:= 
\begin{cases}
\frac{1-a}{2}\times \abs{G/pG}\times \msympinf(G) &\text{if }\log_p(G/pG)\text{ is odd},\\
\frac{1+a}{2}\times \abs{G/pG}\times \msympinf(G) &\text{if }\log_p(G/pG)\text{ is even}.
\end{cases}
\]
The superscript $M$ is for `M\'esz\'aros'. This is to highlight the fact that the measure $P^{M,1}_{p,-1}$ occurs in the special case of Theorem \ref{mes_thm} where $d$ is even and $p=2$.
Also note that the measure $P^{M,1}_{p,0}$ occurs in Conjecture~\ref{conj: p-Syl_conj} when $p=2$.

The surjective moments of $P^{M,1}_{p,a}$ are as follows.
\begin{theorem}\cite[Theorem 9]{FulmanKaplanSinghalWarnaar_SylowSandpileBipartite}\label{Thm: moments of PM1_a}
Suppose $p$ is a prime, $-1\leq a\leq 1$ and $G$ is a finite abelian $p$-group. Let $X$ be a random finite abelian $p$-group distributed according to $P^{M,1}_{p,a}$. Then we have
\[
\E[\abs{\Sur(X,G)}]
=\abs{G/pG}\times\abs{\Lambda^2 G}.
\]
In particular, the moments do not depend on $a$. Moreover,
\[
\PPP[\rank(X)\text{ is odd}]
=\frac{1-a}{2}.
\]
\end{theorem}

Based on this, we should expect that for each finite abelian $p$-group $H$, we have
\[
\lim_{n\to\infty}
\E\left[\abs{\Sur(G(\Lna),H)}\right]
=\begin{cases}
\abs{\Lambda^2H} &\text{if }p\text{ is odd},\\
\abs{H/2H}\times \abs{\Lambda^2H} &\text{if }p=2.
\end{cases}
\]

However, this is not always the case, as we will show in Section~\ref{Sec: Raw moment infty}. These limits sometimes blow up to infinity.
The issue here is that a small subset of matrices contribute greatly to the surjective moments while contributing minimally to the distribution. We can get around this problem by restricting ourselves to a subset of the matrices whose probability goes to $1$ as $n\to\infty$, and computing the moments on this subset.
An identical issue arises in \cite{Singhal_SandpileBipartite_Directed} where the author studies the distribution of sandpile groups of random directed bipartite graphs.

\subsection{Good subset of matrices}

Given $\rho>0$, we define 
\[
\EE_{n,\rho}
=\left\{
\begin{bmatrix}
    t_1\\\vdots\\t_n
\end{bmatrix}\in(\Z/p\Z)^n:
\abs{\{i\in [n]:t_i = 0\}}
<\left(\frac{1}{p}+ \rho\right) n
\right\}.
\]
Define
\[
\eta(\Lnot)
=\begin{bmatrix}
    \sum_{j} a_{1j} \pmod{p}\\
    \vdots\\
    \sum_{j} a_{n_1j} \pmod{p}
\end{bmatrix}
\in (\Z/p\Z)^{n_1}.
\]
Even though $a_{ij}$ are not necessarily Haar uniform, as $n_2$ gets bigger $\sum_{j=1}^{n_2}a_{ij}$ would get close to being uniformly distributed $\mod p$. Therefore, as $n_1$ and $n_2$ get big, we should expect $\eta(\Lnot)\in\EE_{n_1,\rho}$ with a high probability. It is shown in Lemma~\ref{Lem: Prob E rho to 1} that this is indeed the case.

Given a function $f$ that takes $\Lnot$ as input, we denote the conditional expectation as
\[
\Erho[f(\Lnot)]
=\E[f(\Lnot) \mid \eta(\Lnot)\in \EE_{n_1,\rho}].
\]
We also denote the expectation restricted to the event $\eta(\Lnot)\in \EE_{n_1,\rho}$ as
\[
\Erhop[f(\Lnot)]
=\PPP[\eta(\Lnot)\in \EE_{n_1,\rho}] \times \E_{\rho}[f(\Lnot)].
\]

Once we restrict ourselves to this subset, then the issue of moments blowing up to infinity gets resolved and we get the right moments.

\begin{proposition}\label{Prop: Sur moment Laplacian}
Let $p$ be a prime, $\frac{1}{p}<\alpha \leq 1$.
Given $0<\rho<\alpha-\frac{1}{p}$ and a finite abelian $p$-group $G$, we have
\[
\lim_{n\to\infty} 
\Erho\left[\abs{\Sur(G(\Lna),G)}\right] 
=\begin{cases}
    \abs{\Lambda^2 G} &\text{if }p\text{ is odd},\\
    \abs{G/2G}\times\abs{\Lambda^2 G} &\text{if }p=2.
\end{cases}
\]\end{proposition}

Note that the restriction $\rho<\alpha-\frac{1}{p}$ is to avoid the matrices for which 
\[
\abs{\Big\{i\in [n]: \sum_j a_{ij} \equiv 0\pmod{p}\Big\}}
> n_2
=\ceil{\alpha n}.
\] 
These are the matrices responsible for the large surjective moments.
If $\alpha=1$, then no matrix can have 
\[
\abs{\Big\{i\in [n]: \sum_j a_{ij} \equiv 0\pmod{p}\Big\}}
> n_2
=n.
\]  
In fact, if $\alpha=1$, then the issue of moments blowing up does not arise and one could work directly with the raw moments $\lim_{n\to\infty} \E[\big|\Sur(G(L_n^{(1)}),G)\big|]$.

Once we have the moments, we can show that the desired distribution occurs for odd $p$ using the following result of Wood.
\begin{theorem}\cite[Theorem 8.3]{wood2017distribution}\label{Thm: Wood universality}
If $X_n$ is a sequence of random finitely generated $\Zp$-modules and $Y$ is a random finitely generated $\Zp$-module such that for every finite abelian $p$-group $G$, we have
\[
\lim_{n\to\infty} \E[\abs{\Sur(X_n,G)}]
=\E[\abs{\Sur(Y,G)}]
\leq \abs{\Lambda^2G}.
\]
Then for every finite abelian $p$-group $G$, we have
\[
\lim_{n\to\infty}
\PPP[X_n\cong G]
=\PPP[Y\cong G].
\]    
\end{theorem}

We now show that Proposition~\ref{Prop: Sur moment Laplacian} implies Theorem~\ref{Thm: dist laplacian odd p}.

\begin{proof}[Proof of Theorem~\ref{Thm: dist laplacian odd p} assuming Proposition~\ref{Prop: Sur moment Laplacian}]
For each $n$, let $G_n$ be a random finite abelian $p$-group distributed as
\[
\PPP[G_n\cong G]
= \PPP\Big[ G(\Lna)\cong G \mid \eta(\Lna)\in\EE_{n,\rho}\Big].
\]
Therefore, Proposition~\ref{Prop: Sur moment Laplacian} says that
\[
\lim_{n\to\infty} \E[\abs{\Sur(G_n,G)}] =\abs{\Lambda^2 G}.
\]
Therefore, from Theorem~\ref{thm: moment of P Sym} and Theorem~\ref{Thm: Wood universality}, we see that for each finite abelian $p$-group $G$, we have
\[
\lim_{n\to\infty} \PPP[ G(\Lna)\cong G \mid \eta(\Lna)\in\EE_{n,\rho}]
=\lim_{n\to\infty} \PPP[G_n\cong G]
=\msympinf(G).
\]
Note that
\begin{align*}
\PPP[G(\Lna)\cong G]
&= \PPP[ G(\Lna)\cong G \mid \eta(\Lna)\in\EE_{n,\rho}]
\times \PPP[\eta(\Lna)\in \EE_{n,\rho}]\\
&\quad\quad+\PPP[ G(\Lna)\cong G \mid \eta(\Lna)\not\in\EE_{n,\rho}]
\times \PPP[\eta(\Lna)\notin\EE_{n,\rho}].
\end{align*}
The result follows since Lemma~\ref{Lem: Prob E rho to 1} says that $\lim_{n\to\infty} \PPP[\eta(\Lna)\in \EE_{n,\rho}]=1$.
\end{proof}

This shows that for odd $p$, Conjecture~\ref{conj: p-Syl_conj} will follow once we prove Proposition~\ref{Prop: Sur moment Laplacian}.
The following result of Wood computes the distribution of ranks of a random finite abelian $p$-group distributed according to $\msympinf$.
\begin{proposition}\cite[Corollary 9.4]{wood2017distribution}\label{Prop: Wood P Sym rank dist}
If $X$ is a random finite abelian $p$-group chosen according to $\msympinf$, then
\[
\PPP[\rank(X)=r]
= \frac{1}{\qpochn{-p^{-1}}{p^{-1}}{\infty}}  p^{-\binom{r+1}{2}} \frac{1}{\qpochn{p^{-1}}{p^{-1}}{r}}.
\]
\end{proposition}
We see from this that for odd $p$, Conjecture~\ref{conj: p-Syl_conj} implies Conjecture~\ref{p-Syl_ranks_conj}.

For $p=2$, we will compute the surjective moments while proving Proposition~\ref{Prop: Sur moment Laplacian}.
However, Theorem~\ref{Thm: moments of PM1_a} shows that there is a one parameter family of measures with the same surjective moments, and possibly there could be more.
The measures within this family can be distinguished based on $\PPP[\rank(X)\text{ is odd}]$.
We will prove the following which provides further evidence for Conjecture~\ref{conj: p-Syl_conj} being true at $p=2$ as well.

\begin{proposition}\label{Prop: p=2 rank parity}
Suppose $p=2$. For any $0<\alpha\leq 1$,
\[
\lim_{n\to\infty}
\PPP
\left[ 
\rank(G(\Lna))
\text{ is odd}
\right]
=
\frac{1}{2}.
\]
\end{proposition}

Wood shows in \cite{wood2023probability} that if the surjective moments are $\abs{G/pG}\times \abs{\Lambda^2 G}$, then knowing the probability of odd rank determines the distribution of the ranks.

\begin{theorem}\cite[Corollary 2.12]{wood2023probability}\label{Thm: wood determine dist of rank}
If $X_n$ is a sequence of random abelian $p$-groups such that for each $k\in\Z_{\geq 0}$, we have
\[
\lim_{n\to\infty} \E[\abs{\Sur(X_n,(\Z/p\Z)^k)}]
= \abs{(\Z/p\Z)^k} \times \abs{\Lambda^2 (\Z/p\Z)^k},
\]
and there is $v\in [0,1]$ such that
\[
\lim_{n\to\infty} 
\PPP[\rank(X_n) \text{ is odd}]
=v.
\]
Then for every $r\in\Z_{\geq0}$, we have
\[
\lim_{n\to\infty}
\PPP[\rank(X_n)=r]
=\begin{cases}
2v\frac{1}{\qpochn{-1}{p^{-1}}{\infty}} p^{-\binom{r}{2}} \frac{1}{\qpochn{p^{-1}}{p^{-1}}{r}}  &\text{if $r$ is odd},\\
2(1-v)\frac{1}{\qpochn{-1}{p^{-1}}{\infty}} p^{-\binom{r}{2}} \frac{1}{\qpochn{p^{-1}}{p^{-1}}{r}}  &\text{if $r$ is even}.
\end{cases}
\]
\end{theorem}

\begin{proof}[Proof of Theorem~\ref{Thm: sandpile rank dist} assuming Proposition~\ref{Prop: Sur moment Laplacian} and Proposition~\ref{Prop: p=2 rank parity}]
We consider the cases of odd primes and $p=2$ separately.
\begin{enumerate}
\item First consider the case of odd prime $p$. We have shown that Proposition~\ref{Prop: Sur moment Laplacian} implies Theorem~\ref{Thm: dist sandpile odd p}. By Proposition~\ref{Prop: Wood P Sym rank dist} this implies the limiting distribution of ranks, proving Theorem~\ref{Thm: sandpile rank dist} for odd $p$.
\item Next, consider the case $p=2$.
For each $n$, let $G_n$ be a random finite abelian $p$-group distributed as
\[
\PPP[G_n\cong G]
= \PPP\Big[ G(\Lna)\cong G \mid \eta(\Lna)\in\EE_{n,\rho}\Big].
\]
Therefore, Proposition~\ref{Prop: Sur moment Laplacian} says that
\[
\lim_{n\to\infty} \E[\abs{\Sur(G_n,G)}] 
= \abs{G/2G}\times \abs{\Lambda^2 G}.
\]
Proposition~\ref{Prop: p=2 rank parity} says that
\[
\lim_{n\to\infty}
\PPP
\left[ 
\rank(G(\Lna))
\text{ is odd}
\right]
=
\frac{1}{2}.
\]
The law of total probability tells us that
\begin{align*}
\PPP
\left[ 
\rank(G(\Lna))
\text{ is odd}
\right]
&=\PPP
\left[ 
\rank(G(\Lna))
\text{ is odd}
| \eta(\Lna)\in \EE_{n,\rho}
\right]
\times \PPP[\eta(\Lna)\in \EE_{n,\rho}]\\
&\quad+
\PPP
\left[ 
\rank(G(\Lna))
\text{ is odd}
| \eta(\Lna)\notin \EE_{n,\rho}
\right]
\times \PPP[\eta(\Lna)\notin \EE_{n,\rho}].
\end{align*}
Lemma~\ref{Lem: Prob E rho to 1} says that $\lim_{n\to\infty} \PPP[\eta(\Lna)\in \EE_{n,\rho}]=1$, so we see that
\[
\lim_{n\to\infty}
\PPP
\left[ 
\rank(G_n)
\text{ is odd}
\right]
=
\frac{1}{2}.
\]
Now, Theorem~\ref{Thm: wood determine dist of rank} tells us that
\[
\lim_{n\to\infty}
\PPP[\rank(G_n)=r]
=\frac{1}{\qpochn{-1}{p^{-1}}{\infty}} p^{-\binom{r}{2}} \frac{1}{\qpochn{p^{-1}}{p^{-1}}{r}}.
\]
Again, by the law of total probability, we see that
\begin{align*}
\PPP
\left[ 
\rank(G(\Lna))=r
\right]
&=\PPP
\left[ 
\rank(G(\Lna))=r
| \eta(\Lna)\in \EE_{n,\rho}
\right]
\times \PPP[\eta(\Lna)\in \EE_{n,\rho}]\\
&\quad+
\PPP
\left[ 
\rank(G(\Lna))=r
| \eta(\Lna)\notin \EE_{n,\rho}
\right]
\times \PPP[\eta(\Lna)\notin \EE_{n,\rho}].
\end{align*}
Since $\lim_{n\to\infty} \PPP[\eta(\Lna)\in \EE_{n,\rho}]=1$, the result follows. \qedhere
\end{enumerate}
\end{proof}

We suspect that for $p=2$, Proposition~\ref{Prop: Sur moment Laplacian} and Proposition~\ref{Prop: p=2 rank parity} should uniquely determine the distribution of finite abelian $2$-groups, not just the distribution of ranks.
This would imply Conjecture~\ref{conj: p-Syl_conj} for $p=2$. However, we have not been able to prove this.

\subsection{Hom-moments}

The notation $H\trianglelefteq G$ means $H$ is a subgroup of $G$.
We say that $T$ is a translated subgroup of $G$, if $T$ is a coset of some subgroup of $G$. The notation $T\treq G$ will mean that $T$ is a translated subgroup of $G$. If $T$ is a coset of a subgroup $H$, then we will denote $\G(T)=H$.
We will sometimes consider chains of translated subgroups. By $T\treq S\treq G$, we mean that $T$ and $S$ are translated subgroups of $G$ and $T$ is a subset of $S$.

If $G$ is a finite abelian $p$-group, then $\Lambda^2 G$ is the quotient of $G\otimes G$ by the subgroup generated by elements of the form $f\otimes f$.
We define $\Delta^2 G$ to be the quotient of $G\otimes G$ by the subgroup generated by elements of the form $f\otimes g+g\otimes f$.
It is clear that
\[
\abs{\Delta^2 G}
=\begin{cases}
\abs{\Lambda^2 G} &\text{if }p\text{ is odd},\\
\abs{\Lambda^2 G}\times\abs{G/2G} &\text{if }p=2.
\end{cases}
\]

The following result implies Proposition~\ref{Prop: Sur moment Laplacian}.
Note that we are replacing $G(\Lna)$ with $\Coker(\Lna)$ and surjective moments with Hom-moments.

\begin{proposition}\label{Prop: Hom moment coker}
Let $p$ be a prime, $\frac{1}{p}<\alpha \leq 1$.
Given $0<\rho<\alpha-\frac{1}{p}$ and a finite abelian $p$-group $G$, we have
\[
\lim_{n\to\infty} 
\Erho\left[\abs{\Hom(\Coker(\Lna),G)}\right] 
=\sum_{T\treq G}\abs{\Delta^2 \G(T)} \times \abs{\G(T)}.
\]
\end{proposition}

\begin{proof}[Proof of Proposition~\ref{Prop: Sur moment Laplacian} assuming Proposition~\ref{Prop: Hom moment coker}]
Notice that $\Zpnot=\Znot\oplus \Zp$, so\\
$\Coker(\Lnot)\cong G(\Lnot)\oplus \Zp$. Hence, for any finite abelian $p$-group $G$, we have
\[
\abs{\Hom(\Coker(\Lna),G))}
=\abs{\Hom(G(\Lna),G)} \times \abs{G}. 
\]
Therefore, Proposition~\ref{Prop: Hom moment coker} implies that
\begin{align*}
\lim_{n\to\infty} 
\Erho
\left[
\abs{\Hom(G(\Lna),G)}
\right] 
&=\frac{1}{\abs{G}}
\lim_{n\to\infty}
\Erho
\left[
\abs{\Hom(\Coker(\Lna),G)}
\right] \\
&=\frac{1}{\abs{G}} 
\sum_{T\treq G}
\abs{\Delta^2 \G(T)} \times \abs{\G(T)}.
\end{align*}
Now, for each $H\trianglelefteq G$, there are $\frac{\abs{G}}{\abs{H}}$ translated subgroups $T\treq G$ for which $\G(T)=H$.
Therefore,
\[
\lim_{n\to\infty} 
\Erho\left[\abs{\Hom(G(\Lna),G)}\right] 
=\frac{1}{\abs{G}} \sum_{H\trianglelefteq G} \frac{\abs{G}}{\abs{H}} \abs{\Delta^2 H} \times \abs{H}
=\sum_{H\trianglelefteq G}  \abs{\Delta^2 H}.
\]
Since
\[
\abs{\Hom(G(\Lna),G)}
=\sum_{H\trianglelefteq G} \abs{\Sur(G(\Lna),H)},
\]
by induction on the size of $G$, we see that
\[
\lim_{n\to\infty} 
\Erho\left[\abs{\Sur(G(\Lna),G)}\right] 
=  \abs{\Delta^2 G}.\qedhere
\]
\end{proof}

Therefore, for odd $p$, Conjecture~\ref{conj: p-Syl_conj} will follow once we prove Proposition~\ref{Prop: Hom moment coker}.

\subsection{Decomposing Expectation}

We can decompose the expected number of homomorphisms to $G$ as a sum of probabilities similarly to Wood \cite{wood2017distribution}.
Each map in $\Hom(\Coker(\Lnot),G)$ can be lifted to a map in $\Hom(\Zpnot,G)$. Moreover, $F\in \Hom(\Zpnot,G)$ descends to a map in
$\Hom(\Coker(\Lnot),G)$ if and only if $F\circ \Lnot=0\in\Hom(\Zpnot,G)$.
Moreover,
\[
\Hom(\Zpnot,G)
=\Hom(\Zpno,G)\oplus\Hom(\Zpnt,G).
\]
Therefore,
\begin{align*}
&\Erhop[\abs{\Hom(\Coker(\Lnot),G)}]\\
&=\sum_{F\in\Hom(\Zpnot,G)} \PPP\big[F\circ \Lnot=0 \in \Hom(\Zpnot,G), \eta(\Lnot)\in \EE_{n,\rho} \big]\\
&=\sum_{F_1\in\Hom(\Zpno,G)} \sum_{F_2\in\Hom(\Zpnt,G)} \PPP[(F_1,F_2) \circ \Lnot=0 \in \Hom(\Zpnot,G), \eta(\Lnot)\in \EE_{n,\rho}].
\end{align*}

Since $G$ is a finite abelian $p$-group, it is also a $\Zp$-module.
Given a subset $X\subseteq G$, the subgroup of $G$ generated by $X$ is denoted by
\[
\sg{X}
=\left\{\sum a_k x_k : a_k\in \Zp, x_k\in X\right\}.
\]
For a subset $X\subseteq G$, we denote the smallest translated subgroup containing $X$ by
\[
\sgo{X}
=\left\{\sum a_k x_k : a_k\in \Zp, x_k\in X, \sum a_k=1\right\}.
\]
% Note that $\sgo{X}$ is a coset of the following subgroup of $G$
% \[
% \sgz{X}
% =\left\{\sum a_k x_k : a_k\in \Zp, x_k\in X, \sum a_k=0\right\}.
% \]
Let $[n]=\{1,\dots,n\}$.
Denote the standard basis of $\Zpn$ by $e_1,\dots,e_n$.
For $F\in\Hom(\Zpn,G)$, by abuse of notation we denote
\begin{align*}
&\set{F}
=\{F(e_i): i\in [n]\},
&
&\sg{F}
=\sg{F(e_i): i\in [n]},
&
&\sgo{F}
=\sgo{F(e_i): i\in [n]}.
\end{align*}

\begin{definition}
Let $T$ be a translated subgroup of $G$. We say that $F\in \Hom(\Zpn,G)$ is a \emph{translated code} of distance $d$ with image $T$, 
if for every $\sigma\subseteq [n]$ with $\abs{\sigma}<d$, we have
\[
\sgo{F(e_i):i\in [n]\setminus \sigma}
=T.
\]
\end{definition}
Note that if $F\in \Hom(\Zpn,G)$ is a \emph{translated code} of distance $d$ with image $T$, then $\sgo{F}=T$.
If $T_2\treq T_1\treq G$ is a chain of translated subgroups, then we denote $[T_1:T_2]=\frac{\abs{T_1}}{\abs{T_2}}$.

This is closely related to the notion of \emph{code} introduced by Wood \cite{wood2017distribution}.
In fact $F\in \Hom(\Zpn,G)$ is a translated code of distance $d$ with image $T$, if and only if for any $g\in T$, the map $-g+F\in \Hom(\Zpn,G)$ is a code of distance $d$ with image $\G(T)$.

\begin{definition}
Given $\delta>0$ and $F\in\Hom(\Zpn,G)$, the $\delta$-depth of $F$ is defined to be the largest positive integer $D$, 
for which there is a subset $\sigma\subseteq [n]$ with $\abs{\sigma}<\log_p(D)\delta n$ such that 
\[
[\sgo{F}:\sgo{F(e_i):i\in [n]\setminus\sigma}]=D,
\] 
or $1$ if there is no such $\sigma$.

If such a $\sigma$ exists, then we say that $F$ is of $\delta$-type $(\sgo{F},\sgo{F(e_i):i\in [n]\setminus\sigma})$. If there is no such $\sigma$, then we set the $\delta$-type to be $(\sgo{F}, \sgo{F})$.
\end{definition}

Our definition of $\delta$-depth is closely related to, but conflicts with the definition provided by Wood \cite{wood2017distribution}. 

Note that if $F$ has $\delta$-depth $1$, then it is a translated code of distance $\delta n$.
Given a chain of translated subgroups $T\treq S\treq G$, we denote
\[
\SSS{n,\delta}{S}{T}
:=\{F\in \Hom(\Zp^n,G): \delta-\text{type }(S,T)\}.
\]
This gives a cover
\[
\Hom(\Zp^{n},G)
=\bigcup_{T\treq S\treq G}
\SSS{n,\delta}{S}{T}.
\]
In general, this union need not be disjoint, since a given $F$ may have $\delta$-type $(S,T)$ for more than one choice of $T$. However, the pieces $\SSS{n,\delta}{S}{S}$ consist of translated codes with image $S$, and these pieces are pairwise disjoint.

Proposition~\ref{Prop: Hom moment coker} will follow from the following result.

\begin{proposition}\label{Prop: moments on parts}
Let $p$ be a prime, $\frac{1}{p}<\alpha \leq 1$, $0<\rho<\alpha-\frac{1}{p}$ and $G$ be a finite abelian $p$-group.
Consider chains of translated subgroups $T_1\treq S_1\treq G$, $T_2\treq S_2\treq G$. We have
\begin{align*}
&\lim_{n\to\infty}
\sum_{F_1\in\SSS{n,\delta}{S_1}{T_1}} \sum_{F_2\in\SSS{\ceil{\alpha n},\delta}{S_2}{T_2}} \PPP[(F_1,F_2) \circ \Lna=0, \eta(\Lna)\in \EE_{n,\rho}]\\
&\quad\quad=\begin{cases}
\abs{\Delta^2 \G(T_1)} \times \abs{\G(T_1)} &\text{if } T_1=T_2=S_1=S_2,\\
0 &\text{otherwise}.
\end{cases}
\end{align*}
\end{proposition}

\begin{proof}[Proof of Proposition~\ref{Prop: Hom moment coker} assuming Proposition~\ref{Prop: moments on parts}]
Proposition~\ref{Prop: moments on parts} implies at once that
\[
\lim_{n\to\infty} 
\Erhop\left[\abs{\Hom(\Coker(\Lna),G)}\right] 
=\sum_{T\treq G}\abs{\Delta^2 \G(T)} \times \abs{\G(T)}.
\]
Since 
\[
\Erhop\left[\abs{\Hom(\Coker(\Lna),G)}\right]
=\PPP[\eta(\Lna)\in\EE_{n,\rho}] \times \Erho\left[\abs{\Hom(\Coker(\Lna),G)}\right],
\]
the result follows from Lemma~\ref{Lem: Prob E rho to 1}.
\end{proof}

\subsection{Characters}

Denote the standard basis of $\Zpno$ as $e_1,\dots,e_{n_1}$ and the standard basis of $\Zpnt$ as $\epsilon_1,\dots,\epsilon_{n_2}$.
Since $F_1\in\Hom(\Zpno,G)$, $F_2\in \Hom(\Zpnt,G)$ are determined by their values on the basis, we can treat them as functions in $F_1\in\Hom([n_1],G)$, $F_2\in\Hom([n_2],G)$ with $F_1(i)=F_1(e_i)$ and $F_2(j)=F_2(\epsilon_j)$.

For $F_1\in\Hom([n_1],G)$, $F_2\in \Hom([n_2],G)$, denote $T_1=\sgo{F_1}$, $T_2=\sgo{F_2}$.
Then $(F_1,F_2)\circ \Lnot$ is as follows
\begin{equation}\label{Eqn: F1, F2, Lnot}
\begin{split}
((F_1,F_2)\circ \Lnot) (e_i)
&= \sum_{j=1}^{n_2} a_{ij} (F_2(j)-F_1(i)) \in \sg{-F_1(i)+T_2},\\
((F_1,F_2)\circ \Lnot) (\epsilon_j)
&= \sum_{i=1}^{n_1} a_{ij} (F_1(i)-F_2(j)) \in \sg{-F_2(j)+T_1}.
\end{split}
\end{equation}

Given a translated subgroup $T\treq G$ and $F\in\Hom([n],G)$, we define $\CC_{n}(T,F)$ to be the set of maps
\[
F': [n] \to \bigcup_{i=1}^n \sg{-F(i)+T}, 
\]
for which each $F'(i)\in \sg{-F(i)+T}$.
Clearly,
\[
\abs{\CC_n(T,F)}
=\prod_{i=1}^n \abs{\sg{-F(i)+T}}.
\]
Note that $(F_1,F_2)\circ\Lnot\in \CC_{n_1}(\sgo{F_2},F_1)\times \CC_{n_2}(\sgo{F_1},F_2)$.

Denote
\[
G^*= \Hom(G,\C^\times);
\]
this is a multiplicative group.

For a translated subgroup $T\treq G$ and an element $g\in G$ such that $g\notin T$, we construct a character of $\sg{-g+T}$ which will help us determine which $(F_1',F_2')$ can occur as $(F_1,F_2)\circ \Lnot$ if we are given that $\eta(\Lnot)=\vec{t}$.

\begin{lemma}\label{Lem: char existence}
Given a translated subgroup $T\treq G$ and an element $g\in G$ such that $g\notin T$, there is a character $\chi_{T,g}\in \sg{-g+T}^*$ such that for every $x\in -g+T$, $\chi_{T,g}(x)=\zeta_p$.
\end{lemma}
\begin{proof}
To show this we need to show that whenever we have $h_1,\dots, h_m\in T$ and $a_1,\dots,a_m\in \Zp$ for which $\sum_k a_k (-g+h_k)=0$, then $\prod_k \zeta_p^{ a_k}=1$.

Suppose $\sum_k a_k (-g+h_k)=0$, then $(\sum_k a_k) g=\sum_k a_k h_k$. Now, if $\sum_k a_k\not\equiv 0\pmod{p}$, then it is a unit, which means
\[
g
= \sum_l \frac{a_l}{\sum_k a_k} h_l 
\in T. 
\]
This is a contradiction, so $\sum_k a_k\equiv 0\pmod{p}$. Thus $\prod_k \zeta_p^{ a_k}=1$.
\end{proof}

\begin{lemma}\label{Lem: ti right value}
Suppose $\eta(\Lnot)=\vec{t}\in(\Z/p\Z)^{n_1}$.
Consider $F_1\in\Hom([n_1],G)$ and $F_2\in\Hom([n_2],G)$ such that $\sgo{F_1}=T_1$ and $\sgo{F_2}=T_2$.
If 
\[
(F_1,F_2)\circ\Lnot
=(F_1',F_2')
\in \CC_{n_1}(\sgo{F_2},F_1)\times \CC_{n_2}(\sgo{F_1},F_2),
\]
then for every $i\in [n_1]$, if $F_1(i)\notin T_2$, then $\chi_{T_2,F_1(i)}(F_1'(i))=\zeta_p^{t_i}$.
\end{lemma}
\begin{proof}
Since $\eta(\Lnot)=\vec{t}$, we know that $\sum_{j\in [n_2]}a_{ij} \equiv t_i\pmod{p}$.
Also since $(F_1,F_2) \circ \Lnot=(F_1',F_2')$, we know that 
\[
\sum_{j\in [n_2]}a_{ij}(-F_1(i)+F_2(j)) 
=F_1'(i).
\] 
This implies that
\begin{align*}
\chi_{T_2,F_1(i)}(F_1'(i))
&=\chi_{T_2,F_1(i)}\left( \sum_{j\in [n_2]}a_{ij}(-F_1(i)+F_2(j)) \right)\\
&=\prod_{j\in [n_2]}
\chi_{T_2,F_1(i)}\left(-F_1(i)+F_2(j) \right)^{a_{ij}}
=\prod_{j\in [n_2]}
\zeta_p^{a_{ij}}
=\zeta_p^{t_i}. \qedhere
\end{align*}
\end{proof}

This lemma implies that if there is an index $i\in [n_1]$, such that $F_1(i)\in T_1\setminus T_2$ and $t_i\neq 0$ then
\[
\PPP[(F_1,F_2)\circ \Lnot=0, \eta(\Lnot)=\vec{t}]
=0.
\]
This indicates that if $T_1\setminus T_2\neq \emptyset$, then we should expect
\[
\sum_{F_1\in\SSS{n,\delta}{T_1}{T_1}} \sum_{F_2\in\SSS{\ceil{\alpha n},\delta}{T_2}{T_2}} 
\PPP[(F_1,F_2) \circ \Lna=0, \eta(\Lna)=\vec{t}]
\]
to be much larger for $\vec{t}\notin \EE_{n,\rho}$ than for $\vec{t}\in \EE_{n,\rho}$. 
This suggests that restricting to the event $\vec{t}\in \EE_{n,\rho}$, might prevent the moments from blowing up to $\infty$.

\subsection{Coefficients}

Given a translated subgroup $T\treq G$ and $F\in\Hom([n],G)$, we define $\CC^*_{n}(T,F)$ to be the set of the maps
\[
C: [n] \to \bigcup_{i=1}^n \sg{-F(i)+T}^*, 
\]
for which each $C(i)\in \sg{-F(i)+T}^*$.
Clearly,
\[
\abs{\CC^*_n(T,F)}
=\prod_{i=1}^n \abs{\sg{-F(i)+T}}.
\]

We can decompose the probability of $(F_1,F_2)\circ \Lnot =(F_1',F_2')$ as a sum over characters, similarly to Wood \cite{wood2017distribution}.
For any $F_1'\in \CC_{n_1}(\sgo{F_2},F_1)$, $F_2'\in \CC_{n_2}(\sgo{F_1},F_2)$, by \eqref{Eqn: F1, F2, Lnot}, we have
\begin{equation}\label{Eqn: Prob F1 F2 Lnot = F'}
\begin{split}
&\mathbbm{1}_{(F_1,F_2)\circ \Lnot =(F_1',F_2')}\\
&=\frac{1}{\abs{\CC^*_{n_1}(\sgo{F_2},F_1)}}
\frac{1}{\abs{\CC^*_{n_2}(\sgo{F_1},F_2)}}
\sum_{C_1\in \CC^*_{n_1}(\sgo{F_2},F_1)}
\sum_{C_2\in \CC^*_{n_2}(\sgo{F_1},F_2)}\\
&\quad\quad
\prod_{i=1}^{n_1} C_1(i)\Big((F_1,F_2)\circ \Lnot(e_i) -F_1'(i)\Big)
\prod_{j=1}^{n_2} C_2(j)\Big((F_1,F_2)\circ \Lnot(\epsilon_j) -F_2'(j)\Big)\\
&=\frac{1}{\abs{\CC^*_{n_1}(\sgo{F_2},F_1)}}
\frac{1}{\abs{\CC^*_{n_2}(\sgo{F_1},F_2)}}
\sum_{C_1\in \CC^*_{n_1}(\sgo{F_2},F_1)}
\sum_{C_2\in \CC^*_{n_2}(\sgo{F_1},F_2)}
\prod_{i=1}^{n_1} C_1(i)(-F_1'(i))\\
&\quad\quad
\prod_{j=1}^{n_2} C_2(j)(-F_2'(j))
\prod_{i=1}^{n_1} \prod_{j=1}^{n_2}
\Big(
C_1(i)\big(F_2(j)-F_1(i)\big)
C_2(j)\big(F_1(i)-F_2(j)\big)
\Big)^{a_{ij}}.
\end{split}
\end{equation}

Similarly to Wood \cite{wood2017distribution}, we define the coefficients
\[
E[F_1,F_2,C_1,C_2](i,j)
= C_1(i)\big(F_2(j)-F_1(i)\big)
C_2(j)\big(F_1(i)-F_2(j)\big).
\]

Suppose $\vec{x}\in (\Z/p\Z)^{n_1}$ is a random vector, then for fixed $\vec{t}\in (\Z/p\Z)^{n_1}$, we have
\[
\mathbbm{1}_{\vec{x}=\vec{t}}
=\frac{1}{p^{n_1}}
\sum_{\vec{s}\in (\Z/p\Z)^{n_1}} \zeta_p^{\vec{s}\cdot (\vec{x}-\vec{t})}.
\]

Therefore, for $F_1\in \Hom([n_1],G)$, $F_2\in\Hom([n_2],G)$, $\vec{t}\in (\Z/p\Z)^{n_1}$, $F_1'\in\CC_{n_1}(\sgo{F_2},F_1)$ and $F_2'\in\CC_{n_2}(\sgo{F_1},F_2)$,
we have
\begin{align*}
&\PPP[(F_1,F_2)\circ \Lnot=(F_1',F_2') \text{ and } \eta(\Lnot)=\vec{t}]\\
&= \E[\mathbbm{1}_{(F_1,F_2)\circ \Lnot=(F_1',F_2')} \mathbbm{1}_{\eta(\Lnot)=\vec{t}} ]\\
&=\frac{1}{\abs{\CC^*_{n_1}(\sgo{F_2},F_1)}}
\frac{1}{\abs{\CC^*_{n_2}(\sgo{F_1},F_2)}}
\frac{1}{p^{n_1}}
\sum_{C_1\in \CC^*_{n_1}(\sgo{F_2},F_1)}
\sum_{C_2\in \CC^*_{n_2}(\sgo{F_1},F_2)}
\sum_{\vec{s}\in (\Z/p\Z)^{n_1}} \\
&\quad\quad
\zeta_p^{-\vec{s}\cdot\vec{t}}
\prod_{i=1}^{n_1} C_1(i)(-F_1'(i))
\prod_{j=1}^{n_2} C_2(j)(-F_2'(j))
\E\left[
\prod_{i=1}^{n_1} \prod_{j=1}^{n_2}
\Big(
E[F_1,F_2,C_1,C_2](i,j)
\Big)^{a_{ij}}
\prod_{i=1}^{n_1} \zeta_p^{s_i \sum_{j=1}^{n_2}a_{ij}}
\right].
\end{align*}
We denote
\begin{align*}
E'[F_1,F_2,C_1,C_2,\vec{s}](i,j)
&=E[F_1,F_2,C_1,C_2](i,j) \zeta_p^{s_i}\\
&= C_1(i)\big(F_2(j)-F_1(i)\big)
C_2(j)\big(F_1(i)-F_2(j)\big) 
\zeta_p^{s_i}.
\end{align*}
Since the $a_{ij}$ are independent, we see that
\begin{equation}\label{Eqn: prob fourier sum}
\begin{split}
&\PPP[(F_1,F_2)\circ \Lnot=(F_1',F_2') \text{ and } \eta(\Lnot)=\vec{t}]\\
&=\frac{1}{\abs{\CC^*_{n_1}(\sgo{F_2},F_1)}}
\frac{1}{\abs{\CC^*_{n_2}(\sgo{F_1},F_2)}}
\frac{1}{p^{n_1}}
\sum_{C_1\in \CC^*_{n_1}(\sgo{F_2},F_1)}
\sum_{C_2\in\CC^*_{n_2}(\sgo{F_1},F_2)}
\sum_{\vec{s}\in (\Z/p\Z)^{n_1}}\\
&\quad\quad\quad\quad
\zeta_p^{-\vec{s}\cdot\vec{t}}
\prod_{i=1}^{n_1} C_1(i)(-F_1'(i))
\prod_{j=1}^{n_2} C_2(j)(-F_2'(j))
\prod_{i=1}^{n_1} \prod_{j=1}^{n_2}
\E
\left[
E'[F_1,F_2,C_1,C_2,\vec{s}](i,j) ^{a_{ij}}
\right].
\end{split}
\end{equation}

\begin{definition}
Given $F_1\in \Hom([n_1],G)$, $F_2\in \Hom([n_2],G)$, $C_1\in \CC^*_{n_1}(\sgo{F_2},F_1)$, $C_2\in \CC^*_{n_2}(\sgo{F_1},F_2)$ and $\vec{s}\in (\Z/p\Z)^{n_1}$, we say that $(C_1,C_2)$ is $(F_1,F_2,\vec{s})$-\emph{special} if for every $i\in [n_1]$ and every $j\in [n_2]$, we have $E'[F_1,F_2,C_1,C_2,\vec{s}](i,j)=1$.
\end{definition}
This is analogous to the definition of a special character given by Wood in \cite{wood2017distribution}.

As a heuristic,
\begin{align*}
&\PPP[(F_1,F_2)\circ \Lnot=0 \text{ and }\eta(\Lnot)=\vec{t}]\\
&\approx 
\frac{1}{\abs{\CC^*_{n_1}(\sgo{F_2},F_1)}\times \abs{\CC^*_{n_2}(\sgo{F_1},F_2)} p^{n_1}}
\sum_{\vec{s}\in (\Z/p\Z)^{n_1}} \zeta_p^{-\vec{s}\cdot \vec{t}}
\abs{\{(C_1,C_2): (F_1,F_2,\vec{s})\text{-special}\}}.
\end{align*}
This would be an exact equation if all $a_{ij}$ were Haar uniform.

In general, if the $a_{ij}$ are not necessarily Haar uniform, then $\E
\left[
E'[F_1,F_2,C_1,C_2,\vec{s}](i,j) ^{a_{ij}}
\right]$
might not be exactly zero. However, the following lemma of Wood suggests that it would be small as long as $a_{ij}$ are $\epsilon$-balanced.

\begin{lemma}\label{Lem: Wood bound xi^ epsilon balanced}\cite[Lemma 4.2]{wood2017distribution}
If $\xi\neq 1$ is a $b^{th}$ root of unity and $a\in\Zp$ is $\epsilon$-balanced, then
\[
\abs{\E[\xi^{a}]}
\leq \exp\left({-\frac{\epsilon}{b^2}}\right).
\]
\end{lemma}

In \cite{wood2017distribution}, Wood divides the set of non-special $C\in \Hom([n],G^*)$ based on a constant $\gamma>0$ into those that are robust and not robust.
She shows that most $C$ are robust, and that for robust $C$, at least quadratically many of the coefficients $E(i,j)$ are not $1$.
On the other hand for the few $C$ that are not special and not robust, at least linearly many of the coefficients $E(i,j)$ are not $1$.

We follow a similar approach to Wood, but we need to partition $\CC^*_{n_1}(T_2,F_1)$, $\CC^*_{n_2}(T_1,F_2)$ and $(\Z/p\Z)^{n_1}$. We will need slightly different definitions of robustness for the three.
\begin{definition}
\begin{enumerate}
\item Given a translated subgroup $T$ of $G$, $F\in \Hom([n],G)$ and $C\in \CC^*_{n}(T,F)$, we say that $C$ is $(d,F)$-\emph{robust} if for every $\sigma\subseteq [n]$  with $\abs{\sigma}<d$, there are $i_1,i_2\in [n]\setminus \sigma$ such that $F(i_1)=F(i_2)$ and $C(i_1)\neq C(i_2)$.

\item Given a translated subgroup $T$ of $G$, $F\in \Hom([n],G)$, $C\in \CC^*_{n}(T,F)$ and $\vec{s}\in (\Z/p\Z)^n$, we say that $C$ is $(d,F,\vec{s})$-\emph{robust} if for every $\sigma\subseteq [n]$  with $\abs{\sigma}<d$, there are $i_1,i_2\in [n]\setminus \sigma$ such that $F(i_1)=F(i_2)$, $s_{i_1}=s_{i_2}$ and $C(i_1)\neq C(i_2)$.

\item Given $F\in\Hom([n],G)$, $\vec{s}\in(\Z/p\Z)^n$ and a translated subgroup $T\treq G$, we say that $\vec{s}$ is $(d,F,T)$-robust, if for every $\sigma\subseteq [n]$ with $\abs{\sigma}<d$, there are $i_1,i_2\in [n]\setminus \sigma$ such that $F(i_1)=F(i_2)\in T$ and $s_{i_1}\neq s_{i_2}$.
\end{enumerate}
\end{definition}

\section{Robust $C_1$ or $C_2$}\label{Sec: Robust}

In this section, we consider $F_1\in \Hom([n_1],G)$, $F_2\in \Hom([n_2],G)$, $C_1\in \CC^*_{n_1}(\sgo{F_2},F_1)$, $C_2\in \CC^*_{n_2}(\sgo{F_1},F_2)$ and $\vec{s}\in (\Z/p\Z)^{n_1}$ such that $F_1$ and $F_2$ are translated codes and at least one of $C_1$, $C_2$ or $\vec{s}$ is robust.
Our goal is to show that in this case quadratically many of the coefficients $E'[F_1,F_2,C_1,C_2,\vec{s}](i,j)$ are $\neq1$.

\begin{lemma}\label{Lem: Coef one implies F determin C}
Suppose $F_1\in \Hom([n_1],G)$, $F_2\in \Hom([n_2],G)$, $C_1\in \CC^*_{n_1}(\sgo{F_2},F_1)$, $C_2\in \CC^*_{n_2}(\sgo{F_1},F_2)$ and $\vec{s}\in (\Z/p\Z)^{n_1}$.
Suppose $\tau_1\subseteq [n_1]$ and $\tau_2\subseteq [n_2]$ are such that for each $i\in \tau_1$, $j\in \tau_2$
\[
E'[F_1,F_2,C_1,C_2,\vec{s}](i,j)
=1.
\]
Denote $T_1=\sgo{F_1(i): i\in \tau_1}\subseteq \sgo{F_1}$ and $T_2=\sgo{F_2(j): j\in \tau_2}\subseteq \sgo{F_2}$. Then the following hold:
\begin{enumerate}
    \item If $T_1=\sgo{F_1}$, then for every $j_1, j_2\in \tau_2$ with $F_2(j_1)=F_2(j_2)$, we have $C_2(j_1)= C_2(j_2)$.
    \item If $T_2=\sgo{F_2}$, then for every $i_1, i_2\in \tau_1$ with $F_1(i_1)=F_1(i_2)$ and $s_{i_1}=s_{i_2}$, we have $C_1(i_1)= C_1(i_2)$.
    \item For every $i_1, i_2\in \tau_1$ with $F_1(i_1)=F_1(i_2)\in T_2$, we have $s_{i_1}=s_{i_2}$.
\end{enumerate}
\end{lemma}
\begin{proof}
First, assume $T_1=\sgo{F_1}$ and consider $j_1, j_2\in \tau_2$ with $F_2(j_1)=F_2(j_2)=h$. For each $i\in \tau_1$, we know that
\[
C_2(j_1)(F_1(i)-h)
=C_1(i)(F_1(i)-h) \zeta_p^{-s_i},
\] 
and 
\[
C_2(j_2)(F_1(i)-h)
=C_1(i)(F_1(i)-h)\zeta_p^{-s_i},
\] 
meaning
\[
C_2(j_1)(-h+F_1(i))
=C_2(j_2)(-h+F_1(i)).
\]
It follows that for every $x\in -h+T_1$, we have $C_2(j_1)(x)
=C_2(j_2)(x)$.
Since $T_1=\sgo{F_1}$ and the characters $C_2(j_1)$ and $C_2(j_2)$ are defined on $\sg{-h+\sgo{F_1}}$, we see that $C_2(j_1)=C_2(j_2)$.

Next, assume that $T_2=\sgo{F_2}$ and consider $i_1, i_2\in \tau_1$ with $F_1(i_1)=F_1(i_2)=g$ and $s_{i_1}=s_{i_2}$.
For each $j\in \tau_2$, we have
\[
C_1(i_1)(-g+F_2(j))
= C_2(j)(-g+F_2(j)) \zeta_p^{-s_{i_1}},
\] 
and
\[
C_1(i_2)(-g+F_2(j))
= C_2(j)(-g+F_2(j)) \zeta_p^{-s_{i_2}},
\]
meaning
\[
C_1(i_1)(-g+F_2(j))
=C_1(i_2)(-g+F_2(j)).
\]
It follows that for every $x\in -g+T_2$, we have $C_1(i_1)(x)=C_1(i_2)(x)$.
Since $T_2=\sgo{F_2}$ and the characters $C_1(i_1)$ and $C_1(i_2)$ are defined on $\sg{-g+\sgo{F_2}}$, we see that $C_1(i_1)=C_1(i_2)$.

For the last part, consider $i_1, i_2\in \tau_1$ for which $F_1(i_1)=F_1(i_2)\in T_2$. Denote $F_1(i_1)=F_1(i_2)=g$.
For each $j\in \tau_2$, we have
\[
C_1(i_1)(-g+F_2(j))
= C_2(j)(-g+F_2(j)) \zeta_p^{-s_{i_1}},
\] 
and
\[
C_1(i_2)(-g+F_2(j))
= C_2(j)(-g+F_2(j)) \zeta_p^{-s_{i_2}},
\]
meaning
\[
C_1(i_1)(-g+F_2(j))
= \zeta_p^{s_{i_2}-s_{i_1}}
C_1(i_2)(-g+F_2(j)).
\]
It follows that for every $x\in -g+T_2$, we have $C_1(i_1)(x)=\zeta_p^{s_{i_2}-s_{i_1}} C_1(i_2)(x)$. Since $g\in T_2$, we can take $x=0$, which implies $s_{i_1}=s_{i_2}$.
\end{proof}

\begin{lemma}\label{Lem: quad many non one}
Suppose $F_1\in \Hom([n_1],G)$, $F_2\in \Hom([n_2],G)$, $C_1\in \CC^*_{n_1}(\sgo{F_2},F_1)$, $C_2\in \CC^*_{n_2}(\sgo{F_1},F_2)$ and $\vec{s}\in (\Z/p\Z)^{n_1}$.
Then the following hold.
\begin{enumerate}
    \item If $F_1$ is a translated code of distance $d_1$ with image $T_1$ and $C_2$ is $(d_2,F_2)$-robust then
    \[
    \abs{\{(i,j)\in [n_1]\times [n_2]: E'[F_1,F_2,C_1,C_2,\vec{s}]\neq 1\}}
    \geq \frac{d_1 d_2}{p\abs{G}^2}.
    \]
    \item If $F_2$ is a translated code of distance $d_1$ with image $T_2$ and $C_1$ is $(d_2,F_1,\vec{s})$-robust then
    \[
    \abs{\{(i,j)\in [n_1]\times [n_2]: E'[F_1,F_2,C_1,C_2,\vec{s}]\neq 1\}}
    \geq \frac{d_1 d_2}{\abs{G}^2}.
    \]
    \item If $F_2$ is a translated code of distance $d_1$ with image $T_2$, and $\vec{s}$ is $(d_2,F_1, T_2)$-robust, then
    \[
    \abs{\{(i,j)\in [n_1]\times [n_2]: E'[F_1,F_2,C_1,C_2,\vec{s}]\neq 1\}}
    \geq \frac{d_1 d_2}{\abs{G}^2}.
    \]
\end{enumerate}
\end{lemma}
\begin{proof}
\begin{enumerate}[align=left, leftmargin=*]
    \item First, suppose that $F_1$ is a translated code of distance $d_1$ with image $T_1$ and $C_2$ is $(d_2,F_2)$-robust.
    Consider 
    \[
    X_1
    =\left\{g\in T_1 : \abs{\{i\in [n_1]:F_1(i)=g\}}\geq \frac{d_1}{\abs{G}} \right\}.
    \]
    It is clear that $\abs{\{i\in [n_1]: F_1(i)\notin X_1\}} <d_1$. Since $F_1$ is a translated code of distance $d_1$ with image $T_1$, this means $\sgo{X_1}=T_1$.
    
    Enumerate the elements of $X_1$ as $A_1,\dots,A_m$.
    By the pigeonhole principle, we know that there is some $B_k\in \sg{-A_k+\sgo{F_2}}^*$ and $u_k\in \Z/p\Z$ for which
    \[
    \abs{S(A_k,B_k,u_k)}
    =\abs{\{i\in [n_1]: F_1(i)=A_k, C_1(i)=B_k, s_i=u_k\}}
    \geq \frac{1}{p\abs{\sg{-A_k+\sgo{F_2}}}} \frac{d_1}{\abs{G}}
    \geq \frac{d_1}{p\abs{G}^2}.
    \]
    Let $\tau_1=\bigcup_{k=1}^{m} S(A_k,B_k,u_k)$.
    Denote
    \begin{align*}
    &\sigma
    =\{j\in [n_2]: \exists i\in \tau_1 \text{ such that } E'[F_1,F_2,C_1,C_2,\vec{s}](i,j)\neq 1\},
    &
    &\tau_2=[n_2]\setminus \sigma.
    \end{align*}
    
    This means that for each $i\in\tau_1$ and $j\in \tau_2$, we have $ E'[F_1,F_2,C_1,C_2,\vec{s}](i,j)=1$.
    Since $\sgo{X_1}=\sgo{F_1}$, Lemma~\ref{Lem: Coef one implies F determin C} implies that for $j_1,j_2\in \tau_2$, if $F_2(j_1)=F_2(j_2)$ then $C_2(j_1)=C_2(j_2)$.
    Since $C_2$ is $(d_2,F_2)$-robust, it follows that $\abs{\sigma}\geq d_2$.
    
    For each $j\in \sigma$, there is at least one $i\in \tau_1$ for which $E'[F_1,F_2,C_1,C_2,\vec{s}](i,j)\neq 1$. If $i\in S(A_k,B_k,u_k)$, then it follows that for every $i'\in S(A_k,B_k,u_k)$, $E'[F_1,F_2,C_1,C_2,\vec{s}](i',j)\neq 1$.
    The result follows.

    \item Next, suppose that $F_2$ is a translated code of distance $d_1$ with image $T_2$ and $C_1$ is $(d_2,F_1,\vec{s})$-robust. 
    Consider 
    \[
    X_2
    =\left\{h\in T_2 : \abs{\{j\in [n_2]:F_2(j)=h\}}\geq \frac{d_1}{\abs{G}} \right\}.
    \]
    It is clear that $\abs{\{j\in [n_2]: F_2(j)\notin X_2\}} <d_1$. Since $F_2$ is a translated code of distance $d_1$ with image $T_2$, this means $\sgo{X_2}=T_2$.
    
    Enumerate the elements of $X_2$ as $A_1,\dots,A_m$.
    By the pigeonhole principle, we know that there is some $B_k\in \sg{-A_k+\sgo{F_1}}^*$ for which
    \[
    \abs{S(A_k,B_k)}
    =\abs{\{j\in [n_2]: F_2(j)=A_k, C_2(j)=B_k\}}
    \geq \frac{d_1}{\abs{\sg{-A_k+\sgo{F_1}}} \abs{G}}
    \geq \frac{d_1}{\abs{G}^2}.
    \]
    Let $\tau_2=\bigcup_{k=1}^{m} S(A_k,B_k)$.
    Denote
    \begin{align*}
    &\sigma
    =\{i\in [n_1]: \exists j\in \tau_2 \text{ such that } E'[F_1,F_2,C_1,C_2,\vec{s}](i,j)\neq 1\},
    &
    &\tau_1=[n_1]\setminus \sigma.
    \end{align*}
    This means that for each $i\in\tau_1$ and $j\in \tau_2$, we have $ E'[F_1,F_2,C_1,C_2,\vec{s}](i,j)=1$.
    
    Since $\sgo{X_2}=\sgo{F_2}$, Lemma~\ref{Lem: Coef one implies F determin C} implies that for $i_1,i_2\in \tau_1$, if $F_1(i_1)=F_1(i_2)$ and $s_{i_1}=s_{i_2}$ then $C_1(i_1)=C_1(i_2)$.
    Since $C_1$ is $(d_2,F_1,\vec{s})$-robust, it follows that $\abs{\sigma}\geq d_2$.
    
    For each $i\in \sigma$, there is at least one $j\in \tau_2$ for which $E'[F_1,F_2,C_1,C_2,\vec{s}](i,j)\neq 1$. If $j\in S(A_k,B_k)$, then it follows that for every $j'\in S(A_k,B_k)$, $E'[F_1,F_2,C_1,C_2,\vec{s}](i,j')\neq 1$.
    The result follows.

    \item
    For the last part, suppose $F_2$ is a translated code of distance $d_1$ with image $T_2$ and $\vec{s}$ is $(d_2,F_1, T_2)$-robust.
    Consider 
    \[
    X_2
    =\left\{h\in T_2 : \abs{\{j\in [n_2]:F_2(j)=h\}}\geq \frac{d_1}{\abs{G}} \right\}.
    \]
    It is clear that $\abs{\{j\in [n_2]: F_2(j)\notin X_2\}} <d_1$. Since $F_2$ is a translated code of distance $d_1$ with image $T_2$, this means $\sgo{X_2}=T_2$.
    
    Enumerate the elements of $X_2$ as $A_1,\dots,A_m$.
    For each $1\leq k\leq m$, by the pigeonhole principle, we know that there is some $B_k\in \sg{-A_k+\sgo{F_1}}^*$ for which
    \[
    \abs{S(A_k,B_k)}
    =\abs{\{j\in [n_2]: F_2(j)=A_k, C_2(j)=B_k\}}
    \geq \frac{d_1}{\abs{\sg{-A_k+\sgo{F_1}}} \abs{G}}
    \geq \frac{d_1}{\abs{G}^2}.
    \]
    Let $\tau_2=\bigcup_{k=1}^{m} S(A_k,B_k)$.
    Denote
    \begin{align*}
    &\sigma
    =\{i\in [n_1]: \exists j\in \tau_2 \text{ such that } E'[F_1,F_2,C_1,C_2,\vec{s}](i,j)\neq 1\},
    &
    &\tau_1=[n_1]\setminus \sigma.
    \end{align*}
    This means that for each $i\in\tau_1$ and $j\in \tau_2$, we have $ E'[F_1,F_2,C_1,C_2,\vec{s}](i,j)=1$.

    Lemma~\ref{Lem: Coef one implies F determin C} tells us that for $i_1,i_2\in \tau_1$, if $F_1(i_1)=F_1(i_2)\in T_2$, then $s_{i_1}=s_{i_2}$.
    Since $\vec{s}$ is $(d_2,F_1,T_2)$-robust, it follows that $\abs{\sigma}\geq d_2$.

    For each $i\in \sigma$, there is at least one $j\in \tau_2$ for which $E'[F_1,F_2,C_1,C_2,\vec{s}](i,j)\neq 1$. If $j\in S(A_k,B_k)$, then it follows that for every $j'\in S(A_k,B_k)$, $E'[F_1,F_2,C_1,C_2,\vec{s}](i,j')\neq 1$.
    The result follows. \qedhere
\end{enumerate}
\end{proof}

\section{Special $(C_1,C_2)$}\label{Sec: Special}

Consider $F_1\in \Hom([n_1],G)$, $F_2\in \Hom([n_2],G)$ and $\vec{s}\in (\Z/p\Z)^{n_1}$.
In this section, our goal is to count the number of $(C_1,C_2)$, with $C_1\in \CC^*_{n_1}(\sgo{F_2},F_1)$ and $C_2\in \CC^*_{n_2}(\sgo{F_1},F_2)$, that are $(F_1,F_2,\vec{s})$-special. 
%We denote the set of such $(C_1,C_2)$ as $\Special(F_1,F_2,\vec{s})$.

We will show the following.

\begin{proposition}\label{Prop: Count special}
Consider $F_1\in \Hom([n_1],G)$ and $F_2\in \Hom([n_2],G)$. Then the following hold.
\begin{enumerate}
\item The number of triples $(\vec{s},C_1,C_2)$, with $\vec{s}\in (\Z/p\Z)^{n_1}$, $C_1\in \CC^*_{n_1}(\sgo{F_2},F_1)$ and $C_2\in \CC^*_{n_2}(\sgo{F_1},F_2)$, such that $(C_1,C_2)$ is $(F_1,F_2,\vec{s})$-special is at most
\[
p^{\abs{G}}
\times \abs{G}^{(p+1)\abs{G}}
\times p^{\abs{ \{i\in [n_1]: F_1(i)\notin \sgo{F_2}\} }}.
\]
\item If $\set{F_1}=\set{F_2}=T$ for some $T\treq G$, then the number of $C_1\in \CC^*_{n_1}(\sgo{F_2},F_1)$ and $C_2\in \CC^*_{n_2}(\sgo{F_1},F_2)$, such that $(C_1,C_2)$ is $(F_1,F_2,\vec{s})$-special is exactly
\[
=\begin{cases}
\abs{\Delta^2 \G(T)} \times \abs{T} &\text{if }\vec{s}=0,\\
0 &\text{if }\vec{s}\neq0.
\end{cases}
\]
\end{enumerate}
\end{proposition}

\subsection{Upper bound}

Lemma~\ref{Lem: Coef one implies F determin C} tells us that if $(C_1,C_2)$ is $(F_1,F_2,\vec{s})$-special, then the map $C_1$ factors through 
\[
C_1: [n_1]\xrightarrow[]{(F_1,s)} G\times (\Z/p\Z)\xrightarrow[]{\beta_1} \bigcup_{g\in \sgo{F_1}} \sg{-g+\sgo{F_2}}^*,
\] 
for some $\beta_1$ and
the map $C_2$ factors through
\[
C_2: [n_2]\xrightarrow[]{F_2} G\xrightarrow[]{\beta_2} \bigcup_{h\in \sgo{F_2}} \sg{-h+\sgo{F_1}}^*,
\]
for some $\beta_2$.

Given $F\in \Hom([n],G)$ and $\vec{s}\in (\Z/p\Z)^n$, we denote
\[
\set{F,\vec{s}}
=\{(F(e_i),s_i): i\in [n]\} 
\subseteq G\times (\Z/p\Z).
\]
We will denote the projections as $\pi_1:G\times (\Z/p\Z)\to G$ and $\pi_2:G\times (\Z/p\Z)\to \Z/p\Z$.

Given subsets $Y_1\subseteq G\times (\Z/p\Z)$ and $X_2\subseteq G$, let $X_1=\pi_1(Y_1)$. We denote by $\D(Y_1,X_2)$, the set of pairs of functions $(\beta_1,\beta_2)$ such that the following hold:
\begin{enumerate}
    \item $\beta_1$ maps $Y_1$ to $\bigcup_{g\in X_1}\sg{-g+ X_2}^*$ and for every $(g,s)\in Y_1$, $\beta_1(g,s)\in \sg{-g+X_2}^*$;
    \item 
    $\beta_2$ maps $X_2$ to $\bigcup_{h\in X_2}\sg{-h+X_1}^*$ and for every $h\in X_2$, $\beta_2(h)\in \sg{-h+X_1}^*$;
    \item For every $(g,s)\in Y_1$ and $h\in X_2$ we have
    \[
    \beta_1(g,s)(-g+h)\beta_2(h)(-h+g) \zeta_p^{s}=1.
    \]
\end{enumerate}
Clearly, 
\[
\abs{\D(Y_1,X_2)} \leq \abs{G}^{(p+1)\abs{G}}.
\]

\begin{lemma}\label{Lem: special iff factor through}
Consider $F_1\in \Hom(\Zpno,G)$, $F_2\in \Hom(\Zpnt,G)$ and $\vec{s}\in (\Z/p\Z)^{n_1}$. 
For $C_1\in \CC^*(\sgo{F_2},F_1)$, $C_2\in \CC^*(\sgo{F_1},F_2)$, the pair $(C_1,C_2)$ is $(F_1,F_2,\vec{s})$-special if and only if there is $(\beta_1,\beta_2)\in \mathcal{D}(\set{F_1,\vec{s}},\set{F_2})$ such that $C_1(i)=\beta_1(F_1(i),s_i)$ and $C_2(j)=\beta_2(F_2(j))$.

Therefore, the number of pairs $(C_1,C_2)$ that are $(F_1,F_2,\vec{s})$-special is $\abs{\mathcal{D}(\set{F_1,\vec{s}},\set{F_2})}$.
\end{lemma}
\begin{proof}
Denote $Y_1=\set{F_1,\vec{s}}$, $X_1=\set{F_1}$ and $X_2=\set{F_2}$.
If $(C_1,C_2)$ is $(F_1,F_2,\vec{s})$-special, then by Lemma~\ref{Lem: Coef one implies F determin C}, we can construct following functions
\begin{itemize}
    \item $\beta_1:Y_1\to \bigcup_{g\in X_1}\sg{-g+\sgo{X_2}}^*$. For $(g,s)\in Y_1$, choose $i\in [n_1]$ such that $F_1(i)=g$, $s_i=s$, and define $\beta_1(g,s)=C_1(i)\in \sg{-g+\sgo{X_2}}^*$. By Lemma~\ref{Lem: Coef one implies F determin C}, this is well defined.
    \item $\beta_2:X_2\to \bigcup_{h\in X_2}\sg{-h+\sgo{X_1}}^*$. For $h\in X_2$, choose $j\in [n_2]$ such that $F_2(j)=h$, and define $\beta_2(h)=C_2(j)\in \sg{-h+\sgo{X_1}}^*$. By Lemma~\ref{Lem: Coef one implies F determin C}, this is well defined.
\end{itemize}
Since $(C_1,C_2)$ is $(F_1,F_2,\vec{s})$-special, we see that for each $(g,s)\in Y_1$ and $h\in X_2$, we have
\[
\beta_1(g,s)(-g+h)\beta_2(h)(-h+g) \zeta_p^s=1.
\]
Therefore, $(\beta_1,\beta_2)\in \mathcal{D}(Y_1,X_2)$.

Conversely, given $(\beta_1,\beta_2)\in \mathcal{D}(Y_1,X_2)$, we can construct $C_1$ and $C_2$ as $C_1(i)=\beta_1(F_1(i),s_i)$ and $C_2(j)=\beta_2(F_2(j))$.
For $i\in [n_1]$, $j\in[n_2]$, we see that
\begin{align*}
E'[F_1,F_2,C_1,C_2,\vec{s}](i,j)
&=C_1(i)(-F_1(i)+F_2(j))
C_2(j)(-F_2(j)+F_1(i)) \zeta_p^{s_i}\\
&=\beta_1(F_1(i),s_i)(-F_1(i)+F_2(j))
\beta_2(F_2(j)) (-F_2(j)+F_1(i)) \zeta_p^{s_i}
=1.
\end{align*}
Therefore, $(C_1,C_2)$ is $(F_1,F_2,\vec{s})$-special.
\end{proof}

\begin{lemma}\label{Lem: 2 s for same g}
Suppose that we have non-empty subsets $Y_1\subseteq G\times (\Z/p\Z)$ and $X_2\subseteq G$.
If $(g,s), (g,s')\in Y_1$, then for every $(\beta_1,\beta_2)\in \D(Y_1,X_2)$ and every $x\in -g+\sgo{X_2}$, we have
\[
\beta_1(g,s')(x)
=\zeta_p^{s-s'} \beta_1(g,s)(x). 
\]
\end{lemma}
\begin{proof}
For each $h\in X_2$, we have
\[
\beta_1(g,s)(-g+h) 
=\big( \zeta_p^{s}\beta_2(h)(-h+g)\big)^{-1},
\]
\[
\beta_1(g,s')(-g+h) 
= \big(\zeta_p^{s'}\beta_2(h)(-h+g)\big)^{-1}.
\]
This means that for each $x\in -g+X_2$, we have
\[
\beta_1(g,s')(x)
=\zeta_p^{s-s'} \beta_1(g,s)(x). 
\]
It follows that the same would hold for every $x$ in $\sgo{-g+X_2}=-g+\sgo{X_2}$.
\end{proof}

\begin{lemma}\label{Lem: only one s}
Suppose that we have non-empty subsets $Y_1\subseteq G\times (\Z/p\Z)$ and $X_2\subseteq G$.
Let $X_1=\pi_1(Y_1)$.
Then the following hold.
\begin{enumerate}
    \item If there is some $g\in X_1\cap\sgo{X_2}$, for which there are (at least) two different $(g,s)$ and $(g,s')$ in $Y_1$, then $\D(Y_1,X_2)=\emptyset$.
    \item If there is some $g\in X_1\cap X_2$ with $(g,s)\in Y_1$ and $s\neq 0$, then $\D(Y_1,X_2)=\emptyset$.
\end{enumerate}
\end{lemma}
\begin{proof}
First, suppose that there is some $g\in X_1\cap\sgo{X_2}$ for which there are (at least) two different $(g,s)$ and $(g,s')$ in $Y_1$.
Assume for the sake of contradiction that $\D(Y_1,X_2)$ is not empty, consider $(\beta_1,\beta_2)\in\D(Y_1,X_2)$.
By Lemma~\ref{Lem: 2 s for same g}, we know that for every $x\in -g+\sgo{X_2}$, we have
\[
\beta_1(g,s')(x)
=\zeta_p^{s-s'} \beta_1(g,s)(x). 
\]
Since $g\in \sgo{X_2}$, we know that $0\in -g+\sgo{X_2}$. Taking $x=0$, we see that $1=\zeta_p^{s-s'}$. This contradicts the fact that $s\neq s'$. We conclude that $\D(Y_1,X_2)$ is empty.

Next, suppose that there is some $g\in X_1\cap X_2$ with $(g,s)\in Y_1$ and $s\neq 0$. Assume for the sake of contradiction that $\D(Y_1,X_2)\neq \emptyset$. Consider $(\beta_1,\beta_2)\in \D(Y_1,X_2)$. Since $(g,s)\in Y_1$ and $g\in X_2$, we have
\[
\beta_1(g,s)(-g+g)\beta_2(g)(-g+g) \zeta_p^{s}=1.
\]
This means $\zeta_p^s=1$, which contradicts $s\neq 0$. Therefore, $\D(Y_1,X_2)= \emptyset$. 
\end{proof}

We are now ready to prove the first part of Proposition~\ref{Prop: Count special}.

\begin{proof}[Proof of part (1) of Proposition~\ref{Prop: Count special}]
First we count the number of $\vec{s}\in (\Z/p\Z)^{n_1}$ for which $\D(\set{F_1,\vec{s}}, \set{F_2})$ is non-empty.
By the first part of Lemma~\ref{Lem: only one s},
if $\D(\set{F_1,\vec{s}}, \set{F_2})\neq\emptyset$,
then there is a function $\lambda: \set{F_1}\cap \sgo{F_2}\to\Z/p\Z$
such that for each $i\in [n_1]$,
if $F_1(i)\in \sgo{F_2}$, then $s_i=\lambda(F_1(i))$.

The number of choices for $\lambda$ is at most
\[
p^{\abs{\set{F_1}\cap \sgo{F_2}}}
\leq p^{\abs{G}}.
\]
Once we fix $\lambda$, then the number of choices for $\vec{s}$ is
\[
p^{\abs{\{i\in [n_1]: F_1(i)\notin \sgo{F_2}\} }}.
\]
Once we choose $\vec{s}$, then the number of $(C_1,C_2)$ that are $(F_1,F_2,\vec{s})$-special is
\[
\abs{\D(\set{F_1,\vec{s}}, \set{F_2})}
\leq \abs{G}^{(p+1)\abs{G}}.
\]
The result follows.
\end{proof}

\subsection{Exact count for $\set{F_1}=\set{F_2}=T$}

Given a translated subgroup $T\treq G$, we denote by $D(T,T)$, the set of pairs of functions $(\beta_1,\beta_2)$ such that the following hold:
\begin{enumerate}
    \item $\beta_1$ maps $T$ to $\G(T)^*$;
    \item $\beta_2$ maps $T$ to $\G(T)^*$;
    \item For every $g, h\in T$, we have
    \[
    \beta_1(g)(-g+h)\beta_2(h)(-h+g)=1.
    \]
\end{enumerate}
Clearly, 
\[
\abs{D(T,T)}  \leq \abs{G}^{2\abs{G}}.
\]

\begin{proof}[Proof of part (2) of Proposition~\ref{Prop: Count special}]
If $\vec{s}\neq 0$, then the second part of Lemma~\ref{Lem: only one s} implies that
\[
\abs{\D(\set{F_1,\vec{s}}, T)}=0.
\]
On the other hand, if $\vec{s}=0$, then $\set{F_1,\vec{s}}=T\times\{0\}$. In this case it is clear that
\[
\abs{\D(\set{F_1,\vec{s}}, T)}
=\abs{D(T,T)}.
\]
Therefore, we will be done if we show that $\abs{D(T,T)}=\abs{\Delta^2\G(T)}\times \abs{T}$.

Let $\G(T)=H$. We will denote the class of $x\otimes y$ in $\Delta^2 H$ by $x\wedge y$.
We can see that $\abs{D(T,T)}=\abs{D(H,H)}$, due to the following bijection. Fix $g_0\in T$.
\begin{itemize}
\item If $(\beta_1,\beta_2)\in D(T,T)$, then we can define $\alpha_1:H\to H^*$ and $\alpha_2:H\to H^*$ as $\alpha_1(x)=\beta_1(g_0+x)$, $\alpha_2(x)=\beta_2(g_0+x)$. It is clear that $(\alpha_1,\alpha_2)\in D(H,H)$.
\item If $(\alpha_1,\alpha_2)\in D(H,H)$, then we can define $\beta_1:T\to H^*$ and $\beta_2:T\to H^*$ as $\beta_1(g)=\alpha_1(-g_0+g)$, $\beta_2(g)=\alpha_2(-g_0+g)$.
It is clear that $(\beta_1,\beta_2)\in D(T,T)$.
\end{itemize}

Therefore, it is enough to compute the size of $D(H,H)$, that is, the number of maps $\alpha_1:H\to H^*$ and $\alpha_2:H\to H^*$ such that for every $x, y\in H$, we have
\[
\alpha_1(x)(-x+y)\alpha_2(y)(-y+x)=1.
\]
This is equivalent to saying that for every $y,z\in H$,
\begin{equation}\label{Eqn: C}
\alpha_2(y)(z)
= \alpha_1(z+y)(z).
\end{equation}
So $\alpha_2$ is determined by $\alpha_1$. We want to count how many 
$\alpha_1:H\to H^*$ satisfy the condition that the $\alpha_2(y)$ obtained from \eqref{Eqn: C} are characters.
We are counting the $\alpha_1:H\to H^*$, such that for every $y,z_1,z_2\in H$
\begin{equation}\label{Property 0}
\alpha_1(y+z_1)(z_1) 
\alpha_1(y+z_2)(z_2)
=\alpha_1(y+z_1+z_2)(z_1+z_2).
\end{equation}
Let $A$ be the set of all such $\alpha_1$.

Consider the map $f: H^*\times (\Delta^2(H))^* \to A$ defined as follows. For $\theta\in H^*$, $\iota\in (\Delta^2(H))^*$ and $x,y\in H$, we define
\[
f(\theta,\iota)(x)(y)=\theta(y) \iota(x \wedge y).
\]
Note that $f(\theta,\iota)(x)$ is a character since
\begin{align*}
f(\theta,\iota)(x)(y_1+y_2)
&=\theta(y_1+y_2)\iota(x\wedge (y_1+y_2))\\ &=\theta(y_1)\theta(y_2)\iota(x \wedge y_1)\iota(x \wedge y_2) \\
&=f(\theta,\iota)(x)(y_1)f(\theta,\iota)(x)(y_2).  
\end{align*}

Moreover, $f(\theta,\iota)\in A$ since
\begin{align*}
&\frac{f(\theta,\iota)(y+z_1+z_2)(z_1+z_2)}{f(\theta,\iota)(y+z_1)(z_1)
f(\theta,\iota)(y+z_2)(z_2)}\\
&=\frac{\theta(z_1+z_2) \iota((y+z_1+z_2)\wedge (z_1+z_2))}{ \theta(z_1) \iota((y+z_1)\wedge z_1) \theta(z_2) \iota((y+z_2)\wedge z_2) }\\
&=\iota \Big( (y+z_1+z_2)\wedge (z_1+z_2) - (y+z_1)\wedge z_1- (y+z_2)\wedge z_2 \Big)\\
&=\iota \Big(z_1\wedge z_2+z_2\wedge z_1 \Big)
=\iota(0)=1.
\end{align*}

Next, we will show that the map $f$ is injective. Suppose $f(\theta_1,\iota_1)=f(\theta_2,\iota_2)$. Then
\[
\theta_1(y)
= f(\theta_1,\iota_1)(0)(y) 
=f(\theta_2,\iota_2)(0)(y) 
=\theta_2(y).
\]
This means that $\theta_1=\theta_2$. Therefore it follows that $\iota_1=\iota_2$.

It remains to show that the map $f: H^*\times (\Delta^2(H))^* \to A$ is surjective. Consider $\alpha_1\in A$.
Take $\theta=\alpha_1(0)\in H^*$ and define $\lambda\in (\Delta^2(H))^*$ by 
\[
\lambda(x\wedge y)
=\alpha_1(x)(y)
\alpha_1(0)(-y).
\]
In order for this to be a valid definition, we need to check that 
$\lambda((x_1+x_2)\wedge y)=\lambda(x_1\wedge y)\lambda(x_2\wedge y)$, 
$\lambda(x\wedge (y_1+y_2))=\lambda(x\wedge y_1)\lambda(x\wedge y_2)$
and $\lambda(x\wedge y)\lambda(y\wedge x)=1$.

First, we check that 
$\lambda(x\wedge (y_1+y_2))
=\lambda(x\wedge y_1)
\lambda(x\wedge y_2)$.
\begin{align*}
&\frac{\lambda(x\wedge (y_1+y_2))}{\lambda(x\wedge y_1)\lambda(x\wedge y_2)}\\
&=\alpha_1(x)(y_1+y_2) 
\alpha_1(0)(-y_1-y_2) 
\times \alpha_1(x)(-y_1) 
\alpha_1(0)(y_1) 
\times \alpha_1(x)(-y_2) \alpha_1(0)(y_2)\\
&=1. 
\end{align*}

Next, we check that $\lambda(x\wedge y)\lambda(y\wedge x)=1$.
Note that
\begin{align*}
\lambda(x\wedge y)\lambda(y\wedge x)
&=\alpha_1(x)(y)
\alpha_1(0)(-y) 
\times \alpha_1(y)(x)
\alpha_1(0)(-x) \\
&=\alpha_1(x)(y) 
\alpha_1(y)(x)
\alpha_1(0)(-x-y).
\end{align*}
Now \eqref{Property 0}, tells us that
\begin{align*}
\alpha_1(0)(-x-y)
&=\alpha_1((x+y)-x-y)(-x-y)\\
&= \alpha_1((x+y)-x)(-x)
\alpha_1((x+y)-y)(-y)\\
&= \alpha_1(y)(-x)
\alpha_1(x)(-y).
\end{align*}
It follows that $\lambda(x\wedge y)\lambda(y\wedge x)=1$.

Finally, we show that $\lambda((x_1+x_2)\wedge y)=\lambda(x_1\wedge y)\lambda(x_2\wedge y)$.
Notice that
\[
\lambda((x_1+x_2)\wedge y)
=\frac{1}{\lambda(y \wedge(x_1+x_2))}
=\frac{1}{\lambda(y \wedge x_1) \lambda(y \wedge x_2)}
=\lambda(x_1\wedge y)\lambda(x_2\wedge y).
\]
We have thus shown that $\lambda$ is well defined. This shows that $f$ is surjective and hence bijective.
\end{proof}

\section{Non special $(C_1,C_2)$}\label{Sec: Non special}

In this section we consider the case when $F_1$ and $F_2$ are codes and $(C_1,C_2)$ is not $(F_1,F_2,\vec{s})$-special. Our goal is to show that at least linearly many coefficients $E'[F_1,F_2,C_1,C_2,\vec{s}](i,j)$ are $\neq1$.

\begin{lemma}\label{Lem: Non special lin many non one}
Suppose $F_1\in \Hom([n_1],G)$ is a translated code with distance $\delta n_1$ and image $T_1$, $F_2\in \Hom([n_2],G)$ is a translated code with distance $\delta n_2$ and image $T_2$ and $\vec{s}\in (\Z/p\Z)^{n_1}$.
For $C_1\in \CC^*_{n_1}(\sgo{F_2},F_1)$, $C_2\in \CC^*_{n_2}(\sgo{F_1},F_2)$, if $(C_1,C_2)$ is not $(F_1,F_2,\vec{s})$-special, then there are at least $\delta n_2$ pairs $(i,j)\in [n_1]\times[n_2]$ for which $E'[F_1,F_2,C_1,C_2,\vec{s}]\neq 1$.
\end{lemma}
\begin{proof}
Denote $X_2=\set{F_2}$, so $\sgo{X_2}=T_2$. Assume for the sake of contradiction that there are some $C_1\in \CC^*_{n_1}(\sgo{F_2},F_1)$, $C_2\in \CC^*_{n_2}(\sgo{F_1},F_2)$ for which
\[
1\leq
\abs{\{(i,j)\in [n_1]\times[n_2]: E'[F_1,F_2,C_1,C_2,\vec{s}](i,j)\neq 1\}}
< \delta n_2.
\]
Denote
\[
\sigma
=\{i\in [n_1] : \exists j\in [n_2], E'[F_1,F_2,C_1,C_2,\vec{s}](i,j)\neq 1\},
\]
so $1\leq \abs{\sigma} <\delta n_2\leq \delta n_1$. 
Let $Y_1=\{(F_1(i),s_i): i\in [n_1]\setminus \sigma\}$ and $X_1=\pi_1(Y_1)$. Since $F_1$ is a translated code of distance $\delta n_1$ with image $T_1$, it follows that $\sgo{X_1}=T_1$.

Denote $\tau=[n_1]\setminus \sigma$, let $F_{\tau}$ be the restriction of $F_1$ to $\tau$ and $F_{\sigma}$ be the restriction of $F_1$ to $\sigma$. Also let $C_{\tau}$ be the restriction of $C_1$ to $\tau$ and $C_{\sigma}$ be the restriction of $C_1$ to $\sigma$.
If we restrict ourselves to $\tau\times[n_2]$ then $(C_\tau,C_2)$ is $(F_{\tau},F_2,\vec{s})$-special on these indices. Therefore, by Lemma~\ref{Lem: special iff factor through} there is $(\beta_1,\beta_2)\in \D(Y_1,X_2)$ such that $C_{\tau}(i)=\beta_1(F_{\tau}(i),s_i)$ and $C_2(j)=\beta_2(F_2(j))$.

Fix some $i_0\in\sigma$. Consider the set
\[
A=
\big\{h\in X_2: C_{\sigma}(i_0)(-F_{\sigma}(i_0)+h) \beta_2(h)(-h+F_{\sigma}(i_0))\zeta_p^{s_{i_0}}=1 \big\}.
\]
By the definition of $\sigma$, we know that $A\subsetneq X_2$.

We know that $F_{\sigma}(i_0)\in T_1=\sgo{X_1}$. Say $F_{\sigma}(i_0)=\sum_{t} a_t g_t$ for $g_t\in X_1$ and $\sum a_t=1$. Recall that $X_1=\pi_1(Y_1)$, so for each $t$, pick some $(g_t,s_t)\in Y_1$.

We claim that $\sgo{A}\cap X_2=A$. Assume for the sake of contradiction that there is some $h\in \sgo{A}\cap X_2$, which is not in $A$.
Since $h\in\sgo{A}$, $h=\sum_k b_k h_k$ for $h_k\in A$ and $\sum b_k=1$.
Now,
\begin{align*}
C_{\sigma}(i_0)(-F_{\sigma}(i_0)+h)
&=C_{\sigma}(i_0)\left(\sum_{k} b_k(-F_{\sigma}(i_0)+h_k)\right)\\
&=\prod_k C_{\sigma}(i_0)(-F_{\sigma}(i_0)+h_k)^{b_k}\\
&=\prod_k \Big( \beta_2(h_k)(-h_k+F_{\sigma}(i_0)) \zeta_p^{s_{i_0}}\Big)^{-b_k}\\
&=\zeta_p^{-s_{i_0}}\prod_k \left( \beta_2(h_k)\left(\sum_t a_t( -h_k+g_t)\right) \right)^{-b_k}\\
&= \zeta_p^{-s_{i_0}}
\prod_k 
\prod_t
\Big( \beta_2(h_k)( -h_k+g_t) \Big)^{-b_ka_t}\\
&= \zeta_p^{-s_{i_0}}
\prod_k 
\prod_t
\Big( \beta_1(g_t,s_t)(-g_t+h_k) \zeta_p^{s_t}  \Big)^{b_ka_t}\\
&= \zeta_p^{-s_{i_0}}
\prod_t
\left( \beta_1(g_t,s_t)\left( \sum_k b_k (-g_t+h_k)\right) \zeta_p^{s_t\sum_k b_k}  \right)^{a_t}\\
&= \zeta_p^{-s_{i_0}}
\prod_t
\left( \beta_1(g_t,s_t) (-g_t+h) \zeta_p^{s_t}  \right)^{a_t}\\
&= \zeta_p^{-s_{i_0}}
\prod_t
\left( \beta_2(h)(-h+g_t)  \right)^{-a_t}\\
&= \zeta_p^{-s_{i_0}}
\left( \beta_2(h) \left( \sum_{t} a_t(-h+g_t) \right) \right)^{-1}\\
&= \zeta_p^{-s_{i_0}}
\left( \beta_2(h) (-h+F_{\sigma}(i_0))  \right)^{-1}.
\end{align*}
This contradicts the fact that $h\notin A$.
We conclude that $\sgo{A}\cap X_2=A$.
Since $A\subsetneq X_2$, this means that $X_2\not\subseteq \sgo{A}$ and hence $\sgo{A}\neq T_2$.

Now consider
\[
\sigma_2
=\{j\in [n_2]: F_2(j)\in X_2\setminus A\}.
\]
Notice that $\sgo{F_2(j):j\in [n_2]\setminus\sigma_2} =\sgo{A}\neq T_2$. Since $F_2$ is a translated code of distance $\delta n_2$ with image $T_2$, this implies $\abs{\sigma_2}\geq \delta n_2$.
For every $j\in \sigma_2$, $E'[F_1,F_2,C_1,C_2,\vec{s}](i_0,j)\neq 1$. This contradicts the choice of $(C_1,C_2)$.
\end{proof}

\section{Computing moments}

In this section, we will prove Proposition~\ref{Prop: moments on parts}, which will complete the proof of Conjecture~\ref{conj: p-Syl_conj} for odd $p$.

\subsection{Estimation for codes}

We start with the case where $F_1$ and $F_2$ are both codes. We combine the results from Section~\ref{Sec: Robust}, Section~\ref{Sec: Special} and Section~\ref{Sec: Non special} to bound the probability
\[
\PPP[(F_1,F_2) \circ \Lnot=(F_1',F_2'), \eta(\Lnot)=\vec{t}].
\]

\begin{proposition}\label{Prop: Individual code bound}
Let $p$ be a prime, $0<\alpha \leq 1$, $\alpha n_1\leq n_2 \leq n_1$, $G$ be a finite abelian $p$-group and $\delta>0$.
Then there is $K>0$, such that for sufficiently large $n_1$,
for every pair of translated subgroups $T_1,T_2\treq G$. 
Suppose 
\begin{itemize}
\item $F_1\in \Hom([n_1],G)$ is a translated code of distance $\delta n_1$ with image $T_1$;
\item $F_2\in \Hom([n_2],G)$ is a translated code of distance $\delta n_2$ with image $T_2$;
\item $\vec{t}\in (\Z/p\Z)^{n_1}$, $F_1'\in \CC_{n_1}(T_2,F_1)$ and $F_2'\in\CC_{n_2}(T_1,F_2)$;
\item For every $i\in [n_1]$ if $F_1(i)\notin T_2$, then $\chi_{T_2,F_1(i)}(F_1'(i))=\zeta_p^{t_i}$.
\end{itemize}
Then we have
\[
\PPP[(F_1,F_2) \circ \Lnot=(F_1',F_2'), \eta(\Lnot)=\vec{t}]
\leq 
\frac{K}{\abs{\CC^*_{n_1}(T_2,F_1)} \abs{\CC^*_{n_2}(T_1,F_2)} p^{n_1}}
p^{\abs{\{i\in [n_1]: F_1(i)\notin T_1\cap T_2\}}}.
\]
On the other hand, if there is some $i\in [n_1]$ for which $F_1(i)\notin T_2$ and $\chi_{T_2,F_1(i)}(F_1'(i))\neq \zeta_p^{t_i}$, then
\[
\PPP[(F_1,F_2) \circ \Lnot=(F_1',F_2'), \eta(\Lnot)=\vec{t}]
=0.
\]
\end{proposition}
\begin{proof}
Lemma~\ref{Lem: ti right value} tells us that if there is some $i\in [n_1]$ for which $F_1(i)\notin T_2$ and $\chi_{T_2,F_1(i)}(F_1'(i))\neq \zeta_p^{t_i}$, then
\[
\PPP[(F_1,F_2) \circ \Lnot=(F_1',F_2'), \eta(\Lnot)=\vec{t}]
=0.
\]

Now assume that $F_1(i)\notin T_2\implies \chi_{T_2,F_1(i)}(F_1'(i))=\zeta_p^{t_i}$.
Equation \eqref{Eqn: prob fourier sum} implies that
\begin{equation*}
\begin{split}
&\PPP[(F_1,F_2)\circ \Lnot=(F_1',F_2') \text{ and } \eta(\Lnot)=\vec{t}]\\
&\leq 
\frac{1}{\abs{\CC^*_{n_1}(\sgo{F_2},F_1)}}
\frac{1}{\abs{\CC^*_{n_2}(\sgo{F_1},F_2)}}
\frac{1}{p^{n_1}}
\sum_{C_1\in \CC^*_{n_1}(\sgo{F_2},F_1)}
\sum_{C_2\in\CC^*_{n_2}(\sgo{F_1},F_2)}
\sum_{\vec{s}\in (\Z/p\Z)^{n_1}}\\
&\quad\quad\quad\quad
%\zeta_p^{-\vec{s}\cdot\vec{t}}
%\prod_{i=1}^{n_1} C_1(i)(-F_1'(i))
%\prod_{j=1}^{n_2} C_2(j)(-F_2'(j))
\prod_{i=1}^{n_1} \prod_{j=1}^{n_2}
\left|
\E
\left[
\Big(E'[F_1,F_2,C_1,C_2,\vec{s}](i,j) \Big)^{a_{ij}}
\right]
\right|.
\end{split}
\end{equation*}

First, we compute the contribution of $(C_1, C_2, \vec{s})$ for which $(C_1,C_2)$ is $(F_1,F_2,\vec{s})$-special.
By Proposition~\ref{Prop: Count special}, the number of triples $(C_1, C_2, \vec{s})$ for which $(C_1,C_2)$ is $(F_1,F_2,\vec{s})$-special is at most
\[
p^{\abs{G}}
\times \abs{G}^{(p+1)\abs{G}}
\times p^{\abs{\{i\in [n_1]: F_1(i)\notin T_2\}}}.
\]
Thus, their combined contribution is at most
\[
\frac{p^{\abs{G}} 
\times \abs{G}^{(p+1)\abs{G}}}{\abs{\CC^*_{n_1}(T_2,F_1)} \abs{\CC^*_{n_2}(T_1,F_2)} p^{\abs{\{i\in [n_1]: F_1(i)\in T_2\}}}}.
\]

Next, we compute the contribution of triples $(C_1, C_2, \vec{s})$ for which $(C_1,C_2)$ is not $(F_1,F_2,\vec{s})$-special. We divide these into several cases based on a constant $\gamma>0$ that we will pick later.
\begin{enumerate}
\item We compute the contribution of $(C_1, C_2, \vec{s})$ for which $C_1$ is $(\gamma n_1,F_1,\vec{s})$-robust.
Lemma~\ref{Lem: quad many non one} tells us that for such $(\vec{s},C_1,C_2)$, at least $\frac{\delta\gamma \alpha n_1^2}{\abs{G}^2}$ of the coefficients $E'[F_1,F_2,C_1,C_2,\vec{s}](i,j)$ are not $1$. By Lemma~\ref{Lem: Wood bound xi^ epsilon balanced}, for each such coefficient
\[
\left|\E
\left[
\Big(E'[F_1,F_2,C_1,C_2,\vec{s}](i,j) \Big)^{a_{ij}}
\right] \right|
\leq \exp{\left(-\frac{\epsilon}{\abs{G}^2}\right)}.
\]
The number of $(\vec{s},C_1,C_2)$ is at most $p^{n_1}\abs{\CC^*_{n_1}(\sgo{F_2},F_1)} \abs{\CC^*_{n_2}(\sgo{F_1},F_2)}$.
Therefore, the combined contribution of such triples is at most
\[
\frac{\abs{\CC^*_{n_1}(\sgo{F_2},F_1)} \abs{\CC^*_{n_2}(\sgo{F_1},F_2)} p^{n_1}}{\abs{\CC^*_{n_1}(\sgo{F_2},F_1)} \abs{\CC^*_{n_2}(\sgo{F_1},F_2)} p^{n_1}}
\left(\exp{\Big(-\frac{\epsilon}{\abs{G}^2}\Big)}\right)^{\frac{\delta\gamma \alpha n_1^2}{\abs{G}^2}}
= \exp{\left(-\frac{\epsilon \delta\gamma \alpha n_1^2}{\abs{G}^4}\right)}.
\]

\item Next, we compute the contribution of $(\vec{s},C_1,C_2)$ for which $C_2$ is $(\gamma n_2, F_2)$-robust.
Lemma~\ref{Lem: quad many non one} tells us that for such $(\vec{s},C_1,C_2)$, at least $\frac{\delta\gamma \alpha n_1^2}{p\abs{G}^2}$ of the coefficients $E'[F_1,F_2,C_1,C_2,\vec{s}](i,j)$ are not $1$.
Similarly to above, the combined contribution of such triples is at most
\[
\exp{\left(-\frac{\epsilon \delta\gamma \alpha n_1^2}{p\abs{G}^4}\right)}.
\]

\item Next, we compute the contribution of $(\vec{s},C_1,C_2)$ for which $\vec{s}$ is $(\gamma n_1, F_1, T_2)$-robust.
Lemma~\ref{Lem: quad many non one} tells us that for such $(\vec{s},C_1,C_2)$, at least $\frac{\delta\gamma \alpha n_1^2}{\abs{G}^2}$ of the coefficients $E'[F_1,F_2,C_1,C_2,\vec{s}](i,j)$ are not $1$.
Similarly to above, the combined contribution of such triples is at most
\[
\exp{\left(-\frac{\epsilon \delta\gamma \alpha n_1^2}{\abs{G}^4}\right)}.
\]

\item Finally, we compute the contribution of $(\vec{s},C_1,C_2)$ for which $(C_1,C_2)$ is not $(F_1,F_2,\vec{s})$-special, $C_1$ is not $(\gamma n_1, F_1,\vec{s})$-robust, $C_2$ is not $(\gamma n_2, F_2)$-robust and $\vec{s}$ is not $(\gamma n_1, F_1,T_2)$-robust.
Lemma~\ref{Lem: Non special lin many non one} tells us that in this case there will be at least $\delta \alpha n_1$ many coefficients that are not $1$.

Lemma~\ref{Lem: count non robust vec s} tells us that the number of $\vec{s}$ that are not $(\gamma n_1,F_1,T_2)$-robust is at most $p^{\abs{G}} (4p)^{n_1\sqrt{\gamma}} p^{\abs{\{i\in [n_1]:F_1(i)\notin T_2\}}}$.
Lemma~\ref{Lem: count non robust with s} tells us that the number of $C_1\in \CC^*_{n_1}(T,F_1)$ that are not $(\gamma n_1,F_1,\vec{s})$-robust is at most $\abs{G}^{p\abs{G}} \left(4\abs{G}\right)^{n_1\sqrt{\gamma}}$.
Lemma~\ref{Lem: count non robust no s} tells us that the number of $C_2\in \CC^*_{n_2}(T,F_2)$ that are not $(\gamma n_2,F_2)$-robust is at most $\abs{G}^{\abs{G}} \left(4\abs{G}\right)^{n_2\sqrt{\gamma}}$.
Therefore, the combined contribution of such triples is at most
\[
\frac{1}{\abs{\CC^*_{n_1}(T_2,F_1)} \abs{\CC^*_{n_2}(T_1,F_2)}}
\frac{1}{p^{\abs{\{i\in [n_1]:F_1(i)\in T_2\}}}}
\Big(
p^{\abs{G}} 
\abs{G}^{(p+1)\abs{G}}  
\Big)
\left(
\Big( 64p\abs{G}^2 \Big)^{\sqrt{\gamma}}
\exp\Big(-\frac{\epsilon\delta\alpha }{\abs{G}^2}\Big)
\right)^{n_1}.
\]
\end{enumerate}
We choose a sufficiently small $\gamma$ to ensure that
\[
\Big( 64p\abs{G}^2 \Big)^{\sqrt{\gamma}}
\exp\Big(-\frac{\epsilon\delta\alpha }{\abs{G}^2}\Big)
<1.
\]
We fix such a $\gamma$ and denote $c=-\log\left(\Big( 64p\abs{G}^2 \Big)^{\sqrt{\gamma}}
\exp\Big(-\frac{\epsilon\delta\alpha }{\abs{G}^2}\Big)\right) >0$.
For sufficiently large values of $n_1$, we will have
\[
p \abs{G}^{2}
\exp{\left(-\frac{\epsilon \delta\gamma \alpha n_1}{p\abs{G}^4}\right)}
<\exp(-c).
\]
Therefore, for sufficiently large values of $n_1$, the combined contribution of triples $(C_1,C_2,\vec{s})$ coming from $(1)$, $(2)$, $(3)$ and $(4)$ is at most
\[
\frac{1}{\abs{\CC^*_{n_1}(T_2,F_1)} \abs{\CC^*_{n_2}(T_1,F_2)}}
\frac{1}{p^{\abs{\{i\in [n_1]:F_1(i)\in T_2\}}}}
\Big(
3+
p^{\abs{G}} 
\abs{G}^{(p+1)\abs{G}}  
\Big)
\exp({-c n_1}).
\]
We conclude that
\begin{align*}
&\PPP[(F_1,F_2)\circ \Lnot=(F_1',F_2') \text{ and } \eta(\Lnot)=\vec{t}]\\
&\leq
\frac{1}{\abs{\CC^*_{n_1}(T_2,F_1)} \abs{\CC^*_{n_2}(T_1,F_2)}}
\frac{1}{p^{\abs{\{i\in [n_1]:F_1(i)\in T_2\}}}}
\left(
p^{\abs{G}} \abs{G}^{(p+1)\abs{G}}
+\big(
3+
p^{\abs{G}} 
\abs{G}^{(p+1)\abs{G}}  
\big)
\exp({-c n_1})
\right).
\end{align*}
The first part of the result follows by choosing  $K=
3+
2\times p^{\abs{G}} 
\abs{G}^{(p+1)\abs{G}}$.
\end{proof}

\subsection{Main term}

In this subsection, we fix a translated subgroup $T\treq G$ and consider $F_1$ and $F_2$ that are translated codes with the same image $T$.
Note that $\sgo{F_1}=\sgo{F_2}=T$ implies
%$\CC_{n_1}(T,F_1)= \Hom([n_1], \G(T))$,
%$\CC_{n_2}(T,F_2)= \Hom([n_2], \G(T))$, 
$\CC^*_{n_1}(T,F_1)= \Hom([n_1], \G(T)^*)$ and $\CC^*_{n_2}(T,F_2)= \Hom([n_2], \G(T)^*)$.
% Therefore, \eqref{Eqn: prob fourier sum} simplifies to, for any $F_1\in \SSS{n_1}{T}{T}$, $F_2\in \SSS{n_2}{T}{T}$, $F_1'\in \Hom([n_1],\G(T))$, $F_2'\in \Hom([n_2],\G(T))$ and $\vec{t}\in(\Z/p\Z)^{n_1}$, we have
% \begin{equation}\label{Eqn: Fourier T1=T2}
% \begin{split}
% &\PPP[(F_1,F_2)\circ \Lnot=(F_1',F_2') \text{ and } \eta(\Lnot)=\vec{t}]\\
% &=\frac{1}{\abs{T}^{n_1+n_2}}
% \frac{1}{p^{n_1}}
% \sum_{C_1\in \Hom([n_1],\G(T)^*)}
% \sum_{C_2\in \Hom([n_2],\G(T)^*)}
% \sum_{\vec{s}\in (\Z/p\Z)^{n_1}}\\
% &\quad\quad\quad\quad
% \zeta_p^{-\vec{s}\cdot\vec{t}}
% \prod_{i=1}^{n_1} C_1(i)(-F_1'(i))
% \prod_{j=1}^{n_2} C_2(j)(-F_2'(j))
% \prod_{i=1}^{n_1} \prod_{j=1}^{n_2}
% \E
% \left[
% \Big(E'[F_1,F_2,C_1,C_2,\vec{s}](i,j) \Big)^{a_{ij}}
% \right].
% \end{split}
% \end{equation}

\begin{proposition}\label{Prop: Individual code bound same image}
Let $p$ be a prime, $0<\alpha \leq 1$, $\alpha n_1\leq n_2 \leq n_1$, $G$ be a finite abelian $p$-group and $\delta>0$.
Then there are $c>0$ and $K_1>0$, such that for sufficiently large $n_1$,
for every translated subgroup $T\treq G$, 
if $F_1\in \Hom([n_1],G)$ is a translated code of distance $\delta n_1$ with image $T$, $F_2\in \Hom([n_2],G)$ is a translated code of distance $\delta n_2$ with image $T$ such that $\set{F_1}=\set{F_2}=T$ and $\vec{t}\in (\Z/p\Z)^{n_1}$, then we have
\[
\abs{
\PPP[(F_1,F_2) \circ \Lnot=0, \eta(\Lnot)=\vec{t}]
-\frac{\abs{\Delta^2(\G(T))}}{\abs{T}^{n_1+n_2-1}p^{n_1}}
}
\leq \frac{K_1 \exp(-c n_1)}{\abs{T}^{n_1+n_2} p^{n_1}}.
\]
\end{proposition}
\begin{proof}
The proof of Proposition~\ref{Prop: Individual code bound} tells us that
\begin{align*}
&\PPP[(F_1,F_2) \circ \Lnot=0, \eta(\Lnot)=\vec{t}]\\
&= \frac{1}{\abs{T}^{n_1+n_2}}
\frac{1}{p^{n_1}}
\sum_{\substack{C_1\in \CC^*_{n_1}(T,F_1) \\ C_2\in\CC^*_{n_2}(t,F_2) \\ \vec{s}\in (\Z/p\Z)^{n_1}\\ (C_1,C_2)\text{ is }(F_1,F_2,\vec{s})-\text{special} }}
\zeta_p^{-\vec{s}\cdot\vec{t}}
+\text{error term},
\end{align*}
with
\[
\abs{\text{error term}}
\leq \frac{1}{\abs{T}^{n_1+n_2}}
\frac{1}{p^{n_1}}
\Big(
3+
p^{\abs{G}} 
\abs{G}^{(p+1)\abs{G}}  
\Big)
\exp({-c n_1}).
\]

We are given that $\set{F_1}=\set{F_2}=T$.
We compute the contribution of $(C_1,C_2, \vec{s})$ for which $(C_1,C_2)$ is $(F_1,F_2,\vec{s})$-special.
Proposition~\ref{Prop: Count special} tells us that if $\vec{s}\neq0$, then there are no pairs $(C_1,C_2)$ that are $(F_1,F_2,\vec{s})$-special.
On the other hand, for $\vec{s}=0$, the number of pairs $(C_1,C_2)$ that are $(F_1,F_2,\vec{s})$-special is $\abs{\Delta^2\G(T)}\times \abs{T}$.
The result follows.
\end{proof}

\begin{corollary}\label{Cor: main term}
Let $p$ be a prime, $0<\alpha \leq 1$, $\rho, \delta>0$ and $G$ be a finite abelian $p$-group.
Consider a translated subgroup $T\treq G$.
We have
\[
\lim_{n\to\infty}
\sum_{F_1\in\SSS{n,\delta}{T}{T}} \sum_{F_2\in\SSS{\ceil{\alpha n},\delta}{T}{T}} \PPP[(F_1,F_2) \circ \Lna=0, \eta(\Lna)\in \EE_{n,\rho}]
=\abs{\Delta^2 \G(T)} \times \abs{\G(T)}.
\]
\end{corollary}
\begin{proof}
Note that
\begin{align*}
&\left| 
\left(
\sum_{F_1\in\SSS{n,\delta}{T}{T}} 
\sum_{F_2\in\SSS{\ceil{\alpha n},\delta}{T}{T}}
\PPP[(F_1,F_2) \circ \Lna=0, \eta(\Lna)\in \EE_{n,\rho}]
\right) -
\abs{\Delta^2 \G(T)} \times \abs{\G(T)}
\right|\\
&\leq 
\sum_{F_1\in\SSS{n,\delta}{T}{T}} 
\sum_{F_2\in\SSS{\ceil{\alpha n},\delta}{T}{T}}
\sum_{\vec{t}\in \EE_{n,\rho}}
\left|
\PPP[(F_1,F_2) \circ \Lna=0, \eta(\Lna)=\vec{t}]
- \frac{\abs{\Delta^2 \G(T)}}{\abs{T}^{n+\ceil{\alpha n}-1}p^{n}}
\right|\\
&\quad\quad+
\left|
\abs{\SSS{n,\delta}{T}{T}}\times 
\abs{\SSS{\ceil{\alpha n},\delta}{T}{T}} \times
\abs{\EE_{n,\rho}}
\frac{\abs{\Delta^2 \G(T)}}{\abs{T}^{n+\ceil{\alpha n}-1}p^{n_1}}
- \abs{\Delta^2 \G(T)} \times \abs{\G(T)}
\right|.
\end{align*}

For the second term note that
\begin{align*}
&\left|
\abs{\SSS{n,\delta}{T}{T}}\times 
\abs{\SSS{\ceil{\alpha n},\delta}{T}{T}} \times
\abs{\EE_{n,\rho}}
\frac{\abs{\Delta^2 \G(T)}}{\abs{T}^{n_1+n_2-1}p^{n_1}}
- \abs{\Delta^2 \G(T)} \times \abs{\G(T)}
\right|\\
&= \abs{\Delta^2 \G(T)} \times \abs{\G(T)}
\left|
\frac{\abs{\SSS{n,\delta}{T}{T}}\times 
\abs{\SSS{\ceil{\alpha n},\delta}{T}{T}} \times
\abs{\EE_{n,\rho}}}{\abs{T}^{n+\ceil{\alpha n} }p^n}
- 1
\right|.
\end{align*}
By Lemma~\ref{Lem: count codes} and Lemma~\ref{Lem: size of E rho}, this goes to $0$ as $n\to\infty$.

Next, Proposition~\ref{Prop: Individual code bound same image} tells us that
\begin{align*}
&\sum_{\substack{F_1\in\SSS{n,\delta}{T}{T} \\ \set{F_1}=T} } 
\sum_{\substack{F_2\in\SSS{\ceil{\alpha n},\delta}{T}{T} \\\set{F_2}=T } }
\sum_{\vec{t}\in \EE_{n,\rho}}
\left|
\PPP[(F_1,F_2) \circ \Lna=0, \eta(\Lna)=\vec{t}]
- \frac{\abs{\Delta^2 \G(T)}}{\abs{T}^{n+\ceil{\alpha n}-1}p^{n}}
\right|\\
&\leq \abs{\SSS{n,\delta}{T}{T}}
\times \abs{\SSS{\ceil{\alpha n},\delta}{T}{T}}
\times \abs{\EE_{n,\rho}}
\times \frac{K_1 \exp(-c(n+\ceil{\alpha n}))}{\abs{T}^{n+\ceil{\alpha n}} p^n}\\ 
&\leq K_1 \exp(-c(n)).
\end{align*}

Finally, consider subsets $X_1, X_2\subseteq T$. Proposition~\ref{Prop: Individual code bound} tells us that
\begin{align*}
&\sum_{\substack{F_1\in\SSS{n,\delta}{T}{T} \\ \set{F_1}=X_1} } 
\sum_{\substack{F_2\in\SSS{\ceil{\alpha n},\delta}{T}{T} \\\set{F_2}=X_2 } }
\sum_{\vec{t}\in \EE_{n,\rho}}
\left|
\PPP[(F_1,F_2) \circ \Lna=0, \eta(\Lna)=\vec{t}]
- \frac{\abs{\Delta^2 \G(T)}}{\abs{T}^{n+\ceil{\alpha n}-1}p^{n}}
\right|\\
&\leq \sum_{\substack{F_1\in\SSS{n,\delta}{T}{T} \\ \set{F_1}=X_1} } 
\sum_{\substack{F_2\in\SSS{\ceil{\alpha n},\delta}{T}{T} \\\set{F_2}=X_2 } }
\sum_{\vec{t}\in \EE_{n,\rho}}
\left(
\PPP[(F_1,F_2) \circ \Lna=0, \eta(\Lna)=\vec{t}]
+\frac{\abs{\Delta^2 \G(T)}}{\abs{T}^{n+\ceil{\alpha n}-1}p^{n}}
\right)\\
&\leq \sum_{\substack{F_1\in\SSS{n,\delta}{T}{T} \\ \set{F_1}=X_1} } 
\sum_{\substack{F_2\in\SSS{\ceil{\alpha n},\delta}{T}{T} \\\set{F_2}=X_2 } }
\sum_{\vec{t}\in \EE_{n,\rho}}
\left(
\frac{K}{\abs{T}^{n+\ceil{\alpha n}}p^{n}}
+\frac{\abs{\Delta^2 \G(T)}}{\abs{T}^{n+\ceil{\alpha n}-1}p^{n}}
\right)\\
&\leq \left(
\frac{K}{\abs{T}^{n+\ceil{\alpha n}}p^{n}}
+\frac{\abs{\Delta^2 \G(T)}}{\abs{T}^{n+\ceil{\alpha n}-1}p^{n}}
\right)
\times \abs{X_1}^{n}
\times \abs{X_2}^{\ceil{\alpha n}}
\times p^n\\
&= \left(
K
+\abs{\Delta^2 \G(T)} \times \abs{T}
\right)
\times \left( \frac{\abs{X_1}}{\abs{T}} \right)^{n}
\times \left(\frac{\abs{X_2}}{\abs{T}}\right)^{\ceil{\alpha n}}.
\end{align*}
If $X_1\subsetneq T$ or $X_2\subsetneq T$, then this limit will be $0$. The result follows.
\end{proof}

\subsection{Error term, Type 1}

In this subsection, we fix chains of translated subgroups $T_1\treq S_1\treq G$ and $T_2\treq S_2\treq G$ such that $T_1$, $T_2$, $S_1$ and $S_2$ are not all the same and $\abs{T_1}\leq \abs{T_2}$.

\begin{lemma}\label{Lem: bound prob withoit t}
Let $p$ be a prime, $0<\alpha \leq 1$, $\alpha n_1\leq n_2 \leq n_1$, $G$ be a finite abelian $p$-group and $\delta>0$.
Then is a constant $K>0$, such that for sufficiently large $n_1$, if
\begin{enumerate}
\item $F_1\in \Hom([n_1],G)$ is a translated code of distance $\delta n_1$ with image $T_1$;
\item $F_2\in \Hom([n_2],G)$ is a translated code of distance $\delta n_2$ with image $T_2$;
\item $F_1'\in \CC_{n_1}(T_2,F_1)$ and $F_2'\in\CC_{n_2}(T_1,F_2)$,
\end{enumerate}
then we have
\[
\PPP[(F_1,F_2) \circ \Lnot=(F_1',F_2')]
\leq 
\frac{K}{\abs{\CC_{n_1}(T_2,F_1)}\times \abs{\CC_{n_2}(T_1,F_2)}}.
\]
\end{lemma}
\begin{proof}
Note that
\[
\PPP[(F_1,F_2) \circ \Lnot=(F_1',F_2')]
=\sum_{\vec{t}\in(\Z/p\Z)^{n_1}}
\PPP[(F_1,F_2) \circ \Lnot=(F_1',F_2'), \eta(\Lnot)=\vec{t}].
\]
Let $A_1=\{i\in [n_1]: F_1(i)\in T_2\}$ and $A_2=\{i\in [n_1]: F_1(i)\in T_1\setminus T_2\}$.
Proposition~\ref{Prop: Individual code bound} tells us that if there is some $i\in A_2$ for which $\chi_{T_2,F_1(i)}(F_1'(i))\neq \zeta_p^{t_i}$, then
\[
\PPP[(F_1,F_2) \circ \Lnot=(F_1',F_2'), \eta(\Lnot)=\vec{t}]
=0.
\]
Let $B$ be the set of vectors in $(\Z/p\Z)^{n_1}$ for which each $i\in A_2$, we have $\chi_{T_2,F_1(i)}(F_1'(i))= \zeta_p^{t_i}$. Clearly $\abs{B}=p^{\abs{A_1}}$.

Moreover, Proposition~\ref{Prop: Individual code bound} tells us that for each $\vec{t}\in B$,
\[
\PPP[(F_1,F_2) \circ \Lnot=(F_1',F_2'), \eta(\Lnot)=\vec{t}]
\leq
\frac{K}{\abs{\CC^*_{n_1}(T_2,F_1)} \abs{\CC^*_{n_2}(T_1,F_2)} p^{A_1}}.
\]
The result follows.
\end{proof}

\begin{proposition}\label{Prop: Error term type 1 to zero}
Let $p$ be a prime, $0<\alpha \leq 1$, $\rho>0$, $G$ be a finite abelian $p$-group.
Suppose we have chains of translated subgroups $T_1\treq S_1\treq G$ and $T_2\treq S_2\treq G$.
If
\begin{enumerate}
\item $T_1$, $T_2$, $S_1$, $S_2$ are not all the same;
\item $\abs{T_1}\leq \abs{T_2}$,
\end{enumerate}
then for sufficiently small $\delta>0$, we have
\[
\lim_{n\to\infty}
\sum_{F_1\in\SSS{n,\delta}{T_1}{S_1}} \sum_{F_2\in\SSS{\ceil{\alpha n},\delta}{T_2}{S_2}} \PPP[(F_1,F_2) \circ \Lna=0, \eta(\Lna)\in \EE_{n,\rho}]
=0.
\]
\end{proposition}
\begin{proof}
We will show that
\[
\lim_{n\to\infty}
\sum_{F_1\in\SSS{n,\delta}{T_1}{S_1}} \sum_{F_2\in\SSS{\ceil{\alpha n},\delta}{T_2}{S_2}} \PPP[(F_1,F_2) \circ \Lna=0]
=0,
\]
this will clearly imply the result.

Denote $D_1=[S_1:T_1]$ and $D_2=[S_2:T_2]$.
Further, denote
\begin{align*}
&e_1
=\begin{cases}
    1-\epsilon &\text{if }D_1>1 \text{ and } T_2\subseteq T_1\\
    1 &\text{otherwise}
\end{cases},
&
&e_2
=\begin{cases}
    1-\epsilon &\text{if }D_2>1 \text{ and } T_1\subseteq T_2\\
    1 &\text{otherwise}
\end{cases}.
\end{align*}

Consider $F_1\in\SSS{n,\delta}{T_1}{S_1}$ and $F_2\in\SSS{\ceil{\alpha n},\delta}{T_2}{S_2}$.
Denote
\begin{align*}
\sigma_1&=\{i\in [n]: F_1(i)\notin T_1\},
&
\tau_1=[n]\setminus \sigma_1,\\
\sigma_2&=\{j\in [\ceil{\alpha n}]: F_2(j)\notin T_2\},
&
\tau_2=[\ceil{\alpha n}]\setminus \sigma_2.
\end{align*}
We know that $\abs{\sigma_1}< \delta \log_p(D_1) n$ and $\abs{\sigma_2}< \delta \log_p(D_2) \ceil{\alpha n}$. Moreover, if $D_1>1$, then $\sigma_1\neq\emptyset$ and if $D_2>1$, then $\sigma_2\neq\emptyset$.
Let $F_{\tau_1}$ be the restriction of $F_1$ to $\tau_1$, $F_{\sigma_1}$ be the restriction of $F_1$ to $\sigma_1$, $F_{\tau_2}$ be the restriction of $F_2$ to $\tau_2$ and $F_{\sigma_2}$ be the restriction of $F_2$ to $\sigma_2$.
We know that $F_{\tau_1}:\tau_1\to T_1$ is a translated code of distance $\delta n$ with image $T_1$ and $F_{\tau_2}:\tau_2\to T_2$ is a translated code of distance $\delta\alpha n$ with image $T_2$.

If $Z_1\subseteq [n]$ and $Z_2\subseteq [\ceil{\alpha n}]$, then we denote:
\begin{itemize}
    \item $A_{Z_1,Z_2}$ is $\abs{Z_1}\times\abs{Z_2}$ matrix whose entries are $a_{ij}$ for $i\in Z_1$ and $j\in Z_2$.
    \item $B_{Z_1,Z_2}$ is $\abs{Z_1}\times \abs{Z_1}$ diagonal matrix, whose diagonal entries are $-\sum_{j\in Z_2} a_{ij}$ for $i\in Z_1$.
    \item $C_{Z_1,Z_2}$ is a $\abs{Z_2}\times \abs{Z_2}$ diagonal matrix, whose diagonal entries are $-\sum_{i\in Z_1} a_{ij}$ for $j\in Z_2$.
\end{itemize}
Therefore, $\Lna$ has the block structure
\[
\Lna=\begin{bmatrix}
    B_{\tau_1,\tau_2}+B_{\tau_1,\sigma_2}
    &0
    &A_{\tau_1,\tau_2}
    & A_{\tau_1,\sigma_2}\\
    0
    & B_{\sigma_1,\tau_2}+B_{\sigma_1,\sigma_2}
    &A_{\sigma_1,\tau_2}
    & A_{\sigma_1,\sigma_2}\\
    A_{\tau_1,\tau_2}^T
    &A_{\sigma_1,\tau_2}^T
    &C_{\tau_1,\tau_2}+ C_{\sigma_1,\tau_2}
    &0\\
    A_{\tau_1,\sigma_2}^T
    & A_{\sigma_1,\sigma_2}^T
    &0
    & C_{\tau_1,\sigma_2}+ C_{\sigma_1,\sigma_2}
\end{bmatrix}.
\]
Denote the restriction of $F_1$ to $\tau_1$ as $F_{\tau_1}$ and the restriction to $\sigma_1$ as $F_{\sigma_1}$.
Similarly denote the restriction of $F_2$ to $\tau_2$ as $F_{\tau_2}$ and the restriction to $\sigma_2$ as $F_{\sigma_2}$.
We construct the matrix $\Lna(\tau_1,\tau_2)$ as
\[
\Lna(\tau_1,\tau_2)
=\begin{bmatrix}
B_{\tau_1,\tau_2}
& A_{\tau_1,\tau_2}\\
A_{\tau_1,\tau_2}^T
& C_{\tau_1,\tau_2}
\end{bmatrix}.
\]
Note that if $(F_1,F_2)\circ \Lna=0$, then
\[
(F_{\tau_1},F_{\tau_2}) \circ (\Lna(\tau_1,\tau_2)) 
=(-F_{\tau_1}\circ B_{\tau_1,\sigma_2}
-F_{\sigma_2}\circ A_{\tau_1,\sigma_2}^T,
-F_{\sigma_1}\circ A_{\sigma_1,\tau_2}
-F_{\tau_2}\circ C_{\sigma_1,\tau_2}).
\]
We denote
\begin{align*}
F_1'
&=-F_{\tau_1}\circ B_{\tau_1,\sigma_2}
-F_{\sigma_2}\circ A_{\tau_1,\sigma_2}^T
\in \Hom(\tau_1,G),\\
F_2'
&=-F_{\sigma_1}\circ A_{\sigma_1,\tau_2}
-F_{\tau_2}\circ C_{\sigma_1,\tau_2}
\in \Hom(\tau_2,G).
\end{align*}
The important thing to note here is that $F_1'$, $F_2'$ and $\Lna(\tau_1,\tau_2)$ are independent.
Therefore, by Lemma~\ref{Lem: bound prob withoit t}
\begin{align*}
&\PPP\Big[(F_{\tau_1},F_{\tau_2})(\Lna(\tau_1,\tau_2))=(F_1',F_2') 
\mid F_1' \in \CC_{\tau_1}(T_2,F_{\tau_1}), 
F_2'\in \CC_{\tau_2}(T_1,F_{\tau_2})\Big]\\
&\leq \frac{K}{\abs{\CC_{\tau_1}(T_2,F_{\tau_1})}\times \abs{\CC_{\tau_2}(T_1,F_{\tau_2})}}.
\end{align*}

Next, we bound the probability that $F_1' \in \CC_{\tau_1}(T_2,F_{\tau_1})$ and $F_2'\in \CC_{\tau_2}(T_1,F_{\tau_2})$.
\begin{itemize}
\item Consider the case when $e_2=1-\epsilon$, that is $D_2>1$ and $T_1\subseteq T_2$.
For $i\in\tau_1$,
\[
F_1'(i)
=\sum_{j\in\sigma_2} a_{ij}(F_{\sigma_2}(j)-F_{\tau_1}(i)).
\]
By the definition of $\tau_1$ and $\sigma_2$, we know that $F_{\sigma_2}(j)\notin T_2$ and $F_{\tau_1}(i)\in T_1\subseteq T_2$. This means that $\sg{-F_{\tau_1}(i)+T_2}=\G(T_2)$ and $-F_{\tau_1}(i)+F_{\sigma_2}(j)\notin \G(T_2)$.
Since $\sigma_2\neq \emptyset$ and $a_{ij}$ are $\epsilon$-balanced, the probability of $F_1'(i)\in \G(T_2)$ is at most $1-\epsilon$.
\item Similarly in the case when $e_1=1-\epsilon$, the probability of $F_2'(j)\in \sg{-F_{\tau_2}(j)+T_1}$ is at most $1-\epsilon$.
\item Each of these conditions are independent and hence
\[
\PPP\Big[
F_1' \in \CC_{\tau_1}(T_2,F_{\tau_1})
\text{ and }
F_2'\in \CC_{\tau_2}(T_1,F_{\tau_2})
\Big]
\leq e_1^{\abs{\tau_1}} e_2^{\abs{\tau_2}}.
\]
\end{itemize}
We conclude that
\[
\PPP[(F_1,F_2)\circ \Lna =0]
\leq e_1^{\abs{\tau_1}} e_2^{\abs{\tau_2}}
\times \frac{K}{\abs{\CC_{\tau_1}(T_2,F_{\tau_1})}\times \abs{\CC_{\tau_2}(T_1,F_{\tau_2})}}.
\]
Next, we sum over different $F_1$ and $F_2$ corresponding to a fixed choice of $\tau_1$, $\tau_2$
\begin{align*}
&\sum_{\substack{F_1\in\SSS{n,\delta}{T_1}{S_1}\\ F_1(\tau_1)\subseteq T_1 } } 
\sum_{\substack{F_2\in\SSS{\ceil{\alpha n},\delta}{T_2}{S_2}\\ F_2(\tau_2)\subseteq T_2 }} 
\PPP[(F_1,F_2) \circ \Lna=0, \eta(\Lna)\in \EE_{n,\rho}]\\
&\leq 
\sum_{\substack{F_1\in\SSS{n,\delta}{T_1}{S_1}\\ F_1(\tau_1)\subseteq T_1 } } 
\sum_{\substack{F_2\in\SSS{\ceil{\alpha n},\delta}{T_2}{S_2}\\ F_2(\tau_2)\subseteq T_2 }} 
\PPP[(F_1,F_2) \circ \Lna=0]\\
&\leq K e_1^{\abs{\tau_1}} e_2^{\abs{\tau_2}}
\sum_{F_1\in\SSS{n,\delta}{T_1}{S_1}} 
\sum_{F_2\in\SSS{\ceil{\alpha n},\delta}{T_2}{S_2}}
\prod_{i\in \tau_1}\frac{1}{\abs{\sg{-F_1(i)+T_2}}}
\prod_{j\in \tau_2}\frac{1}{\abs{\sg{-F_2(j)+T_1}}}\\
&\leq K e_1^{\abs{\tau_1}} e_2^{\abs{\tau_2}}
\abs{G}^{\abs{\sigma_1}+\abs{\sigma_2}}
\left(
\sum_{g\in T_1} \frac{1}{\abs{\sg{-g+T_2}}}
\right)^{\abs{\tau_1}}
\left(
\sum_{h\in T_2} \frac{1}{\abs{\sg{-h+T_1}}}
\right)^{\abs{\tau_2}}.
\end{align*}

Note that $\sum_{g\in T_1} \frac{1}{\abs{\sg{-g+T_2}}} \geq \frac{1}{\abs{G}}$ and $\sum_{h\in T_2} \frac{1}{\abs{\sg{-h+T_1}}}\geq \frac{1}{\abs{G}}$.
Therefore,
\begin{align*}
&\sum_{\substack{F_1\in\SSS{n,\delta}{T_1}{S_1}\\ F_1(\tau_1)\subseteq T_1 } } 
\sum_{\substack{F_2\in\SSS{\ceil{\alpha n},\delta}{T_2}{S_2}\\ F_2(\tau_2)\subseteq T_2 }} 
\PPP[(F_1,F_2) \circ \Lna=0, \eta(\Lna)\in \EE_{n,\rho}]\\
&\leq K 
\left(\frac{\abs{G}^2}{1-\epsilon}\right)^{\abs{\sigma_1}}
\left(\frac{\abs{G}^2}{1-\epsilon}\right)^{\abs{\sigma_2}}
e_1^{n} e_2^{\alpha n}
\left(
\sum_{g\in T_1} \frac{1}{\abs{\sg{-g+T_2}}}
\right)^{n}
\left(
\sum_{h\in T_2} \frac{1}{\abs{\sg{-h+T_1}}}
\right)^{\alpha n}\\
&\leq  K 
\left(
\left(
\left(\frac{\abs{G}^2}{1-\epsilon}\right)^{\log_p(D_1)+\alpha\log_p(D_2)}
\right)^{\delta}
e_1 e_2^{\alpha }
\left(
\sum_{g\in T_1} \frac{1}{\abs{\sg{-g+T_2}}}
\right)
\left(
\sum_{h\in T_2} \frac{1}{\abs{\sg{-h+T_1}}}
\right)^{\alpha}
\right)^n.
\end{align*}
Moreover, the number of choices for $\sigma_1$ is at most
\[
\sum_{k\leq \log_p(D_1)\delta n} \binom{n}{k}
\leq \exp\left(2\sqrt{\log_p(D_1)\delta} n\right),
\]
and the number of choices for $\sigma_2$ is at most
\[
\sum_{k\leq \log_p(D_2)\delta \alpha n} \binom{\alpha n}{k}
\leq \exp\left(2\sqrt{\log_p(D_2)\delta} \alpha n\right).
\]
We conclude that
\begin{align*}
&\sum_{F_1\in\SSS{n,\delta}{T_1}{S_1}} \sum_{F_2\in\SSS{\ceil{\alpha n},\delta}{T_2}{S_2}} \PPP[(F_1,F_2) \circ \Lna=0, \eta(\Lna)\in \EE_{n,\rho}]\\
&\leq K 
\left(
\left(
\left(\frac{\abs{G}^2}{1-\epsilon}\right)^{2\log_p(\abs{G})}
\!\!\!\!\!\!\!\!\!\!\!\!
\exp\left(4\sqrt{\log_p(\abs{G})}\right)
\right)^{\sqrt{\delta}}
e_1 e_2^{\alpha }
\left(
\sum_{g\in T_1} \frac{1}{\abs{\sg{-g+T_2}}}
\right)
\left(
\sum_{h\in T_2} \frac{1}{\abs{\sg{-h+T_1}}}
\right)^{\alpha}
\right)^n.
\end{align*}

We know that $e_1\leq 1$ and $e_2\leq 1$. Since $\abs{T_1}\leq \abs{T_2}$, we see that
\begin{equation}\label{Eqn: prod leq 1}
\begin{split}
\left(
\sum_{g\in T_1} \frac{1}{\abs{\sg{-g+T_2}}}
\right)
\left(
\sum_{h\in T_2} \frac{1}{\abs{\sg{-h+T_1}}}
\right)^{\alpha}
&\leq \left(
\sum_{g\in T_1} \frac{1}{\abs{T_2}}
\right)
\left(
\sum_{h\in T_2} \frac{1}{\abs{T_1}}
\right)^{\alpha}\\
&= \frac{\abs{T_1}}{\abs{T_2}}
\left(\frac{\abs{T_2}}{\abs{T_1}}\right)^{\alpha}
\leq 1.
\end{split}
\end{equation}
This means that
\begin{equation}\label{Eqn: <1}
e_1 e_2^{\alpha }
\left(
\sum_{g\in T_1} \frac{1}{\abs{\sg{-g+T_2}}}
\right)
\left(
\sum_{h\in T_2} \frac{1}{\abs{\sg{-h+T_1}}}
\right)^{\alpha}
\leq 1.
\end{equation}

If the inequality in \eqref{Eqn: prod leq 1} was strict, then the inequality in \eqref{Eqn: <1} will also be strict.
On the other hand, if equality holds in \eqref{Eqn: prod leq 1}, then for every $T_1\subseteq T_2$ and $T_2\subseteq T_1$, hence $T_1=T_2$.
This would mean that either $T_1\subsetneq S_1$ or $T_2\subsetneq S_2$. Thus, either $e_1<1$ or $e_2<1$. In either case, the inequality in \eqref{Eqn: <1} is strict.

Therefore, we can pick a sufficiently small $\delta$ to ensure
\[
\left(
\left(\frac{\abs{G}^2}{1-\epsilon}\right)^{2\log_p(\abs{G})}
\!\!\!\!\!\!\!\!\!\!\!\!
\exp\left(4\sqrt{\log_p(\abs{G})}\right)
\right)^{\sqrt{\delta}}
e_1 e_2^{\alpha }
\left(
\sum_{g\in T_1} \frac{1}{\abs{\sg{-g+T_2}}}
\right)
\left(
\sum_{h\in T_2} \frac{1}{\abs{\sg{-h+T_1}}}
\right)^{\alpha}<1.
\]
This completes the proof.
\end{proof}

\subsection{Error term, Type 2}

In this subsection, we fix chains of translated subgroups $T_1\treq S_1\treq G$ and $T_2\treq S_2\treq G$ such that $\abs{T_1}> \abs{T_2}$.
This is the first time that we will use the assumption that $a_{ij}$ are identically distributed.

Given $g\in G$ and $F_2\in \SSS{\ceil{\alpha n},\delta}{T_2}{S_2}$, let $\sigma_2=\{j\in [n_2]: F_2(j)\notin T_2\}$.
We denote
\[
\nu_{F_2}(g)
=\PPP \left[
\sum_{j\in\sigma_2} a_{1j}(-g+F_2(j)) \in \sg{-g+T_2}
\right].
\]
For $g\in G\setminus T_2$ and $l\in \Z/p\Z$,
we denote
\begin{align*}
&\nu_{F_2}(g,l)\\
&=\PPP \left[
\sum_{j\in\sigma_2} a_{1j}(-g+F_2(j)) \in \sg{-g+T_2}
\text{ and }
\chi_{T_2,g}\Big( -\sum_{j\in\sigma_2} a_{1j}(-g+F_2(j)) \Big) \zeta_p^{\sum_{j\in\sigma_2} a_{1j}}
=\zeta_p^{l}
\right].
\end{align*}

\begin{lemma}\label{Lem: l neq0 sum <=1}
Consider $F_2\in \SSS{\ceil{\alpha n},\delta}{T_2}{S_2}$ and $l\in \Z/p\Z$. If $l\neq0$, then
\[
\sum_{g\in T_1\cap T_2} \frac{\nu_{F_2}(g)}{\abs{T_2}}
+\sum_{g\in T_1\setminus T_2} \frac{p\nu_{F_2}(g,l)}{\abs{\sg{-g+T_2}}}
\leq 1.
\]
\end{lemma}
\begin{proof}
For each $h\in T_2$, let $b_h\in \Z_p$ be a Haar-uniform random variable such that the $b_h$ and $a_{ij}$ are independent.
Note that
\begin{align*}
&\sum_{g\in T_1} 
\PPP\left[ \sum_{h\in T_2} b_h(-g+h) + \sum_{j\in\sigma_2} a_{1j} (-g+F_2(j))
=0
\mid \sum_{h\in T_2} b_h+\sum_{j\in\sigma_2} a_{1j}\equiv l \pmod{p} \right]\\
&=\sum_{g\in T_1} 
\PPP\left[ \sum_{h\in T_2} b_h(h) + \sum_{j\in\sigma_2} a_{1j} (-F_2(j))
=l g
\mid \sum_{h\in T_2} b_h+\sum_{j\in\sigma_2} a_{1j}\equiv l \pmod{p} \right]
\leq 1.
\end{align*}
Also note that
\begin{align*}
&\sum_{g\in T_1} 
\PPP\left[ \sum_{h\in T_2} b_h(-g+h) + \sum_{j\in\sigma_2} a_{1j} (-g+F_2(j))
=0
\mid \sum_{h\in T_2} b_h+\sum_{j\in\sigma_2} a_{1j}\equiv l \pmod{p} \right]\\
&=\sum_{g\in T_1} 
\PPP\left[ \sum_{h\in T_2} b_h(-g+h) =- \sum_{j\in\sigma_2} a_{1j} (-g+F_2(j))
\mid \sum_{h\in T_2} b_h
\equiv l-\sum_{j\in\sigma_2}a_{1j} \pmod{p} \right]
\end{align*}
We break into cases based on $g\in T_1\cap T_2$ or $g\in T_1\setminus T_2$.
\begin{enumerate}
\item Case 1: $g\in T_1\cap T_2$. If $\sum_{h\in T_2}b_h\equiv k\pmod{p}$, then 
$\sum_{h\in T_2} b_h(-g+h)$ is a uniformly random element of $\sg{-g+T_2}=\G(T_2)$.
Therefore,
\begin{align*}
&\PPP\left[ \sum_{h\in T_2} b_h(-g+h) =- \sum_{j\in\sigma_2} a_{1j} (-g+F_2(j))
\mid \sum_{h\in T_2} b_h
\equiv l-\sum_{j\in\sigma_2}a_{1j} \pmod{p} \right]\\
&=\frac{1}{\abs{T_2}} 
\PPP\left[ 
\sum_{j\in\sigma_2} a_{1j} (-g+F_2(j)) \in \G(T_2)
\right]
= \frac{1}{\abs{T_2}} \nu_{F_2}(g).
\end{align*}

\item Case 2: $g\in T_1\setminus T_2$. If $\sum_{h\in T_2}b_h\equiv k\pmod{p}$, then 
$\sum_{h\in T_2} b_h(-g+h)$ is a uniformly random element of
\[
\big\{x\in \sg{-g+T_2} : \chi_{T_2,g}(x)=\zeta_p^k \big\}.
\]
The size of this set is $\frac{\abs{\sg{-g+T_2}}}{p}$. Therefore,
\begin{align*}
&\PPP\left[ \sum_{h\in T_2} b_h(-g+h) =- \sum_{j\in\sigma_2} a_{1j} (-g+F_2(j))
\mid \sum_{h\in T_2} b_h
\equiv l-\sum_{j\in\sigma_2}a_{1j} \pmod{p} \right]\\
&=\frac{p}{\abs{\sg{-g+T_2}}} 
\PPP\left[ 
\sum_{j\in\sigma_2} a_{1j} (-g+F_2(j)) \in \sg{-g+T_2},
\chi_{T_2,g}\Big( -\sum_{j\in\sigma_2} a_{1j}(-g+F_2(j)) \Big)
=\zeta_p^{l-\sum_{j\in\sigma_2} a_{1j}}
\right]\\
&= \frac{p}{\abs{\sg{-g+T_2}}}
 \nu_{F_2}(g,l).
\end{align*}
\end{enumerate}
The result follows.
\end{proof}

\begin{proposition}\label{Prop: Error term type 2 to zero}
Let $p$ be a prime, $\frac{1}{p}<\alpha \leq 1$, $0<\rho<\alpha-\frac{1}{p}$, $G$ be a finite abelian $p$-group.
Suppose we have chains of translated subgroups $T_1\treq S_1\treq G$ and $T_2\treq S_2\treq G$.
If $\abs{T_1}> \abs{T_2}$,
then for sufficiently small $\delta>0$, we have
\[
\lim_{n\to\infty}
\sum_{F_1\in\SSS{n,\delta}{T_1}{S_1}} \sum_{F_2\in\SSS{\ceil{\alpha n},\delta}{T_2}{S_2}} \PPP[(F_1,F_2) \circ \Lna=0, \eta(\Lna)\in \EE_{n,\rho}]
=0.
\]
\end{proposition}
\begin{proof}
Let $D_1$, $D_2$, $\sigma_1$, $\sigma_2$, $\tau_1$, $\tau_2$ be as in the proof of Proposition~\ref{Prop: Error term type 1 to zero}. Define $F_{\tau_1}$, $F_{\tau_2}$, block matrix $\Lnot(\tau_1,\tau_2)$ and $F_1'\in\Hom(\tau_1,G)$, $F_2'\in\Hom(\tau_2,G)$ in the same way as in the proof of Proposition~\ref{Prop: Error term type 1 to zero}.
Recall that $F_1'$, $F_2'$ and $\Lna(\tau_1,\tau_2)$ are independent. Also recall that if $(F_1,F_2)\circ \Lna=0$, then $(F_{\tau_1},F_{\tau_2})\circ \Lna(\tau_1,\tau_2)=(F_1',F_2')$.

Pick $\vec{t}\in \EE_{n,\rho}$.
Let $\vec{u}\in (\Z/p\Z)^{\tau_1}$ be the vector whose entries are $u_i=t_i-\sum_{j\in \sigma_2} a_{ij}$ for $i\in \tau_1$. This is a random vector, but it is independent of $\Lna(\tau_1,\tau_2)$.
Note that
\[
\PPP[(F_1,F_2) \circ \Lna=0, \eta(\Lna)=\vec{t}]
\leq 
\PPP[(F_{\tau_1},F_{\tau_2})\circ (\Lna(\tau_1,\tau_2))=(F_1',F_2'), \eta(\Lna(\tau_1,\tau_2)) =\vec{u}].
\]
This probability is zero unless the following conditions hold.
\begin{enumerate}
\item For each $i\in \tau_1$, $F_1'(i)\in \sg{-F_1(i)+T_2}$. 
\item For each $j\in \tau_2$, $F_2'(j)\in \sg{-F_2(j)+T_1}$.
\item For each $i\in \tau_1$, if $F_{\tau_1}(i)\notin T_2$, then $\chi_{T_2,F_1(i)}(F_1'(i))=\zeta_p^{u_i}$.
\end{enumerate}
Note that
\[
F_1'(i)=-\sum_{j\in \sigma_2} a_{ij} (-F_1(i)+F_2(j)).
\]
Since all $a_{ij}$ are independent and identically distributed, the probability for conditions $(1), (3)$ to hold is
\[
\prod_{\substack{i\in \tau_1\\ F_1(i)\in T_2}} \nu_{F_2}(F_1(i))
\prod_{\substack{i\in \tau_1\\ F_1(i)\notin T_2}} \nu_{F_2}(F_1(i),t_i).
\]
The probability for condition $(2)$ to hold is at most $1$.
Proposition~\ref{Prop: Individual code bound} tells us that
\begin{align*}
&\PPP[(F_{\tau_1},F_{\tau_2})\circ (\Lna(\tau_1,\tau_2))=(F_1'F_2'),
\eta(\Lna(\tau_1,\tau_2)) =\vec{u}]\\
&=\PPP[(F_{\tau_1},F_{\tau_2})\circ (\Lna(\tau_1,\tau_2))=(F_1',F_2'), 
\eta(\Lna(\tau_1,\tau_2)) =\vec{u} \mid (1), (2), (3)] 
\times\PPP[(1), (2), (3)]\\
&\leq \frac{K}{\abs{\CC^*_{\abs{\tau_1}}(T_2,F_{\tau_1})} \abs{\CC^*_{\abs{\tau_2}}(T_1,F_{\tau_2})} p^{\abs{\tau_1}}}
p^{\abs{\{i\in \tau_1: F_1(i)\notin  T_2\}}}
\prod_{\substack{i\in \tau_1\\ F_1(i)\in T_2}} \nu_{F_2}(F_1(i))
\prod_{\substack{i\in \tau_1\\ F_1(i)\notin T_2}} \nu_{F_2}(F_1(i),t_i)\\
&=\frac{K}{p^{\abs{\tau_1}}}
\prod_{\substack{i\in \tau_1\\ F_1(i)\in T_2}} \frac{\nu_{F_2}(F_1(i))}{\abs{\sg{-F_1(i)+T_2}}} 
\prod_{\substack{i\in \tau_1\\ F_1(i)\notin T_2}} \frac{p\nu_{F_2}(F_1(i),t_i)}{\abs{\sg{-F_1(i)+T_2}}} 
\prod_{j\in \tau_2} \frac{1}{\abs{\sg{-F_2(j)+T_1}}}.
\end{align*}
Note that if $F_1(i)\in T_2$, then $\abs{\sg{-F_1(i)+T_2}}=\abs{T_2}$.
Therefore,
\begin{align*}
&\PPP[(F_1,F_2) \circ \Lna=0, \eta(\Lna)=\vec{t}]\\
&\leq \frac{K}{p^{\abs{\tau_1}}}
\prod_{\substack{i\in \tau_1\\ F_1(i)\in T_2}} \frac{\nu_{F_2}(F_1(i))}{\abs{T_2}} 
\prod_{\substack{i\in \tau_1\\ F_1(i)\notin T_2}} \frac{p\nu_{F_2}(F_1(i),t_i)}{\abs{\sg{-F_1(i)+T_2}}} 
\prod_{j\in \tau_2} \frac{1}{\abs{\sg{-F_2(j)+T_1}}}.
\end{align*}
If we sum over different $F_1$ corresponding to the same $\tau_1$, we see that
\begin{align*}
&\sum_{\substack{F_1\in\SSS{n,\delta}{T_1}{S_1}\\ F_1(\tau_1)\subseteq T_1 } } 
\PPP[(F_1,F_2) \circ \Lna=0, \eta(\Lna)=\vec{t}]\\
&\leq \frac{K}{p^{\abs{\tau_1}}}
\prod_{j\in \tau_2} \frac{1}{\abs{\sg{-F_2(j)+T_1}}}
\abs{S_1}^{\abs{\sigma_1}}
\prod_{i\in\tau_1}
\left(
\sum_{g\in T_1\cap T_2} \frac{\nu_{F_2}(g)}{\abs{T_2}}
+\sum_{g\in T_1\setminus T_2} \frac{p\nu_{F_2}(g,t_i)}{\abs{\sg{-g+T_2}}} 
\right)\\
&= \frac{K}{p^{n}}
(p\abs{G})^{\abs{\sigma_1}}
\prod_{j\in \tau_2} \frac{1}{\abs{\sg{-F_2(j)+T_1}}}
\prod_{l\in \Z/p\Z}
\left(
\sum_{g\in T_1\cap T_2} \frac{\nu_{F_2}(g)}{\abs{T_2}}
+\sum_{g\in T_1\setminus T_2} \frac{p\nu_{F_2}(g,l)}{\abs{\sg{-g+T_2}}} 
\right)^{\abs{\{i\in \tau_1: t_i=l\}}}.
\end{align*}
By Lemma~\ref{Lem: l neq0 sum <=1}, we see that
\begin{align*}
&\sum_{\substack{F_1\in\SSS{n,\delta}{T_1}{S_1}\\ F_1(\tau_1)\subseteq T_1 } } 
\PPP[(F_1,F_2) \circ \Lna=0, \eta(\Lna)=\vec{t}]\\
&\leq 
\frac{K}{p^{n}}
(p\abs{G})^{\log_p(D_1)\delta n}
\prod_{j\in \tau_2} \frac{1}{\abs{\sg{-F_2(j)+T_1}}}
\left(
\sum_{g\in T_1\cap T_2} \frac{\nu_{F_2}(g)}{\abs{T_2}}
+\sum_{g\in T_1\setminus T_2} \frac{p\nu_{F_2}(g,0)}{\abs{\sg{-g+T_2}}}
\right)^{\abs{\{i\in \tau_1: t_i=0\}}}.
\end{align*}
If $g\notin T_2$, then $\sg{-g+T_2}\supsetneq \G(T_2)$, hence $\abs{\sg{-g+T_2}}\geq p\times \abs{T_2}$.
This implies that
\[
\sum_{g\in T_1\cap T_2} \frac{\nu_{F_2}(g)}{\abs{T_2}}
+\sum_{g\in T_1\setminus T_2} \frac{p\nu_{F_2}(g,0)}{\abs{\sg{-g+T_2}}}
\leq 
\sum_{g\in T_1\cap T_2} \frac{1}{\abs{T_2}}
+\sum_{g\in T_1\setminus T_2} \frac{p}{p\abs{T_2}}
=\frac{\abs{T_1}}{\abs{T_2}}.
\]
We see that
\begin{align*}
&\sum_{\substack{F_1\in\SSS{n,\delta}{T_1}{S_1}\\ F_1(\tau_1)\subseteq T_1 } } 
\PPP[(F_1,F_2) \circ \Lna=0, \eta(\Lna)=\vec{t}]\\
&\leq 
\frac{K}{p^{n}}
(p\abs{G})^{\log_p(D_1)\delta n}
\prod_{j\in \tau_2} \frac{1}{\abs{\sg{-F_2(j)+T_1}}}
\left(
\frac{\abs{T_1}}{\abs{T_2}}
\right)^{\abs{\{i\in \tau_1: t_i=0\}}}.
\end{align*}
Since $\abs{T_1}>\abs{T_2}$ and $\vec{t}\in \EE_{n,\rho}$, we know that
\[
\left(
\frac{\abs{T_1}}{\abs{T_2}}
\right)^{\abs{\{i\in \tau_1: t_i=0\}}}
\leq
\left(
\frac{\abs{T_1}}{\abs{T_2}}
\right)^{\abs{\{i\in [n_1]: t_i=0\}}}
\leq
\left(
\frac{\abs{T_1}}{\abs{T_2}}
\right)^{(\frac{1}{p}+\rho)n}.
\]
The number of choices for $\tau_1$ is at most
\[
\sum_{k\leq \log_p(D_1)\delta n} \binom{n}{k}
\leq \exp\left(2\sqrt{\log_p(D_1)\delta} n\right)
\leq
\exp\left(2\sqrt{\log_p(\abs{G})}\right)^{n\sqrt{\delta}}.
\]
Therefore,
\begin{align*}
&\sum_{F_1\in\SSS{n,\delta}{T_1}{S_1}} 
\PPP[(F_1,F_2) \circ \Lna=0, \eta(\Lna)=\vec{t}]\\
&\leq \frac{K}{p^{n}}
\left(
(p\abs{G})^{\log_p(D_1)}
\exp\left(2\sqrt{\log_p(\abs{G})}\right)
\right)^{\sqrt{\delta}n}
\left(
\frac{\abs{T_1}}{\abs{T_2}}
\right)^{(\frac{1}{p}+\rho)n}
\prod_{j\in \tau_2} \frac{1}{\abs{\sg{-F_2(j)+T_1}}}.
\end{align*}
By summing over different the $F_2$ corresponding to the same $\tau_2$, we see that
\begin{align*}
&\sum_{\substack{F_2\in\SSS{\ceil{\alpha n},\delta}{T_2}{S_2}\\ F_2(\tau_2)\subseteq T_2 } } 
\sum_{F_1\in\SSS{n,\delta}{T_1}{S_1}} 
\PPP[(F_1,F_2) \circ \Lna=0, \eta(\Lna)=\vec{t}]\\
&\leq \frac{K}{p^{n}}
\left(
(p\abs{G})^{\log_p(D_1)}
\exp\left(2\sqrt{\log_p(\abs{G})}\right)
\right)^{\sqrt{\delta}n}
\left(
\frac{\abs{T_1}}{\abs{T_2}}
\right)^{(\frac{1}{p}+\rho)n}
\abs{S_2}^{\abs{\sigma_2}}
\left(
\sum_{h\in T_2} \frac{1}{\abs{\sg{-h+T_1}}} 
\right)^{\abs{\tau_2}}\\
&\leq \frac{K}{p^{n}}
\left(
(p\abs{G}^2)^{\log_p(\abs{G})}
\exp\left(2\sqrt{\log_p(\abs{G})}\right)
\right)^{\sqrt{\delta}n}
\left(
\frac{\abs{T_1}}{\abs{T_2}}
\right)^{(\frac{1}{p}+\rho)n}
\left(
\frac{\abs{T_2}}{\abs{T_1}} 
\right)^{\abs{\tau_2}}.
\end{align*}
Since $\abs{T_2}<\abs{T_1}$, we see that
\[
\left(
\frac{\abs{T_2}}{\abs{T_1}} 
\right)^{\abs{\tau_2}}
\leq
\left(
\frac{\abs{T_2}}{\abs{T_1}} 
\right)^{\alpha n(1-\delta\log_p(D_2))}
\leq 
\left(
\frac{\abs{T_2}}{\abs{T_1}} 
\right)^{\alpha n}
\abs{G}^{n\delta\log_p(\abs{G})}.
\]
Therefore,
\begin{align*}
&\sum_{\substack{F_2\in\SSS{\ceil{\alpha n},\delta}{T_2}{S_2}\\ F_2(\tau_2)\subseteq T_2 } } 
\sum_{F_1\in\SSS{n,\delta}{T_1}{S_1}} 
\PPP[(F_1,F_2) \circ \Lna=0, \eta(\Lna)=\vec{t}]\\
&\leq \frac{K}{p^{n}}
\left(
(p\abs{G}^3)^{\log_p(\abs{G})}
\exp\left(2\sqrt{\log_p(\abs{G})}\right)
\right)^{\sqrt{\delta}n}
\left(
\frac{\abs{T_2}}{\abs{T_1}} 
\right)^{(\alpha-\frac{1}{p}-\rho) n}.
\end{align*}
The number of choices for $\tau_2$ is at most
\[
\sum_{k\leq \log_p(D_2)\delta \alpha n} \binom{\alpha n}{k}
\leq \exp\left(2\sqrt{\log_p(D_2)\delta} \alpha n\right)
\leq \exp\left(2\sqrt{\log_p(\abs{G})}\right)^{n\sqrt{\delta}}.
\]
We conclude that
\begin{align*}
&\sum_{F_2\in\SSS{\ceil{\alpha n},\delta}{T_2}{S_2}}
\sum_{F_1\in\SSS{n,\delta}{T_1}{S_1}} 
\PPP[(F_1,F_2) \circ \Lna=0, \eta(\Lna)=\vec{t}]\\
&\leq \frac{K}{p^{n}}
\left(
(p\abs{G}^3)^{\log_p(\abs{G})}
\exp\left(4\sqrt{\log_p(\abs{G})} \right)
\right)^{\sqrt{\delta}n}
\left(
\frac{\abs{T_2}}{\abs{T_1}} 
\right)^{(\alpha-\frac{1}{p}-\rho) n}.
\end{align*}
Since $\abs{\EE_{n,\rho}}\leq p^n$, we see that
\begin{align*}
&\sum_{F_2\in\SSS{\ceil{\alpha n},\delta}{T_2}{S_2}}
\sum_{F_1\in\SSS{n,\delta}{T_1}{S_1}} 
\PPP[(F_1,F_2) \circ \Lna=0, \eta(\Lna)\in \EE_{n,\rho}]\\
&\leq K
\left(
(p\abs{G}^3)^{\log_p(\abs{G})}
\exp\left(4\sqrt{\log_p(\abs{G})} \right)
\right)^{\sqrt{\delta}n}
\left(
\frac{\abs{T_2}}{\abs{T_1}} 
\right)^{(\alpha-\frac{1}{p}-\rho) n}.
\end{align*}

We know that $\abs{T_2}<\abs{T_1}$ and $\rho<\alpha -\frac{1}{p}$, so
\[
\left(
\frac{\abs{T_2}}{\abs{T_1}} 
\right)^{(\alpha-\frac{1}{p}-\rho) n}
<1.
\]
We can choose a sufficiently small $\delta>0$ to ensure that
\[
\left(
(p\abs{G}^3)^{\log_p(\abs{G})}
\exp\left(4\sqrt{\log_p(\abs{G})} \right)
\right)^{\sqrt{\delta}}
\left(
\frac{\abs{T_2}}{\abs{T_1}} 
\right)^{(\alpha-\frac{1}{p}-\rho)}
<1.
\]
The result follows.
\end{proof}

\begin{proof}[Proof of Proposition~\ref{Prop: moments on parts}]
This follows from Corollary~\ref{Cor: main term}, Proposition~\ref{Prop: Error term type 1 to zero} and Proposition~\ref{Prop: Error term type 2 to zero}.
\end{proof}

\section{Parity of $2$-rank}

In this section $p=2$. We work with matrices over $\mathbb{F}_2$, and ranks of matrices are taken over $\mathbb{F}_2$. For a vector $v$, we write $\Diag(v)$ for the diagonal matrix whose diagonal entries are the entries of $v$, and we write $\mathbf{1}_m$ for the all-ones column vector of length $m$.

Reducing $\Lnot$ modulo $2$, the minus signs disappear, so if $A=(a_{ij})\in \MM_{n_1\times n_2}(\mathbb{F}_2)$, then
\[
\overline{\Lnot}
=
\begin{bmatrix}
\Diag(A\mathbf{1}_{n_2}) & A\\
A^T & \Diag(A^T\mathbf{1}_{n_1})
\end{bmatrix}.
\]

Our goal is to determine the parity of $\rank(\overline{\Lnot})$.

\begin{lemma}\label{lem:diag-in-image}
Let $B\in \MM_k(\mathbb{F}_2)$ be symmetric, and let
\[
v=\begin{bmatrix} b_{11}\\ \vdots \\ b_{kk}\end{bmatrix}\in \mathbb{F}_2^k
\]
be the vector of diagonal entries of $B$. Then $v\in \im(B)$.
\end{lemma}

\begin{proof}
It suffices to show that for every $y\in \mathbb{F}_2^{1\times k}$, if $yB=0$, then $yv=0$. We have
\[
yBy^T
=
\sum_{i=1}^k y_i^2 b_{ii}
+
\sum_{1\le i<j\le k} y_i y_j (b_{ij}+b_{ji}).
\]
Over $\mathbb{F}_2$ we have $y_i^2=y_i$ and $b_{ij}+b_{ji}=0$, since $B$ is symmetric. This means that
\[
yBy^T=\sum_{i=1}^k y_i b_{ii}=yv.
\]
Therefore, if $yB=0$, then $yv=0$. This completes the proof.
\end{proof}

\begin{proposition}\label{prop:parity-rank}
Let $A=(a_{ij})\in \MM_{n_1\times n_2}(\mathbb{F}_2)$, and let
\[
L(A)=
\begin{bmatrix}
\Diag(A\mathbf{1}_{n_2}) & A\\
A^T & \Diag(A^T\mathbf{1}_{n_1})
\end{bmatrix}\in \MM_{n_1+n_2}(\mathbb{F}_2).
\]
Then
\[
\rank(L(A))\equiv \sum_{i=1}^{n_1}\sum_{j=1}^{n_2} a_{ij}\pmod 2.
\]
\end{proposition}

\begin{proof}
We induct on the number of nonzero entries of $A$.
If $A=0$, then $L(A)=0$, so the base case is clear.

Now assume $A\neq 0$, and choose $(i_0,j_0)$ such that $a_{i_0j_0}=1$. Assume without loss of generality that $i_0=1$ and $j_0=1$. Let $E_{11}$ be the matrix with $1$ at the $(1,1)$ entry and zeros elsewhere. 
Set $A'=A+E_{11}$, so $A'$ has one fewer nonzero entry than $A$. By the induction hypothesis,
\[
\rank(L(A'))\equiv \sum_{i,j} a'_{ij}\equiv \sum_{i,j} a_{ij}+1 \pmod 2.
\]
It therefore suffices to prove that
\[
\rank(L(A))\equiv \rank(L(A'))+1\pmod 2.
\]

We first remove a redundant row and column. Every row sum and every column sum of both $L(A)$ and $L(A')$ is zero. Hence, in each matrix, the $(n_1+1)$-th column is the sum of the other columns, so deleting that column does not change the column rank. After deleting that column, the $(n_1+1)$-th row is the sum of the remaining rows, so deleting that row does not change the rank. Let $\widetilde{L}(A)$ and $\widetilde{L}(A')$ denote the resulting principal submatrices. Then
\[
\rank(L(A))=\rank(\widetilde{L}(A)),
\qquad
\rank(L(A'))=\rank(\widetilde{L}(A')).
\]

There is a scalar $d\in \mathbb{F}_2$ such that
\[
\widetilde{L}(A')=
\begin{bmatrix}
d & \vec{x}^T\\
\vec{x} & K
\end{bmatrix},
\qquad
\widetilde{L}(A)=
\begin{bmatrix}
d+1 & \vec{x}^T\\
\vec{x} & K
\end{bmatrix},
\]
for some symmetric matrix $K$ and some column vector $\vec{x}$.

We now describe $K$ and $\vec{x}$ explicitly. Let $B$ be the $(n_1-1)\times (n_2-1)$ matrix obtained from $A$ by deleting row $1$ and column $1$. Let $\vec{u}\in \mathbb{F}_2^{n_2-1}$
be the first row of $A$ with the first entry removed, turned into a column. 
Let $\vec{b}\in \mathbb{F}_2^{n_1-1}$ be the vector of row sums of $A$, ignoring the first row.
Let $\vec{c}\in \mathbb{F}_2^{n_2-1}$ be the vector of column sums of $A$, ignoring the first column.
Then
\[
K=
\begin{bmatrix}
\Diag(\vec{b}) & B\\
B^T & \Diag(\vec{c})
\end{bmatrix},
\qquad
\vec{x}=
\begin{bmatrix}
0\\
\vec{u}
\end{bmatrix},
\]
where the top zero vector has length $n_1-1$.
Also note that $\mathbf{1}_{n_1-1}^T B +\vec{u}^T=\vec{c}^T$.

We will show that $\vec{x}\in \im(K)$, assume for the moment that this is the case.
Choose $\vec{z}\in\mathbb{F}_2^{n_1+n_2-2}$ such that $K\vec{z}=\vec{x}$. For $e\in \{0,1\}$, we have
\[
\begin{bmatrix}
1 & 0\\
\vec{z} & I
\end{bmatrix}^T
\begin{bmatrix}
d+e & \vec{x}^T\\
\vec{x} & K
\end{bmatrix}
\begin{bmatrix}
1 & 0\\
\vec{z} & I
\end{bmatrix}
=
\begin{bmatrix}
d+e+\vec{z}^T K \vec{z} & \vec{x}^T+\vec{z}^TK\\
\vec{x}+K\vec{z} & K
\end{bmatrix}.
\]
Since $K$ is symmetric and we are working over $\mathbb{F}_2$, we have $\vec{x}+K\vec{z}=0$ and $\vec{x}^T+\vec{z}^TK=0$.
Thus exactly one of $\widetilde{L}(A')$ and $\widetilde{L}(A)$ has rank $\rank(K)$, and the other has rank $\rank(K)+1$. Therefore
\[
\rank(\widetilde{L}(A'))
\equiv
\rank(\widetilde{L}(A))+1 \pmod 2.
\]
This completes the induction step assuming $\vec{x}\in\im(K)$.

Now, we want to show that $\vec{x}\in\im(K)$.
Equivalently, we want to find $\vec{y}_1\in\mathbb{F}_2^{n_1-1}$ and $\vec{y}_2\in\mathbb{F}_2^{n_2-1}$ such that
\begin{align}
\Diag(\vec{b})\vec{y}_1 + B\vec{y}_2 &= 0,\label{eq:first-block}\\
B^T\vec{y}_1 + \Diag(\vec{c})\vec{y}_2 &= \vec{u}.\label{eq:second-block}
\end{align}

Let
\[
I_0=\{r\in \{1,\dots,n_1-1\}:  \vec{b}_r=0\},
\]
and let $W\subseteq \mathbb{F}_2^{n_2-1}$ be the span of the rows of $B$ indexed by $I_0$. 
Note that \eqref{eq:first-block} implies that $\vec{y}_2\in W^{\perp}$.
Here $W^{\perp}$ is defined in terms of the usual dot product on $\mathbb{F}_2^{n_2-1}$.

Define a linear map
\[
T:W^\perp\to \mathbb{F}_2^{n_2-1},
\qquad
T(\vec{y})
=B^TB\vec{y}+\Diag(\vec{c})\vec{y}.
\]
Choose a basis $w_1,\dots,w_k$ of $W^\perp$, and define $N\in \MM_k(\mathbb{F}_2)$ by $N_{l m}:=w_l^T T(w_m)$.
Since
\[
w_l^T T(w_m)
=
(Bw_l)^T(Bw_m)+w_l^T\Diag(\vec{c})w_m,
\]
the matrix $N$ is symmetric.

Next we compute its diagonal. Let $w\in W^\perp$. Since we are working over $\mathbb{F}_2$, we have
\[
w^T T(w)
=
(Bw)^T(Bw)+w^T\Diag(\vec{c})w
=
\mathbf{1}_{n_1-1}^T B w + \vec{c}^T w.
\]
We have $\mathbf{1}_{n_1-1}^T B = (\vec{c}+\vec{u})^T$, so
\[
w^T T(w)
=(\vec{c}+\vec{u})^T w + \vec{c}^T w 
= \vec{u}^T w.
\]
Therefore the diagonal vector of $N$ is
\[
\begin{bmatrix}
\vec{u}^T w_1\\
\vdots\\
\vec{u}^T w_k
\end{bmatrix}.
\]
By Lemma~\ref{lem:diag-in-image}, this vector lies in $\im(N)$. Thus there exists
\begin{align*}
&q=
\begin{bmatrix}
q_1\\
\vdots\\
q_k
\end{bmatrix}\in \mathbb{F}_2^k,
&
&\text{such that}
&
&Nq=
\begin{bmatrix}
\vec{u}^T w_1\\
\vdots\\
\vec{u}^T w_k
\end{bmatrix}.
\end{align*}
Choose
\[
\vec{y}_2
=q_1 w_1+\cdots+q_k w_k\in W^\perp.
\]
Then for each $l$, $w_l^T T(\vec{y}_2)
=\vec{u}^T w_l$, so
\[
\vec{u}-T(\vec{y}_2)\in (W^\perp)^\perp = W.
\]

Since $W$ is spanned by the rows of $B$ indexed by $I_0$, there exists a vector $\vec{g}\in \mathbb{F}_2^{n_1-1}$ supported on $I_0$ such that
\[
B^T \vec{g} 
= \vec{u}-T(\vec{y}_2).
\]
Choose $\vec{y}_1=B\vec{y}_2+\vec{g}$.

We verify \eqref{eq:first-block} and \eqref{eq:second-block}. Since $\vec{y}_2\in W^\perp$, the vector $B\vec{y}_2$ vanishes on $I_0$. 
Since we are working over $\mathbb{F}_2$, $\Diag(\vec{b})B\vec{y}_2=B\vec{y}_2$. Also, $\Diag(\vec{b})\vec{g}=0$ because $\vec{g}$ is supported on $I_0$. Therefore,
\[
\Diag(\vec{b})\vec{y}_1 + B\vec{y}_2
=
\Diag(\vec{b})(B\vec{y}_2+\vec{g})+B\vec{y}_2
=
B\vec{y}_2+B\vec{y}_2
=
0.
\]
This proves \eqref{eq:first-block}. Also,
\[
B^T\vec{y}_1+\Diag(\vec{c})\vec{y}_2
=
B^T(B\vec{y}_2+\vec{g})+\Diag(\vec{c})\vec{y}_2
=
T(\vec{y}_2)+B^T \vec{g}
=
T(\vec{y}_2)+\vec{u}-T(\vec{y}_2)
=
\vec{u},
\]
which proves \eqref{eq:second-block}. Hence $\vec{x}\in \im(K)$.
\end{proof}

\begin{customprop}{\ref{Prop: p=2 rank parity}}
Suppose $p=2$. For any $0<\alpha\leq 1$,
\[
\lim_{n\to\infty}
\PPP
\left[ 
\rank(G(\Lna))
\text{ is odd}
\right]
=
\frac{1}{2}.
\]
\end{customprop}
\begin{proof}
Note that 
\[
\rank(G(\Lna))
=n+\ceil{\alpha n}-1-\rank\left(\overline{\Lna}\right).
\]
Therefore,
\[
\PPP
\left[ 
\rank(G(\Lna))
\text{ is odd}
\right]
=\PPP
\left[
\rank\left(\overline{\Lna}\right)
\equiv n+\ceil{\alpha n} \pmod 2
\right].
\]
By Proposition~\ref{prop:parity-rank}, this is
\[
\PPP
\left[
\sum_{i=1}^n \sum_{j=1}^{\ceil{\alpha n}} a_{ij}
\equiv n+\ceil{\alpha n} \pmod 2
\right].
\]
Since the $a_{ij}$ are independent and $\epsilon$-balanced, \cite[Lemma 2.1]{wood2019random} implies that
\[
\left|
\PPP
\left[
\sum_{i=1}^n \sum_{j=1}^{\ceil{\alpha n}} a_{ij}
\equiv n+\ceil{\alpha n}+1 \pmod 2
\right]
-\frac{1}{2}
\right|
< \exp\Big(-\frac{\epsilon \alpha n^2}{2^2}\Big).
\]
The result follows.
\end{proof}

\section{Raw moments sometimes diverge to infinity}\label{Sec: Raw moment infty}

We consider the special case of $\Lna$ when all $a_{ij}$ are Haar-uniform and show that certain moments blow up to infinity.

\begin{proposition}
Consider $0<\alpha<1$ and $m>\frac{2-\alpha}{1-\alpha}$.
If all $a_{ij}$ in $\Lna$ are Haar-uniform, then we have
\[
\lim_{n\to\infty} 
\E[\abs{\Hom(G(\Lna),(\Z/p\Z)^m))}]
=\infty.
\]
\end{proposition}
\begin{proof}
Denote $G=(\Z/p\Z)^m$, $n_1=n$ and $n_2=\ceil{\alpha n}$.
Recall that
\[
\abs{\Hom(G(\Lna),(\Z/p\Z)^m))}
=\frac{1}{p^m} \abs{\Hom(\Coker(\Lna),(\Z/p\Z)^m))}.
\]
Moreover,
\[
\E\left[
\abs{\Hom(\Coker(\Lna),(\Z/p\Z)^m))}
\right]
=\sum_{F_1\in \Hom([n_1],G)}
\sum_{F_2\in \Hom([n_2],G)}
\PPP [ (F_1,F_2)\circ \Lna =0 ].
\]

It follows from \eqref{Eqn: Prob F1 F2 Lnot = F'} that
\begin{align*}
&\PPP [ (F_1,F_2)\circ \Lna =0 ]\\
&=\frac{1}{\abs{\CC^*_{n_1}(\sgo{F_2},F_1)}}
\frac{1}{\abs{\CC^*_{n_2}(\sgo{F_1},F_2)}}
\sum_{C_1\in \CC^*_{n_1}(\sgo{F_2},F_1)}
\sum_{C_2\in \CC^*_{n_2}(\sgo{F_1},F_2)}
\prod_{i=1}^{n_1} \prod_{j=1}^{n_2}
\E\left[E[F_1,F_2,C_1,C_2](i,j)^{a_{ij}}\right].
\end{align*}
Since all the $a_{ij}$ are Haar-uniform, we know that the only $(C_1,C_2)$ that will contribute are the ones for which all coefficients $E[F_1,F_2,C_1,C_2](i,j)=1$.
We call $(C_1,C_2)$ as $(F_1,F_2)$-special if this condition holds.

Let $T$ be a subgroup of $G$ isomorphic to $\Z/p\Z$. Let $D(G,T)$ be the set of pairs of functions $(\beta_1,\beta_2)$ such that the following hold:
\begin{enumerate}
    \item $\beta_1$ maps $G$ to $\bigcup_{g\in G}\sg{-g+ T}^*$ and for every $g\in G$, $\beta_1(g)\in \sg{-g+T}^*$;
    \item 
    $\beta_2$ maps $T$ to $G^*$;
    \item For every $g\in G$ and $h\in T$, we have
    \[
    \beta_1(g)(-g+h)\beta_2(h)(-h+g)
    =1.
    \]
\end{enumerate}

As a lower bound, we only sum over $F_1\in \Hom([n_1],G)$, $F_2\in \Hom([n_2],T)$ with $\set{F_1}=G$ and $\set{F_2}=T$.
For such $F_1$, $F_2$, by an argument identical to Section~\ref{Sec: Special}, we see that the number of pairs $(C_1,C_2)$ that are $(F_1,F_2)$-special is $\abs{D(G,T)}$.
Therefore,
\[
\PPP [ (F_1,F_2)\circ \Lna =0 ]
=\frac{ \abs{ D(G,T)  }  }{\abs{\CC_{n_1}^*(T,F_1)}\times \abs{\CC_{n_2}^*(G,F_2)}}.
\]
Note that
\[
\abs{\CC_{n_2}^*(G,F_2)}
=\abs{G}^{n_2}
=p^{m n_2}.
\]
Moreover,
\[
\abs{\sg{-g+T}}
=\begin{cases}
p &\text{if }g\in T,\\
p^2 &\text{if }g\notin T,
\end{cases}
\]
so
\[
\abs{\CC_{n_1}^*(T,F_1)}
\leq p^{2n_1}.
\]

Finally, note that the number of $F_1\in \Hom([n_1],G)$, with $\set{F_1}=G$ is $\abs{G}^{n_1}(1-o(1))$ and the number of $F_2\in \Hom([n_1],T)$ with $\set{F_2}=T$ is $p^{n_2}(1-o(1))$.
Therefore,
\begin{align*}
\E\left[
\abs{\Hom(\Coker(\Lna),(\Z/p\Z)^m))}
\right]
&\geq p^{mn}(1-o(1))
\times p^{\alpha n}(1-o(1))
\frac{\abs{D(G,T)}}{p^{2n} p^{m\alpha n}}\\
&=(1-o(1)) \abs{D(G,T)} \Big(p^{m+\alpha-2-m\alpha}\Big)^n.
\end{align*}
We know that $D(G,T)\geq 1$.
By choice of $\alpha$, we know that $m+\alpha-2-m\alpha>0$, which completes the proof.
\end{proof}

\section{Lemmas}

Denote the binary entropy function as $H(x)=-x\log_2(x)-(1-x)\log_2(1-x)$. It is known that $\binom{a}{b}\leq 2^{H(\frac{b}{a})a}$ and that $H(\alpha)\leq 2\sqrt{\alpha}$. It follows that
\[
\binom{a}{b}\leq 2^{2\sqrt{ab}}.
\]

\begin{lemma}
For every integer $a \ge 1$ and every real number $0 < \delta \le 1$, we have
\[
\sum_{k=0}^{\floor{\delta a}} \binom{a}{k}
\leq 2^{2\sqrt{\delta} a}.
\]
\end{lemma}
\begin{proof}
First assume $0<\delta\leq \frac{1}{2}$. 
Using the standard entropy bound for the lower tail of binomial coefficients together with the estimate $H(\delta)\leq 2\sqrt{\delta}$, we get
\[
\sum_{k=0}^{\floor{\delta a}}\binom{a}{k}
\leq 2^{H(\delta)a}
\leq 2^{2\sqrt{\delta} a}.
\]

For $\frac{1}{2}<\delta\leq 1$, we have the trivial bound $\sum_{k=0}^{\floor{\delta a}}\binom{a}{k}\leq 2^a<2^{2\sqrt{\delta} a}$.
\end{proof}

\subsection{Counting codes and non-codes}

We will obtain bounds on the sizes of $\SSS{n,\delta}{S}{T}$. We start with the case $T\subsetneq S$.

\begin{lemma}\cite[Lemma 5.2]{wood2017distribution}\label{Lem: bound S ST}
Given a chain of translated subgroups $T\treq S\treq G$, denote $D=[S:T]>1$. Then we have
\[
\abs{\SSS{n,\delta}{S}{T}}
\leq 4^{n\sqrt{\delta\log_p(\abs{G})}}
\abs{T}^n D^{\delta\log_p(D)n}.
\]
\end{lemma}
\begin{proof}
If $F\in \SSS{n,\delta}{S}{T}$, then there is $\sigma\subseteq [n]$ with $\abs{\sigma}<\delta \log_p(D) n$ such that for $i\in [n]\setminus \sigma$ $F(i)\in T$ and for $i\in \sigma$ $F(i)\in S$.
By enlarging $\sigma$, we can assume that $\abs{\sigma}=\ceil{\delta \log_p(D)n}-1$.

We have $\binom{n}{\ceil{\delta \log_p(D)n}-1}$ choices for $\sigma$. Once $\sigma$ is chosen, there are $\abs{T}^{n-(\ceil{\delta \log_p(D)n}-1)} \abs{S}^{\ceil{\delta \log_p(D)n}-1}$ choices for $F$. Therefore,
\[
\abs{\SSS{n,\delta}{S}{T}}
\leq\binom{n}{\ceil{\delta\log_p(D) n}-1} \abs{T}^n D^{\ceil{\delta\log_p(D)n}-1}
\leq 2^{2\sqrt{n\delta\log_p(D) n}} \abs{T}^n D^{\delta\log_p(D) n}.\qedhere
\]
\end{proof}

Given a group $G$, let $K_G$ be the number of translated subgroups of $G$.
We estimate the size of $\SSS{n,\delta}{T}{T}$.

\begin{lemma}\label{Lem: count codes}
Consider $0<\delta<\frac{1}{9\log_p(\abs{G})}$.
There is a constant $c_{G,\delta}>0$ (depending on $G$ and $\delta$), such that for each translated subgroup $T\treq G$
\[
\abs{T}^n \Big(1- K_G \exp(-c_{G,\delta}n) - K_G 2^{-n}\Big)
\leq \abs{\SSS{n,\delta}{T}{T}}
\leq \abs{T}^n.
\]
\end{lemma}
\begin{proof}
The upper bound follows from the fact that
\[
\SSS{n,\delta}{T}{T} \subseteq \Hom([n], T).
\]
Moreover,
\[
\Hom([n], T) \setminus \SSS{n,\delta}{T}{T}
=\bigcup_{T'\treq T} \SSS{n,\delta}{T}{T'} \cup \Hom([n],T').
\]
The number of $T'$ is at most $K_G$, and 
\[
\abs{\Hom([n],T')}
= \abs{T'}^n \leq \frac{\abs{T}^n}{2^n}.
\]
Denote $[T:T']=D>1$. Then we know that
\[
\abs{\SSS{n,\delta}{T}{T'}}
\leq 2^{2n\sqrt{\delta\log_p(\abs{G})}}
\abs{T}^n D^{-n+\delta\log_p(D)n}
\leq 
\abs{T}^n \Big(2^{2\sqrt{\delta\log_p(\abs{G})}-1+\delta\log_p(\abs{G})}\Big)^n.
\]
The bound on $\delta$ also ensures that $2\sqrt{\delta\log_p(\abs{G})}-1+\delta\log_p(\abs{G})<0$. The result follows.
\end{proof}

\subsection{Counting non-robust $C$}

\begin{lemma}\label{Lem: count non robust no s}
Given $0<\gamma<1$, $F\in \Hom([n],G)$ and a translated subgroup $T\treq G$, the number of $C\in \CC^*_{n}(T,F)$ that are not $(\gamma n,F)$-robust is at most
\[
\abs{G}^{\abs{G}}
\left(4\abs{G}\right)^{n\sqrt{\gamma}}.
\]
\end{lemma}
\begin{proof}
If $C\in \CC^*_{n}(T,F)$ is not $(\gamma n,F)$-robust, then there is some $\sigma\subseteq [n]$ with $\abs{\sigma}<\gamma n$ such that for $i_1,i_2\in [n]\setminus \sigma$, $F(i_1)=F(i_2)$ implies $C(i_1)=C(i_2)$. 
This means there is some $\beta:G\to \bigcup_{g\in G} \sg{-g+T}^*$, such that for each $i\in [n]\setminus \sigma$ we have $C(i)=\beta(F(i))$.
By enlarging $\sigma$, we can assume that $\abs{\sigma}=\ceil{\gamma n}-1$.

It follows that the number of choices for $\sigma$ is $\binom{n}{\ceil{\gamma n}-1}$. The number of choices for $\beta$ is at most $\abs{G}^{\abs{G}}$. Once we choose $\sigma$ and $\beta$, then $C(i)$ is determined for $i\in [n]\setminus \sigma$. For $i\in \sigma$, $C(i)$ has at most $\abs{G}$ choices. Therefore, the number of $C\in \CC^*_{n}(T,F)$ that are not $(\gamma n,F)$-robust is at most
\[
\binom{n}{\ceil{\gamma n}-1}
\abs{G}^{\abs{G}}
\abs{G}^{\ceil{\gamma n}-1}
\leq 
2^{2\sqrt{n(\gamma n)}}
\abs{G}^{\abs{G}}
\abs{G}^{\gamma n}
\leq \abs{G}^{\abs{G}}
\left(4\abs{G}\right)^{n\sqrt{\gamma}}.\qedhere
\]
\end{proof}

\begin{lemma}\label{Lem: count non robust with s}
Given $0<\gamma<1$, $F\in \Hom([n],G)$, $\vec{s}\in (\Z/p\Z)^n$ and a translated subgroup $T\treq G$, the number of $C\in \CC^*_{n}(T,F)$ that are not $(\gamma n,F,\vec{s})$-robust is at most
\[
\abs{G}^{p\abs{G}}
\left(4\abs{G}\right)^{n\sqrt{\gamma}}.
\]
\end{lemma}
\begin{proof}
If $C\in \CC^*_{n}(T,F)$ is not $(\gamma n,F,\vec{s})$-robust, then there is some $\sigma\subseteq [n]$ with $\abs{\sigma}<\gamma n$ such that for $i_1,i_2\in [n]\setminus \sigma$, $(F(i_1),s_{i_1})=(F(i_2),s_{i_2})$ implies $C(i_1)=C(i_2)$. 
This means there is some $\beta:G\times(\Z/p\Z)\to \bigcup_{g\in G} \sg{-g+T}^*$, such that for each $i\in [n]\setminus \sigma$, we have $C(i)=\beta(F(i),s_i)$.
By enlarging $\sigma$, we can assume that $\abs{\sigma}=\ceil{\gamma n}-1$.

It follows that the number of choices for $\sigma$ is $\binom{n}{\ceil{\gamma n}-1}$. The number of choices for $\beta$ is at most $\abs{G}^{p\abs{G}}$. Once we choose $\sigma$ and $\beta$, then $C(i)$ is determined for $i\in [n]\setminus \sigma$. For $i\in \sigma$, $C(i)$ has at most $\abs{G}$ choices. Therefore, the number of $C\in \CC^*_{n}(T,F)$ that are not $(\gamma n,F,\vec{s})$-robust is at most
\[
\binom{n}{\ceil{\gamma n}-1}
\abs{G}^{p\abs{G}}
\abs{G}^{\ceil{\gamma n}-1}
\leq 
2^{2\sqrt{n(\gamma n)}}
\abs{G}^{p\abs{G}}
\abs{G}^{\gamma n}
\leq \abs{G}^{p\abs{G}}
\left(4\abs{G}\right)^{n\sqrt{\gamma}}.\qedhere
\]
\end{proof}

\begin{lemma}\label{Lem: count non robust vec s}
Given $0<\gamma<1$, $F\in \Hom([n],G)$ and a translated subgroup $T\treq G$, the number of $\vec{s}\in (\Z/p\Z)^n$ that are not $(\gamma n,F, T)$-robust is at most
\[
p^{\abs{G}}
\left(4p\right)^{n\sqrt{\gamma}}
p^{\abs{\{i\in [n]: F(i)\notin T\}}}.
\]
\end{lemma}
\begin{proof}
If $\vec{s}\in (\Z/p\Z)^n$ is not $(\gamma n,F,T)$-robust, then there is some $\sigma\subseteq [n]$ with $\abs{\sigma}<\gamma n$ such that for $i_1,i_2\in [n]\setminus \sigma$, $F(i_1)=F(i_2)\in T$ implies $s_{i_1}=s_{i_2}$. 
This means there is some $\lambda:T\to (\Z/p\Z)$, such that for each $i\in [n]\setminus \sigma$, if $F(i)\in T$, then $s_i=\lambda(F(i))$.
By enlarging $\sigma$, we can assume that $\abs{\sigma}=\ceil{\gamma n}-1$.

The number of choices for $\sigma$ is $\binom{n}{\ceil{\gamma n}-1}$. The number of choices for $\lambda$ is at most $p^{\abs{G}}$. Once we choose $\sigma$ and $\lambda$, then $s_i$ is determined for $i\in [n]\setminus \sigma$ with $F(i)\in T$. For other $i$, $s_i$ has at most $p$ choices. Therefore, the number of $\vec{s}\in (\Z/p\Z)^n$ that are not $(\gamma n,F,T)$-robust is at most
\[
\binom{n}{\ceil{\gamma n}-1}
p^{\abs{G}}
p^{\abs{\{i\in [n]: F(i)\notin T\}}}
p^{\ceil{\gamma n}-1}
\leq 
2^{2\sqrt{n(\gamma n)}}
p^{\abs{G}}
p^{\gamma n}
p^{\abs{\{i\in [n]: F(i)\notin T\}}}
\leq p^{\abs{G}}
\left(4p\right)^{n\sqrt{\gamma}}
p^{\abs{\{i\in [n]: F(i)\notin T\}}}.\qedhere
\]
\end{proof}

\subsection{Size of $\EE_{n,\rho}$}

\begin{lemma}\label{Lem: size of E rho}
Given $\rho>0$, we have
\[
\lim_{n\to\infty} \frac{\abs{\EE_{n,\rho}}}{p^n}
=1.
\]
\end{lemma}
\begin{proof}
By Hoeffding's inequality, we see that
\[
\abs{1- \frac{\abs{\EE_{n,\rho}}}{p^n}}
\leq
\abs{1
-\sum_{k:\abs{k-\frac{1}{p}n}\leq \rho n}
\binom{n}{k} \frac{1}{p^k} \Big(1-\frac{1}{p}\Big)^{n-k}}
\leq 2\exp\big(-2\rho^2 n\big).
\]
The result follows.
\end{proof}

\begin{lemma}\label{Lem: Prob E rho to 1}
Given $\rho>0$, we have
\[
\lim_{n\to\infty} \PPP[ \eta(\Lna)\in \EE_{n,\rho} ]
=1.
\]    
\end{lemma}
\begin{proof}
This follows from \cite[Lemma 8.2]{Singhal_SandpileBipartite_Directed}.
\end{proof}

\section*{Acknowledgments}
We thank Nathan Kaplan for introducing me to this problem and for many helpful discussions about the problem.
We acknowledge support from NSF Grant DMS 2154223.

\bibliographystyle{plain}
\bibliography{Bibliography}

\end{document}